\documentclass[psamsfonts]{amsart}
\usepackage[utf8]{inputenc}
\usepackage{amsfonts}
\usepackage{hyperref}
\hypersetup{
pdftitle={Endoscopy for affine Hecke categories},
pdfsubject={Mathematics},
pdfauthor={Yau Wing Li},
pdfkeywords={Hecke algebra, arXiv, Preprint}
}
\usepackage{amsmath}
\usepackage{xcolor}
\usepackage{amsthm}
\usepackage{pdflscape}
\usepackage{mathrsfs}
\usepackage{mathtools}

\usepackage{tikz}
\usepackage{amssymb}

\usepackage[all]{xy}

\usepackage{color}
\usepackage[shortlabels]{enumitem}
\usepackage{hyperref}
\usepackage{pgf,tikz}
\usepackage{mathrsfs}
\usetikzlibrary{arrows}
\usepackage{tikz-cd}
\usepackage{graphicx}
\usepackage{filecontents}

\newtheorem{thm}{Theorem}[section]
\newtheorem{cor}[thm]{Corollary}
\newtheorem{prop}[thm]{Proposition}
\newtheorem{lem}[thm]{Lemma}

\newtheorem{claim}[thm]{Claim}

\theoremstyle{definition}
\newtheorem{defn}[thm]{Definition}

\newtheorem{con}[thm]{Construction}
\newtheorem{exmp}[thm]{Example}

\theoremstyle{remark}
\newtheorem{rem}[thm]{Remark}

\makeatletter
\let\c@equation\c@thm
\makeatother
\numberwithin{equation}{section}

\newcommand{\F}{\mathbb{F}} % fancy field name
\newcommand{\Z}{\mathbb{Z}}   %integers
\newcommand{\p}{\mathbb{P}}   %projective plane
\newcommand{\Q}{\mathbb{Q}}
\newcommand{\N}{\mathbb{N}}
\newcommand{\D}{\mathbb{D}}

\newcommand{\So}{\mathbb{S}}

\newcommand{\sE}{\mathcal{E}}
\newcommand{\sF}{\mathcal{F}}
\newcommand{\sG}{\mathcal{G}}
\newcommand{\sH}{\mathcal{H}}
\newcommand{\sL}{\mathcal{L}}

\newcommand{\sK}{\mathcal{K}}

\newcommand{\sS}{\mathcal{S}}
\newcommand{\sV}{\mathcal{V}}
\newcommand{\sW}{\mathcal{W}}

\newcommand{\lig}{\mathfrak{g}}

\newcommand{\co}{\mathcal{O}}

\newcommand{\M}{\mathbb{M}}  
\newcommand{\ql}{\overline{\Q_\ell}}
\newcommand{\A}{\mathbb{A}}
\newcommand{\X}{\mathbb{X}}

\newcommand{\lit}{\mathfrak{t}}\newcommand{\lio}{\mathfrak{o}}

\newcommand{\fq}{\mathbb{F}_q}
\newcommand{\pha}{\phantom{ }}
\newcommand{\bim}{\bold{I^-}}
\newcommand{\bium}{\bold{I^-_u}}
\newcommand{\biu}{\bold{I_u}}
\newcommand{\bius}{\textbf{I}^s_\textbf{u}}
\newcommand{\bi}{\bold{I}}
\newcommand{\bsl}{\backslash}
\newcommand{\hk}{\bold{H_K}}
\newcommand{\gk}{\bold{G_K}}
\newcommand{\glnk}{\bold{GL_{n,K}}}

\newcommand{\go}{\bold{G_\co}}
\newcommand{\gkw}{\bold{G_K}_{,w}}

\newcommand{\gkv}{\bold{G_K}_{,v}}

\newcommand{\gksw}{\bold{G_K}_{,\le w}}

\newcommand{\gksi}{\bold{G_K}_{,s_i}}
\newcommand{\gkwp}{\bold{G_K}_{,w'}}
\newcommand{\gkswp}{\bold{G_K}_{,\le w'}}
\newcommand{\gkswpp}{\bold{G_K}_{,\le w''}}

\newcommand{\eglnk}{\widetilde{\bold{GL_{n,K}}}}

\newcommand{\wdo}{\dot{w}}
\newcommand{\sdo}{\dot{s}}

\newcommand{\extw}{\widetilde{W}}

\newcommand{\gm}{\mathbb{G}_m}

\newcommand{\extsL}{\widetilde{\sL}}
\newcommand{\ldswl}{_{\sL} D( \le w)_\sL}

\newcommand{\conp}{\star_+}
\newcommand{\conn}{\star_-}
\newcommand{\sC}{\mathcal{C}}
\newcommand{\ulextw}{\underline{\extw}}
\newcommand{\ulextd}{\underline{D}}
\newcommand{\lpdl}{_{\sL'} D_ \sL}

\newcommand{\lpdull}{_{\sL'} \underline{D}_ \sL}
\newcommand{\lppdl}{_{\sL''} D_ \sL}

\newcommand{\ldl}{_{\sL} D_ \sL}

\newcommand{\ldlm}{_{\sL} D_ \sL^-}

\newcommand{\ldlo}{_{\sL} D_ \sL ^\circ}
\newcommand{\ludlo}{_{\sL} \underline{D}_ \sL ^\circ}
\newcommand{\lpudl}{_{\sL'} \underline{D}_ \sL}
\newcommand{\ludl}{_{\sL} \underline{D}_ \sL}

\newcommand{\lpewul}{_{\sL'} \underline{\extw}_ \sL}
\newcommand{\wlewl}{_{w\sL} \extw_ \sL}
\newcommand{\lppewulp}{_{\sL''} \underline{\extw}_ {\sL'}}
\newcommand{\ewlo}{\extw_\sL^\circ}
\newcommand{\ewlpo}{\extw_{\sL'}^\circ}
\newcommand{\ewho}{\extw_H^\circ}
\newcommand{\ewwlo}{\extw_{w\sL}^\circ}

\newcommand{\lpdlm}{_{\sL'} D^-_ \sL}
\newcommand{\lppdlm}{_{\sL''} D^-_ \sL}
\newcommand{\lpdelm}{_{\sL'} D^-_{ \widetilde{\sL}}}
\newcommand{\lpmdlmm}{_{\sL'^{-1}} D^-_{\sL^{-1}}}
\newcommand{\lpmdlm}{_{\sL'^{-1}} D_{\sL^{-1}}}
\newcommand{\lpppdlppm}{_{\sL'''} D_{\sL''}^-}
\newcommand{\lppdlp}{_{\sL''} D_{\sL'}}
\newcommand{\lppdlpm}{_{\sL''} D^-_{\sL'}}

\newcommand{\lpdsl}{_{\sL'} D _{s\sL}}
\newcommand{\lpdwl}{_{\sL'} D _{w\sL}}
\newcommand{\wldvl}{_{w\sL} D (v)_{\sL}}

\newcommand{\lpdwlm}{_{\sL'} D^- _{w\sL}}
\newcommand{\lpdslm}{_{\sL'} D ^-_{s\sL}}
\newcommand{\lpdul}{_{\sL'}{\ulad} \pha _\sL}
\newcommand{\lpdulk}{_{\sL'}{\ulad} \pha _{\sL,k}}
\newcommand{\lpduel}{_{\sL'}{\ulad} \pha _{\widetilde{\sL}}}
\newcommand{\lppdul}{_{\sL''}{\ulad} \pha _\sL}
\newcommand{\lppdulp}{_{\sL''}{\ulad} \pha _{\sL'}}

\newcommand{\ltlm}{_{\sL} \Theta^- _{\sL}}
\newcommand{\lptlm}{_{\sL'} \Theta^- _{\sL}}
\newcommand{\lptlpm}{_{\sL'} \Theta^- _{\sL'}}
\newcommand{\lpptlppm}{_{\sL''} \Theta^- _{\sL''}}

\newcommand{\ltulm}{_{\sL} \underline{\Theta}^- _{\sL}}
\newcommand{\lptulm}{_{\sL'} \underline{\Theta}^- _{\sL}}
\newcommand{\lptulmpp}{_{\sL'} \underline{\Theta}^- _{\sL''}}
\newcommand{\avltlm}{\Av_! \pha \ltlm}
\newcommand{\avlptlm}{\Av_! \pha \lptlm}
\newcommand{\avlptlpm}{\Av_! \pha \lptlpm}
\newcommand{\avltulm}{\Av_! \pha \ltulm}

\newcommand{\gmr}{\gm^{\rot}}
\newcommand{\gmc}{\gm^{\cen}}
\newcommand{\barr}{\overline{R}}
\newcommand{\bSo}{\overline{\mathbb{S}}}

\newcommand{\wu}{\underline{w}}

\newcommand{\bbb}{\Bun_{0', \infty'}}
\newcommand{\bub}{\Bun_{0, \infty'}}
\newcommand{\bbu}{\Bun_{0', \infty}}
\newcommand{\buu}{\Bun_{0, \infty}}
\newcommand{\bcu}{\Bun_{\hat{0}, \infty}}
\newcommand{\bsu}{\Bun_{0_s, \infty}}
\newcommand{\bsuw}{\Bun_{0_s, \infty}^w}

\newcommand{\bujw}{\Bun_{0, J_\infty}^w}
\newcommand{\bujsw}{\Bun_{0, J_\infty}^{\le w}}
\newcommand{\bujn}{\Bun_{0, J^n}}
\newcommand{\bujnw}{\Bun_{0, J^n}^w}
\newcommand{\bujnsw}{\Bun_{0, J^n}^{\le w}}
\newcommand{\bbjsw}{\Bun_{0', J_\infty}^{\le w}}
\newcommand{\bbjw}{\Bun_{0', J_\infty}^{w}}
\newcommand{\bbjv}{\Bun_{0', J_\infty}^{v}}

\newcommand{\bbbe}{\Bun_{0', \infty'}^{e}}
\newcommand{\bbbw}{\Bun_{0', \infty'}^{w}}
\newcommand{\bbbwp}{\Bun_{0', \infty'}^{w'}}
\newcommand{\bbbsw}{\Bun_{0', \infty'}^{\le w}}

\newcommand{\bbbgw}{\Bun_{0', \infty'}^{\ge w}}

\newcommand{\bubw}{\Bun_{0, \infty'}^{w}}

\newcommand{\bubgw}{\Bun_{0, \infty'}^{\ge w}}
\newcommand{\buuw}{\Bun_{0, \infty}^{w}}

\newcommand{\buuwpp}{\Bun_{0, \infty}^{w''}}
\newcommand{\buusw}{\Bun_{0, \infty}^{\le w}}
\newcommand{\buuswp}{\Bun_{0, \infty}^{\le w'}}
\newcommand{\buuswpp}{\Bun_{0, \infty}^{\le w''}}
\newcommand{\buuswppp}{\Bun_{0, \infty}^{\le w'''}}
\newcommand{\buuwppsin}{\Bun_{0, \infty}^{w''s_{i_n}}}
\newcommand{\buuwppy}{\Bun_{0, \infty}^{w''y}}

\newcommand{\hec}{\mathcal{H}_{0,\infty}}
\newcommand{\hecb}{\mathcal{H}_{0',\infty}}
\newcommand{\hecwa}{\mathcal{H}_{0,\infty}^w}
\newcommand{\hecswa}{\mathcal{H}_{0,\infty}^{\le w}}
\newcommand{\hecw}{\mathcal{H}_{0,\infty ,w}}
\newcommand{\hecsw}{\mathcal{H}_{0,\infty, \le w}}

\newcommand{\hecswpa}{\mathcal{H}_{0,\infty}^{\le w'}}
\newcommand{\hecbwf}{\mathcal{H}_{0',\infty, w_4}}
\newcommand{\hecbs}{\mathcal{H}_{0',\infty, s}}
\newcommand{\hecbswpp}{\mathcal{H}_{0',\infty, \le w''}}

\newcommand{\hecswswpa}{\mathcal{H}_{0,\infty, \le w}^{\le w'}}

\newcommand{\hecswpswppa}{\mathcal{H}_{0,\infty, \le w'}^{\le w''}}
\newcommand{\hecbswswpa}{\mathcal{H}_{0',\infty, \le w}^{\le w'}}

\newcommand{\hecbswpswppa}{\mathcal{H}_{0',\infty, \le w'}^{\le w''}}
\newcommand{\hecwwpa}{\mathcal{H}_{0,\infty, w}^{w'}}
\newcommand{\hecswwpa}{\mathcal{H}_{0,\infty, \le w}^{w'}}
\newcommand{\hecwswpa}{\mathcal{H}_{0,\infty, w}^{\le w'}}

\newcommand{\jinf}{J_\infty}

\newcommand{\ra}{\rightarrow}
\newcommand{\la}{\leftarrow}
\newcommand{\sur}{\twoheadrightarrow}
\newcommand{\inj}{\hookrightarrow}

\newcommand{\rasim}{\xrightarrow{\sim}}
\newcommand{\xla}{\xleftarrow}
\newcommand{\xra}{\xrightarrow}

\DeclareMathOperator{\Hom}{Hom}
\DeclareMathOperator{\End}{End}
\DeclareMathOperator{\Ext}{Ext}
\DeclareMathOperator{\Ad}{Ad}

\DeclareMathOperator{\Spec}{Spec}
\DeclareMathOperator{\chara}{char}
\DeclareMathOperator{\Sym}{Sym}
\DeclareMathOperator{\Supp}{Supp}
\DeclareMathOperator{\Ch}{Ch}

\DeclareMathOperator{\Av}{Av}
\DeclareMathOperator{\Ind}{Ind}
\DeclareMathOperator{\Res}{Res}

\DeclareMathOperator{\Pic}{Pic}
\DeclareMathOperator*{\ulalim}{\underleftarrow\lim}
\DeclareMathOperator*{\uralim}{\underrightarrow\lim}
\DeclareMathOperator*{\ulad}{\underleftarrow D}

\DeclareMathOperator{\pt}{pt}
\DeclareMathOperator{\Fr}{Fr}
\DeclareMathOperator{\ext}{ext}
\DeclareMathOperator{\pure}{pure}
\DeclareMathOperator{\HOM}{HOM}
\DeclareMathOperator{\RHom}{\textbf{R} \kern -2pt \Hom}
\DeclareMathOperator{\Rhom}{\textbf{R}\kern -2pt \hom}
\DeclareMathOperator{\RuHom}{\textbf{R} \underline{\kern -0pt \Hom}}
\DeclareMathOperator{\IC}{IC}
\DeclareMathOperator{\ICu}{\underline{IC}}

\DeclareMathOperator{\rot}{rot}
\DeclareMathOperator{\cen}{cen}
\DeclareMathOperator{\alge}{alg}

\DeclareMathOperator{\Delu}{\underline{\Delta}}
\DeclareMathOperator{\nabu}{\underline{\nabla}}
\DeclareMathOperator{\Gr}{Gr}
\DeclareMathOperator{\Lemma}{Lemma}
\DeclareMathOperator{\gmod}{gmod}

\DeclareMathOperator{\tH}{H}
\DeclareMathOperator{\SB}{SB}
\DeclareMathOperator{\bSB}{\overline{SB}}

\DeclareMathOperator{\red}{red}
\DeclareMathOperator{\xiu}{\underline{\xi}}
\DeclareMathOperator{\Mu}{\underline{\M}}
\DeclareMathOperator{\sFu}{\underline{\sF}}
\DeclareMathOperator{\Bun}{Bun}

\newcommand{\esF}{\widetilde{\mathcal{F}}}
\newcommand{\esG}{\widetilde{\mathcal{G}}}

\newcommand{\eR}{\widetilde{R}}
\newcommand{\eM}{\mathbb{\widetilde{M}}}
\newcommand{\eT}{\widetilde{T}}
\newcommand{\elit}{\widetilde{\lit}}

\newcommand{\ePsi}{\widetilde{\Psi}}
\newcommand{\eD}{\widetilde{D}}
\DeclareMathOperator{\exi}{\widetilde{\xi}}

\DeclareMathOperator{\eIC}{\widetilde{IC}}
\DeclareMathOperator{\eICu}{\widetilde{\ICu}}
\DeclareMathOperator{\eTh}{\widetilde{\Theta}}

\DeclareMathOperator{\eDe}{\widetilde{\Delta}}
\DeclareMathOperator{\ena}{\widetilde{\nabla}}
\newcommand{\ebi}{\widetilde{\bold{I}}}
\newcommand{\egk}{\widetilde{\bold{G_K}}}
\newcommand{\ego}{\widetilde{\bold{G_\co}}}
\newcommand{\egksw}{\widetilde{\bold{G_K}_{,\le w}}}
\newcommand{\ebim}{\widetilde{\bold{I^-}}}

\newcommand{\ledl}{_{\sL} \widetilde{D}_ \sL}
\newcommand{\lpedl}{_{\sL'} \widetilde{D}_ \sL}

\newcommand{\lptulpm}{ _{\sL'}\underline{\Theta}\pha^- _{\sL'}}
\newcommand{\ldet}{\sL_{det}}

\newcommand{\usF}{\underline{\sF}}
\newcommand{\usG}{\underline{\sG}}
\newcommand{\eMu}{\underline{\widetilde{\mathbb{M}}}}

\newcommand{\ldswlo}{_{\sL} D( \le w)^{\circ}_\sL}
\newcommand{\lcswlo}{_{\sL} C( \le w)^{\circ}_\sL}

\newcommand{\lclo}{_{\sL} C^{\circ}_\sL}
\newcommand{\lculo}{ _{\sL}\underline{C}\pha^{\circ}_\sL}

\newcommand{\lcuswlo}{_{\sL} \underline{C}( \le w)^{\circ}_\sL}

\newcommand{\ldulo}{ _{\sL}\underline{D}^{\circ}_\sL}

\newcommand{\lduswlo}{_{\sL} \underline{D}( \le w)^{\circ}_\sL}

\title{Endoscopy for affine Hecke categories}

\author{Yau Wing Li}

\begin{document}

\begin{abstract}
We show that the neutral block of the affine monodromic Hecke category for a reductive group is monoidally equivalent to the neutral block of the affine Hecke category for its endoscopic group. The semisimple complexes of both categories can be identified with the generalized Soergel bimodules via the Soergel functor. We extend this identification of semisimple complexes to the neutral blocks of the affine Hecke categories by the technical machinery developed by Bezrukavnikov and Yun.

\end{abstract}
\address{University of Melbourne, School of Mathematics and Statistics}
\email{albert.li2@unimelb.edu.au}
\date{\today}
\maketitle
\tableofcontents

\section{Introduction}

\subsection{Hecke category}
Let $G$ be a connected split reductive group over $\fq$, $T$ be a maximal split torus inside a Borel subgroup $B$ and $W$ be the Weyl group of $G$. The Hecke algebra $\sH_q(W)$ is the ring of $\ql$-valued functions on the double cosets $B(\fq) \bsl G(\fq) /B(\fq)$, that is the bi-$B(\fq)$-invariant functions on $G(\fq)$. The convolution product gives the multiplication structure. The Hecke algebra is ubiquitous in representation theory. For instance, it is related to the representations of finite groups of Lie type \cite{DL,L2} and the characters of the Verma modules in the BGG category $\co$ via Kazhdan-Lusztig canonical basis \cite{BB,BK,Hu,KL}. 

The affine Hecke algebra (or the Iwahori-Hecke algebra) is a generalization of the Hecke algebra, where the loop group $\gk$ and the Iwahori subgroup $\bi$ take the roles of $G$ and $B$ respectively. This can be thought of as a quantization of the group algebra of the extended affine Weyl group $\extw$. The affine Hecke algebra is a central object in the local geometric Langlands program. It governs the irreducible admissible representations of $G(\fq((t)))$ with Iwahori-invariant vectors. 

The affine Hecke category categorifies the affine Hecke algebra. We consider $\bi$-equivariant constructible sheaves instead of $\bi$-equivariant functions. Under the sheaf-function dictionary of Grothendieck, we can associate a sheaf with a function by taking the trace of Frobenius. To adapt the six functors formalism, it is better to consider the derived category of equivariant sheaves $D^b_m(\bi \bsl \gk / \bi)$, which is a monoidal category. The affine Hecke category and its relatives play extremely important roles in the geometric representation theory. They are related to the representations of the quantum group at a root of unity and coherent sheaves on the Springer resolution  \cite{ABG}.

%When $V$ is a representation of $G(\fq)$, the convolution product gives an $\sH_q(W)$-action on the $B(\fq)$-invariant subspce of $V$, which is much smaller than $V$ in general. In this way, the Hecke algebra comes into the study of the unipotent principal series representation of finite groups of Lie type \cite{DL}. 

%The Hecke algebra also plays a role in knot theory. A fundamental idea is the standard generators of $\sH_q(W)$ labelled by the simple reflections satisfy the braid relations. Hence there is a surjective morphism from the group algebra of the braid groupto the Hecke algebra. 

\subsection{Results}
%Let $\gk$ be the loop group of $G$ and $\bi$ be the Iwahori subgroup. The affine Hecke category $D^b_m(\bi \bsl \gk / \bi)$ is defined as the $\bi$-equivariant derived category of complexes of sheaves with $\ql$-coefficients on the affine flag variety $\gk / \bi$ whose cohomology sheaves are mixed in the sense of Deligne \cite{D}. The convolution product provides $D^b_m(\bi \bsl \gk / \bi)$ a monoidal structure. Under the sheaf-function dictionary of Grothendieck, the corresponding algebra contains the affine Hecke algebra $\sH_q(\extw)$, where $\extw$ is the extended affine Weyl group of $G$. In this sense, $D^b_m(\bi \bsl \gk / \bi)$ categorifies $\sH_q(\extw)$.

The affine Hecke category has monodromic counterparts. Let $\biu$ be the pro-unipotent radical of $\bi$. Suppose $\sL'$ and $\sL$ are two rank one character sheaves over $T$. We can consider the $\biu$-equivariant derived category of mixed complexes of sheaves on $\gk/ \biu$ whose monodromy under left (resp. right) $T$-action is $\sL'$ (resp. $\sL^{-1}$). Denote the category as $\lpdl$. For any two sheaves in $\lppdlp$ and $\lpdl$, one can define their convolution product, which is a sheaf in $\lppdl$. In particular, $\ldl$ is a monoidal category. $\ldl$ can be decomposed into several smaller (not necessarily monoidal) subcategories. Denote the identity block of the decomposition as $\ldlo$. 

\begin{thm}
Let $G$ be a connected split reductive group over $\F_q$, and $T$ be a maximal split torus inside a Borel subgroup $B$. Assume that $\sL$ is a character sheaf over $T$ and $H$ is the endoscopic group for $\sL$. Let $\hk$ and $\bi_H$ be the loop group and the Iwahori subgroup of $H$ respectively.

Then there is a canonical monoidal equivalence of triangulated categories, 
\[\Psi_{\sL}^\circ : D^b_m(\bi_H \bsl \hk / \bi_H)^\circ \rasim \pha _\sL D_{\sL} ^\circ\]
preserving the standard objects, the IC sheaves and the costandard objects, i.e. $\Psi_{\sL}^\circ$ sends $\Delta(w)_H, \IC(w)_H, \nabla(w)_H$ to $\Delta(w)_{\sL}, \IC(w)_{\sL}, \nabla(w)_{\sL}$ respectively for all $w \in \widetilde{W_\sL^\circ}=\widetilde{W_H^\circ}$.  $D^b_m(\bi_H \bsl \hk / \bi_H)^\circ$ is the identity block of $D^b_m(\bi_H \bsl \hk / \bi_H)$. In other words, it contains the sheaves supported on the identity connected component of $\bi_H \bsl \hk / \bi_H$. 
\end{thm}

In fact $\Psi_{\sL}^\circ$ satisfies more properties which we cover in Theorem \ref{mainthm}. We also prove the analogous statement for the central extension of the loop group in Theorem \ref{mainthm2}. In this case, $\sL$ is a character sheaf over $T$, but not over its central extension. The similar statement for the reductive group case was proved by Lusztig and Yun in \cite{LY}.

\subsection{Potential application to local geometric Langlands program}
In \cite{Be}, Bezrukavnikov proved the case of the unipotently ramified local geometric Langlands correspondence. More precisely, he proved the equivalence between the affine Hecke category associated with a reductive group $G$ over $\overline{\fq}$ and the DG category of equivariant coherent sheaves on the derived Steinberg variety of the Langlands dual group $G^L$. He then conjectured the tamely ramified local geometric Langlands correspondence, which is a similar statement for the affine monodromic Hecke category and the DG category of equivariant coherent sheaves on the twisted derived Steinberg variety. See Conjecture 58 in \textit{loc.cit.}. 

When $G$ is simple and of adjoint type, it can be shown that the twisted derived Steinberg variety for $G^L$ is isomorphic to the derived Steinberg variety for $H^L$. Therefore the conjecture is equivalent to the equivalence of non-neutral blocks via Bezrukavnikov's equivalence. This will be discussed in a forthcoming joint work with Gurbir Dhillon, Zhiwei Yun, and Xinwen Zhu.

%The associativity constraint of non-neutral blocks gives a 3-cocycle of $\Omega$ valued in the group $\{ \pm 1\}$, where $\Omega$ is the group of the minimal representatives with respect to the block decomposition. When $G$ is simple and of adjoint type, it can be shown that the twisted derived Steinberg variety for $G^L$ is isomorphic to the derived Steinberg variety for $H^L$. Therefore the conjecture is equivalent to the equivalence of non-neutral blocks via Bezrukavnikov's equivalence. Moreover, proving that such equivalence exists is equivalent to proving that the 3-cocycle is cohomologically trivial. This will be discussed in a forthcoming paper.

\subsection{Method of proof}
To explain the proof, we recall how the reductive group case is proved in \cite{LY}. Their proof makes extensive use of the monodromic Soergel functor $\M$ which maps a semisimple complex $\sF$ to the $H_{T}^*(pt)$-bimodule of homomorphisms from a maximal IC sheaf to $\sF$. As a special case when $\sL$ is trivial, $\M$ is the same as the global sections functor. Via this functor, the IC sheaves can be identified with generalized Soergel bimodules. The general DG model construction in \cite{BY} is then used to extend this equivalence to the neutral blocks of derived categories.
 
The main difficulty in extending the above argument in \cite{LY} to the loop group case is that there are no IC sheaves with maximal support, hence the above definition of $\M$ does not apply. The main contribution of this paper is to consider monodromic sheaves on $\buu$, where $\buu$ is the moduli stack classifying $G$-bundles on $\p^1$ along with a Borel reduction at zero and infinity. There is a unique maximal orbit in each block of $_{\sL'} \extw _\sL$. We can define the maximal IC sheaves as the IC sheaves which are supported on these orbits. This allows us to define $\M$ for the loop group case. When $\sL$ is trivial, the maximal IC sheaf is the constant sheaf. In this special case, H\"arterich proved in \cite{Ha} that the subcategory of semisimple complexes can be realized as the Soergel bimodules of the affine Weyl group via taking equivariant cohomology. 

%Firstly, the definition of maximal IC sheaf has to be modified. In the reductive group case, every block has a maximal element because of finiteness of the Weyl group. Maximal IC sheaf can be simply defined as the IC sheaf with the maximal support. However, such object does not exist in the derived category of $\biu$-equivariant $\ql$-sheaves on affine flag variety. One way to get around this is to consider $\bium$-equivariant $\ql$-sheaves, where $\bium$ is the pro-unipotent radical of opposite Iwahori subgroup. 

To prove that $\M$ is a monoidal functor, we construct a coalgebra structure on each maximal IC sheaf $\Theta$. Note that we cannot define the two-fold convolution product of $\Theta$ directly. Instead, we consider the convolution product $\Theta \star \Av_!\Theta$, where $\Av_!$ is the functor averaging sheaves along $\biu$-orbits. The key in constructing the coalgebra map from $\Theta$ to $\Theta \star \Av_!\Theta$ boils down to computing stalks of $\Av_!\Theta$. These are computed in Proposition \ref{avstalk} by applying Braden's hyperbolic localization in \cite{Br}.

%This appearance of $\bium$-equivariance introduces another difficulty. The maximal IC sheaf $\Theta$ is now left $\bium$-equivariant and right $\biu$-equivariant, hence the two-fold convolution product of $\Theta$ cannot be defined directly. Our substitution is the convolution product $\Theta \star \Av_!\Theta$, where $\Av_!$ is the averaging functor which takes left $\bium$-equivariant sheaves to pro-objects in the category of left $\biu$-equivariant sheaves. To obtain the coalgebra map from $\Theta$ to $\Theta \star \Av_!\Theta$, information about the stalks of $\Av_!\Theta$ is needed. It is done in Proposition \ref{avstalk} by applying Braden's hyperbolic localization in \cite{Br}.

%Here is the table of content. Section \ref{sect2} provides basic notations. The affine monodromic Hecke categories are defined in Section \ref{sect3}, and their convolution products are discussed in Section \ref{sect4}. Some results in \cite{LY} are recapped and generalized in Section \ref{sect5}. Most proofs are omitted because of similarity to the reductive group case. 

%Section \ref{sect6} is the most crucial part of this article. Maximal IC sheaf $\Theta$ and averaging functor are defined. Moreover, $\Theta$ is given a coalgebra structure. 

%The results in Section \ref{sect6} implies the affine monodromic Soergel functor $\M$ is a monoidal functor. IC sheaves are identified with generalized Soergel bimodules in section \ref{sect7} and \ref{sect8}. Section \ref{sect9} concludes the proof. 

\subsection{Acknowledgements}
The author would like to thank his PhD advisor Zhiwei Yun for introducing the subject and problem. Gratitude is expressed to Gurbir Dhillon, Davis Dougal, Mark Goresky, and Xinwen Zhu for stimulating mathematical discussions. Additionally, the author would like to thank the anonymous referee for pointing out some mistakes. This material is based upon work supported by the University of Melbourne, IAS School of Mathematics, the National Science Foundation under Grant No. DMS-1926686, and the ARC grant FL200100141.

\section{Notations and conventions}\label{sect2}

We follow most of the notations and conventions in \cite{LY}.

\subsection{Frobenius}\label{sectFrob}
Throughout the article, $k= \overline{\F_q}$ is a fixed algebraic closure of $\F_q$ and $\ell$ is a prime different from $p=\chara(k)$. Denote pt to be $\Spec(\F_q)$.

For any Artin stack $X$ over $\fq$, we denote by $D^b_m(X)$ the derived category of \'etale $\ql$-complexes on $X$ whose cohomology sheaves are mixed. When $J$ is an algebraic group over $\fq$ acting on a scheme $Y$ over $\fq$, $D^b_{m}([J \backslash Y])$ is the $J$-equivariant derived category of \'etale $\ql$-complexes on $Y$ whose cohomology sheaves are mixed. In this case, we sometimes write $D^b_{m,J}(Y)$ instead.

For any Artin stack $X$ over $\fq$, denote $\Fr_X$ to be the geometric Frobenius of $X$. The subscript will be suppressed whenever the context is clear. 

For any Artin stack $X$ over $\fq$, the projection map $X_k:=X \times_{\pt}\Spec(k) \ra X$ induces a pullback functor $ \omega: D^b_m(X) \ra D^b_c(X_k)$. Let $\sF, \sG \in D^b_m(X)$. Define the following two types of homomorphisms between $\sF$ and $\sG$,
\[\ext^i(\sF, \sG):=\Hom_{D^b_m(X)}(\sF, \sG[i]),\] 
\[\Ext^i(\sF, \sG):=\Hom_{D^b_c(X _k)}(\omega \sF, \omega \sG[i]).\]

Note that the latter module carries a Frobenius action. We distinguish the corresponding derived functors by $\Rhom$ and $\RHom$ respectively. Unless otherwise specified, $\RHom(\sF,\sG)$ always means $\RHom(\omega \sF, \omega \sG)$.

We note the following well-known exact sequence, which plays a significant role in establishing a connection between the mixed category and the non-mixed category (see \cite[Section 5.1.2.5]{BBD} for details). 

\begin{lem}\label{mixnonmix}
	Let $A$ and $B$ be two mixed complexes. We have the following exact sequence:
	\[0 \ra \Hom^{-1}(\omega A, \omega B)_{\Fr} \ra \ext^0(A,B) \ra \Ext^0(\omega A, \omega B)^{\Fr} \ra 0.\]
\end{lem}	

\begin{cor}
	Let $A$ and $B$ be two mixed perverse sheaves. The $\ql$-vector space $\ext^0(A,B)$ is a subspace of $\Ext^0(\omega A, \omega B)$.
\end{cor}

Let $M$ be a graded module over a graded algebra $A$. Denote $M[1]$ to be the graded module with $M[1]_i=M_{i+1}$.

Suppose Frobenius acts on $A$ and $M$ is a graded $A$-module with a compatible Frobenius action $\Fr_M$, that is $\Fr_A(a)\Fr_M(m)=\Fr_M(am)$. Define $M(1)$ to be the same underlying $A$-module with the twisted Frobenius action $\Fr_{M(1)}(m)=q\Fr_{M}(m)$.

Fix a square root of $q$ in $\ql$. We define $M \langle 1 \rangle = M[1](1/2)$. This convention applies to both mixed sheaves and graded modules with Frobenius action.

%From the definition of $D^b_{(T \times T, \sL \boxtimes \sL'),m}(\biu \bsl \gk / \biu) ^\circ$, it is clear that 

%Let $Q$ be the root lattice of $G$, and 

%Let $D'$ be the subcategory of $D^b_{(T \times T, \sL \boxtimes \sL'),m}(\biu \bsl \gk / \biu)$, consists of s whose support are on the orbits whose indexes are $W \ltimes Q$. It is clear that it is a sub-monoidal category. 

%Claim: $D'$ is isomorphic to 

%Without loss of generality, we may and will assume $G$ is simply connected. This assumption implies $\extw$ is a Coxeter group. 
%It is indexed by $D^b_{(T \times T, \sL \boxtimes \sL'),m}(\biu \bsl \gk / \biu) $.

\subsection{Rank one character sheaves}
Recall the definition of rank one character sheaves.

\begin{defn}Let $J$ be an algebraic group over $\fq$. A rank one character sheaf over $J$ is a triple $(\sL, \iota, \phi)$ where $\sL$ is a rank one local system on $J$,   $\iota: \sL_e \xrightarrow{\sim}\ql$ is the rigidification of stalk at the identity element $e$ of $J$ and $\phi :m^*\sL\xrightarrow{\sim} \sL \boxtimes \sL$ is an identification between the pullback of $\sL$ along $m$, the multiplication map of $J$, and the exterior product of $\sL$, satisfying the associativity and unital axioms (see \cite[Appendix A]{Y2} for the details). Denote $\Ch (J)$ to be the isomorphism class of rank one character sheaves over $J$. 
\end{defn}

By abuse of notation, we write $\sL$ instead of $(\sL, \iota, \phi)$. Details about rank one character sheaves can be found in \cite[Section 2]{LY} and \cite{Y2}. We mention two important properties of character sheaves here. Firstly, the automorphism group of $\sL$ is trivial when $J$ is connected. Secondly, when $J$ is a split torus over $\fq$, $\Ch(J)$ can be identified with $\Hom(J(\fq), \ql)=J^L(\ql)$, where $J^L$ is the Langlands dual of $J$ over $\ql$.

%There is a criterion to determine whether a character sheaf on a subgroup can be extended to the whole group.

%\begin{lemma}
%For a reductive algebraic group $J$ over $\fq$ with maximal torus $T'$ and $\sL \in \Ch(T')$. Then $\sL$ can be extended to $\widetilde{\mathcal{L}}\in \Ch(J)$ if and only if for every coroot $\alpha^\vee$ of $J$, the pullback $(\alpha^{\vee})^*\sL$ is trivial. 
%\end{lemma}

\subsection{Group theory}
\subsubsection{Loop group}
Let $G$ be a connected split reductive group over $\F_q$. Fix a pair of opposite Borel subgroups $B$ and $B^-$. Let $T$ be the intersection of $B$ and $B^-$.  Denote $U$ and $U^-$ as the unipotent radical of $B$ and $B^-$ respectively. Recall the functor of points of some group ind-schemes associated with $G, B, B^-, U, U^-, T$. Let $A$ be a $\F_q$-algebra.

\begin{enumerate}[{$(1)$}]
\item Loop group $\bold{G_K}$ is a group ind-scheme. It sends $A$ to $G(A((t)))$;

\item Positive loop group $\go$ is a group scheme. It sends $A$ to $G(A[[t]])$;

\item Standard Iwahori subgroup $\bold{I}$ is a group scheme. It sends $A$ to the elements of $G(A[[t]])$ whose image is in $B(A)$ when evaluating $t$ at zero.

\item Standard opposite Iwahori subgroup $\bold{I^-}$ is a group ind-scheme. It sends $A$ to the elements of $G(A[t^{-1}])$ whose image is in $B^-(A)$ when evaluating $t^{-1}$ at zero.

\item Pro-unipotent radical of standard Iwahori subgroup $\bold{I_u}$ is a group scheme. It sends $A$ to the elements of $G(A[[t]])$ whose image is in $U(A)$ when evaluating $t$ at zero.

\item Pro-unipotent radical of standard opposite Iwahori subgroup $\bold{I_u^-}$ is a group ind-scheme. It sends $A$ to the elements of $G(A[t^{-1}])$ whose image is in $U^-(A)$ when evaluating $t^{-1}$ at zero.

%\item $\bold{T_K}$, which sends a $\F_q$ algebra $A$ to $T(A((t)))$.

%\item $\bold{T_\co}$, which sends a $\F_q$ algebra $A$ to $T(A[[t]]))$.
\end{enumerate}

%Except for the loop group, any other group scheme in the list has a natural strict ind-variety structure over $\fq$. 

The affine flag variety $\gk / \bi$ and the enhanced affine flag variety $\gk / \biu$ play a central role in this paper. We will also work with the affine Grassmannian $\Gr_G= \gk / \go$ in Section \ref{sect7}. Note that these ind-schemes may have a non-reduced structure.

\subsubsection{Extended affine Weyl group}
Let $\extw$ be the extended affine Weyl group associated to $\gk$ and $\le$ be the standard Bruhat order of $\extw$. The extended affine Weyl group $\extw$ is the semidirect product $W \ltimes \X_*(T)$ of the finite Weyl group $W$ and the cocharacter lattice $\X_*(T)$ of $T$. Define $\bold{T_K}$ be the loop group of $T$ and $\bold{T_\co}$ be the positive loop group of $T$. We view $\bold{T_K}$ and $\bold{T_\co}$ as subgroup scheme of $\gk$. The extended affine Weyl group $\widetilde{W}$ is isomorphic to $N_{\gk}(\bold{T_K})(\fq) / \bold{T_\co}(\fq)$ as an abstract group, where $N_{\gk}(\bold{T_K})$ is the normalizer of $\bold{T_K}$ in $\gk$. 
We fix a set of coset representatives $\{\wdo\} \subset N_\gk(\bold{T_K})(\fq)$, where $\wdo$ is the coset representative of $w \in \extw$.

For $s$ a simple reflection in $\extw$, let $P_s$ be the standard parahoric subgroup corresponding to $s$. Denote the pro-unipotent radical of $P_s$ as $\bius$ and the Levi subgroup as $L_s$. 

Define $M(w,w')$ to be the subset $\{x \in \extw: x \ge uv$ for any $u\le w$ and $v \le w'\}$ for any $w, w' \in \extw$.

%After taking reduced structure, they become an ind-variety. 

%The natural projection $\pi:\gk / \biu \ra \gk / \bi$ is a $T$-torsor in Zariski topology because of the existence of big open cell. \cite{BY} lemma 2.3.2 

%Sometimes by abuse of notation we may write 

\subsection{Geometry}
All varieties in this article refer to quasi-projective schemes of finite type over $\fq$. In particular varieties can be non-reduced. For example, $\Spec(\fq[x]/(x^2))$ is a variety in our sense. The fact that $D^b(X)$ is equivalent to $D^b(X^{\red})$ implies no information is lost when passing to the reduced structure. 

In a triangulated category $\sC$, we say that $E \in \sC$ has a "filtration" whose subquotients are given by $Q_1, Q_2, \cdots, Q_n \in \sC$ if there exists a sequence of objects $0=F_0,F_1,\cdots , F_n=E$ in $\sC$ such that we can find distinguished triangles $F_{i-1} \ra F_i \ra Q_i \ra F_{i-1}[1]$ for any $1 \le i \le n$.  We use quotation marks to distinguish this notion of filtration from the usual one in abelian categories.

\subsubsection{Affine flag variety}

The left multiplication of $\bi$ induces a stratification on $\gk / \bi$. The strata are indexed by $\extw$ by the Bruhat decomposition. Each stratum contains a unique $\wdo$.  We denote $\bi \wdo \bi$ as $\gkw$. It is well known that $\gkw / \bi$ is an affine space of dimension $\ell(w)$ and $\gkw/ \bi \subset \overline{\gkwp/ \bi}$ if and only if $w \le w'$, where the overline stands for the closure. We denote $\overline{\gkw}$ by $\gksw$.

\subsubsection{Moduli space of G-bundles over $\p^1$} \label{modsec}

Let $\Bun_{G,0',\infty'}$ be the moduli stack classifying the data of $(\mathcal{E}, F_{0'}, F_{\infty'})$, where $\mathcal{E}$ is a $G$-bundle on $\p^1$, $F_{0'}$ is a $B$-reduction at zero and $F_{\infty'}$ is a $B^-$-reduction at infinity. For instance, when $G=GL_n$, $\Bun_{G,0',\infty'}$ classifies rank $n$ vector bundles on $\p^1$ together with a complete flag at zero and infinity respectively. 

There is a natural map from $\Bun_{G,0',\infty'}$ to $\Bun_{G,\p^1}=: \Bun_G$ by forgetting the two Borel reductions. Via this map we may view $\Bun_{G,0',\infty'}$ as a $(G/B \times G/B^-)$-fibration over $\Bun_{G}$. Since $\Bun_G$ is locally of finite type and smooth (see \cite[Proposition 4.4.6 and Corollary 4.5.2]{Beh}), so is $\Bun_{G,0',\infty'}$. Since we will be working with a fixed $G$, in the remainder of the paper we will write $\Bun_{0', \infty'}$ instead of $\Bun_{G,0',\infty'}$.

$\gk / \bi$ can be interpreted as the moduli space of $G$-bundles $\mathcal{E}$ on $\p^1$ with a Borel reduction at $0$ and a trivialization of $\mathcal{E}$ on $\p^1 \bsl \{0\}$ (see \cite[Section 1.1]{G1}). From this, we know that there is a natural map from $\gk / \bi$ to $\bbb$. This map induces a natural isomorphism $[\bim \bsl \gk / \bi] \simeq \bbb$ (see \cite[Theorem 1.3]{LS} and \cite[Theorem 2.3.7]{Z} ). From the Birkhoff decomposition, $[\bim \bsl \gk / \bi]$ is naturally stratified and the strata are indexed by $\extw$. For $w \in \extw$, let $\bbbw$ be the locally closed substack of $\bbb$ corresponding to the stratum $w$. $\bbbw$ is of codimension $\ell(w)$ in $\bbb$. In particular, $\bbbe$ is an open substack of $\bbb$.

It is well known that $\bbbw$ is in the closure of $\bbbwp$ if and only if $w \geq w'$. See the discussion before Lemma 6 in \cite{F} for a proof. Let $\bbbgw$ be the union of $\bbbwp$ for $w' \ge w$. It is the closure of $\bbbw$. Let $\bbbsw$ be the union of $\bbbwp$ for $w' \le w$. It is an open substack of $\bbb$. 

There are some variants of $\bbb$. Let $\bub$ (resp. $\bbu$, $\buu$) be the moduli stack of $G$-bundles on $\p^1$ with a $U$-reduction (resp. a $B$-reduction, a $U$-reduction) at zero and a $B^-$-reduction (resp. a $U^-$-reduction, a $U^-$-reduction) at infinity. There are natural smooth maps from each variant to $\bbb$ as a $U$-reduction gives a $B$-reduction naturally. In particular, $\buu$ is a $T \times T$-bundle over $\bbb$, which implies that $\buu$ is smooth. The stratification of $\bbb$ gives rise to a stratification of each variant. Define $\bubw$ as the stratum indexed by $w$ of $\bub$. Similarly, we can define $\buuw$, $\bubgw$ etc. 

We will also consider bundles with parahoric level structures. Let $\bcu$ be the moduli stack of $G$-bundle on $\p^1$ with a full level structure at zero and a $U^-$-reduction at infinity. That is, $\bcu$ classifies the data of $(\sE, F_\infty, \tau)$, where $\sE$ is a $G$-bundle on $\p^1$, $F_\infty$ is a $U^-$-reduction at infinity and $\tau$ is a trivialization of the restriction of $\sE$ to the formal disc at zero, which is identified with Spec $\fq[[t]]$. Notice that $\gk$ naturally acts on $\bcu$ from the right (see \cite[Construction 4.2.2]{Y1}).

Let $\bsu$ be the quotient stack of $\bcu$ by $\bius$, that is the moduli stack of $G$-bundles on $\p^1$ with a $\bius$-level structure at zero and a $U^-$-reduction at infinity. From its construction, we know that $\bsu$ is locally of finite type (\cite[Corollary 4.2.6]{Y1}) and it is naturally stratified with the index set $\{w: ws <w\}$. Let $\bsuw$ be the stratum corresponding to $w$. 

\begin{exmp}
	When $G=GL_2$ and $s$ is the non-affine simple reflection, $\bsu$ classifies the data of $(\sE_0 \supset \sE_1 \supset \sE_2=\sE_0 \otimes \co_{\p^1}(-\infty), \gamma, \gamma_0, \gamma_1)$, where $\co_{\p^1}(-\infty)$ is the ideal sheaf of $\infty$ in $\p^1$, $\sE_0$ and $\sE_1$ are vector bundles of rank $2$ on $\p^1$ such that $\sE_0 / \sE_1$ is of length $1$, $\gamma_i$ are isomorphisms $\sE_i/ \sE_{i+1} \xra{\sim} \delta_\infty$ for $i=0,1$, $\gamma$ is an isomorphism $\sE_{0}/ \sE_{0}\otimes \co_{\p^1}(-0) \xra{\sim} \delta_0^{\oplus 2}$, $\delta_0$ and $\delta_\infty$ are the skyscraper sheaf at $0$ and $\infty$ respectively.
\end{exmp}

\section{Affine monodromic Hecke categories}\label{sect3}

In this section, we introduce our main players in this paper $\lpdl$ and $\lpdlm$, the categories of positive sheaves and negative sheaves respectively. The former one is usually called affine monodromic Hecke categories. We give examples at the end. The monodromic sheaves are also discussed in \cite[Section 2]{BFO} and \cite[Section 5]{L}. 

%$D^b_{T \times T, \sL \boxtimes \sL'}(\bium \bsl \gk / \biu)$. When both $\sL$ and $\sL'$ are trivial, it reduces to  $D^b_{m,\bi}(\gk / \bi)$ and  $D^b_{m,\bim} (\gk / \bi)$. 

\subsection{Construction of $\lpdl$}\label{positive}

Let $X$ be a scheme of finite type over $\fq$, $J$ be a connected group acting on $X$ and $\sL$ a character sheaf on $J$. Lusztig and Yun defined a $\ql$-linear triangulated category $D^b_{(J,\sL),m}(X)$ in Section $2.5$ in \cite{LY}. It is a full subcategory of $D^b_{m}([\tilde{J} \bsl X])$ for some  \'etale covering $\tilde{J}$ of $J$. This construction can be generalized to the case when $X$ is a stack of finite type over $\fq$. Lusztig and Yun used it for the stack $U \bsl G /U$ in their paper. 

For each $w \in \extw$, there exists a normal, finite codimensional subgroup $J_w$ of  $\bi$ such that it is contained in $\biu$ and acts trivially on $\gksw / \biu$. Replacing $J_{w'}$ by their common intersection, we may assume $J_{w'} \subset J_w$ for $w \le w'$. Hence $D^b_m(\biu \bsl \gksw / \biu)$ can be defined as $D^b_m((\biu/J_w) \bsl \gksw / \biu)$, which is independent of the choice of $J_w$. We apply the construction of Lusztig and Yun to $(\biu/J_w) \bsl \gksw / \biu$ and obtain $D^b_{(T \times T, \sL \boxtimes \sL'),m}(\biu \bsl \gksw / \biu)$ for any two character sheaves $\sL$, $\sL'$ on $T$. To simplify the notation, we define $_{\sL'} D(\le w )_{\sL}:=D^b_{(T \times T, \sL' \boxtimes \sL^{-1}),m}(\biu \bsl \gksw / \biu)$ and $_{\sL'} D(w)_{\sL}:=D^b_{(T \times T, \sL' \boxtimes \sL^{-1}),m}(\biu \bsl \gkw / \biu)$.

%it the mixed equivariant derived category in the usual sense.

%The fact that $D^b_{J,m}(X)$ is a full subcategory of  $D^b(X)$ consists of complexes whose cohomology sheaf is constant along $J$-orbit implies the definition is independent of the choice of $J_w$ 

Once we define the monodromic equivariant derived categories for finitely many strata, we could define the monodromic equivariant derived categories for the affine flag variety 

\[\lpdl= \operatorname{2-}\uralim_{\\w\in\extw}\pha _{\sL'} D(\le w )_{\sL},\] to be the inductive $2$-limit of $_{\sL'} D(\le w )_{\sL}$ with respect to the inductive system of pushforward functors of closed embedding $i_{w,w'}:\biu \bsl \gksw / \biu \ra \biu \bsl \gkswp / \biu$ for $w \le w'$. Objects in this category are called \textbf{positive sheaves}. 

When $\sL$ and $\sL'$ are trivial, $\lpdl$ coincides with the usual Hecke category $D^b_m(\bi \bsl \gk / \bi)$ in \cite[Section 3]{BY} up to the difference between Kac-Moody group and the loop group.

%More generally, we can define $D^b_{(T \times T, \sL' \boxtimes \sL^{-1}),m}(\biu \bsl \bigcup_{w \in A} \gksw / \biu)$ for any finite set $A \subset \extw$ and denote it by $_{\sL'} D(\le A)_{\sL}$. For any two finite sets $A_1, A_2 \subset \extw$, we say $A_1 \le A_2$ if  $\bigcup_{w \in A_1} \gksw \subset \bigcup_{w \in A_2} \gksw$.
%
%Once we define the monodromic equivariant derived categories for finitely many strata, we could define the monodromic equivariant derived categories for the affine flag variety 
%
%\[\lpdl= \operatorname{2-}\uralim_{A \subset \extw}\pha _{\sL'} D(\le A)_{\sL}, \] 
%to be the inductive $2$-limit of $_{\sL'} D(\le A)_{\sL}$ with respect to the inductive system of pushforward functors of closed embedding $i_{A_1,A_2}:\biu \bsl \bigcup_{w \in A_1} \gksw / \biu \ra \biu \bsl \bigcup_{w \in A_2} \gksw / \biu$ for $A_1 \le A_2$. Objects in this category are called \textbf{positive sheaves}. 
 
\subsection{Construction of $\lpdlm$}
We study monodromic sheaves on $\bbbsw$.

Since $\bbbsw$ is a scheme of finite type quotient by an algebraic group \cite[Lemma 6]{F}, the category $D^b_m(\bbbsw)$ and its monodromic counterparts $_{\sL'} D(\le w)^-_{\sL}:=$ 
 $D^b_{(T \times T, \sL' \boxtimes \sL^{-1}),m}(\buusw)$ can be defined via the construction by Lusztig and Yun. 

For each pair $w' \le w$, the open embedding $j_{w,w'}: \buuswp \ra \buusw$ induces a functor of derived categories %$j_{w,w'}^*: D^b_{(T \times T, \sL \boxtimes \sL'),m}(\bium \bsl \gkswm / \biu) \ra D^b_{(T \times T, \sL \boxtimes \sL'),m}(\bium \bsl \gkswpm / \biu)$.
$j_{w,w'}^*: \pha_{\sL'} D(\le w)^-_{\sL} \ra  \pha _{\sL'} D(\le w')^-_{\sL}$. These functors form a projective system of triangulated categories. Let 
\[\lpdlm=   2-\ulalim_{\\w\in\extw} \pha_{\sL'} D(\le w)^-_{\sL}\]
%\[D^b_{(T \times T, \sL \boxtimes \sL'),m}(\bium \bsl \gk / \biu) =   2-\ulalim_{\\w\in\extw}D^b_{(T \times T, \sL \boxtimes \sL'),m}(\bium \bsl \gkswm / \biu)\]
be the projective $2$-limit of this system.

Concretely, objects in this category are $(\sF_w, \chi_{w,w'})$, where $\sF_w$ is a complex of sheaves on $\buusw$ with the prescribed monodromy and $\chi_{w,w'}$ is an isomorphism from $\sF_w$ to $j_{w,w'}^*\sF_{w'}$ such that $j_{w,w'}^*\chi_{w',w''} \circ \chi_{w,w'}= \chi_{w,w''}$ whenever $w \le w' \le w''$. The morphism between $(\sF_w, \chi_{w,w'})$ and $(\sG_w, \psi_{w,w'})$ is a family of morphisms $\phi_w: \sF_w \ra \sG_w$ such that $\psi_{w,w'} \circ \phi_w = \phi_{w'} \circ \chi_{w,w'}$ for any $w \le w'$. We call objects in this category  \textbf{negative sheaves}. 

\begin{rem}
	We do not claim that $\lpdlm$ is a triangulated category. Despite that the cohomology of a complex in the derived category $_{\sL'} D(\le w)^-_{\sL}$ is constructible, the complex itself can be very complicated. We do not know whether taking projective limit of such complexes is exact. Nevertheless, most of our computations involve only finitely many strata. Therefore, we only use the fact that $_{\sL'} D(\le w)^-_{\sL}$ is triangulated. 
\end{rem}

\subsection{Examples}

Let us mention some objects in the $2$-limit categories defined in the previous subsections. 

\subsubsection{Positive sheaves}

It is clear that $_{\sL'} D(w)_{\sL}$ is zero unless $\sL'= w\sL$. For $w \in \extw$ and its lifting $\wdo$, taking stalks at $\wdo$ induces an equivalence from  $_{w \sL} D(w)_{\sL}$ to $D^b_{(T \times T, w\sL \boxtimes \sL),m}(\wdo T) \simeq D^b_{\Gamma(w),m}(\wdo)$, where $\Gamma(w)=\{(wt,t)\} \subset T \times T$. Denote $C(\wdo)_\sL^+$ to be any sheaf corresponding to the constant sheaf $\ql \langle \ell(w) \rangle$ on $\wdo$ under this equivalence. Note that $C(\wdo)_\sL^+$ is unique up to a scalar. %For different choices of lifting $\wdo$ and $\wdo'$, $C(\wdo)_\sL^+$ is non-canonically isomorphic to  $C(\wdo')_\sL^+$.

From the local system $C(\wdo)_\sL^+$, we can construct other related sheaves in $_{\sL'} D(\le w )_{\sL}$ such as the standard sheaf $\Delta(\wdo)_\sL^+ :=i_{w,!}C(\wdo)_\sL^+$, the costandard sheaf $ \nabla(\wdo)_\sL^+:=i_{w,*}C(\wdo)_\sL^+$ and the intersection cohomology sheaf $\IC(\wdo)_\sL^+:=i_{w,!*}C(\wdo)_\sL^+$, where $i_w$ is the open embedding from $\biu \bsl \gkw / \biu$ to its closure $\biu \bsl \gksw / \biu$. The same notations are used to denote their images in $\lpdl$, and they are called \textbf{positive standard sheaf}, \textbf{positive costandard sheaf} and \textbf{positive IC sheaf} respectively. Let $\delta_\sL$ be the positive standard sheaf corresponding to the identity element $e$. Notice that $\omega\Delta(\wdo)_\sL^+$, $\omega\IC(\wdo)_\sL^+$ and $\omega\nabla(\wdo)_\sL^+$ are isomorphic for any lifting $\wdo$. We denote these isomorphism classes under $\omega$ by $\underline{\Delta}(w)_\sL^+$, $\underline{\IC}(w)_\sL^+$ and $\underline{\nabla}(w)_\sL^+$ respectively.

%Since closed embedding sends perverse sheaves to perverse sheaves, it makes sense to talk about perverse sheaves in $_{\sL'} D_{\sL}$ and perverse cohomology. The positive IC sheaf introduced in the last paragraph is perverse, but not necessarily for positive standard sheaf and positive costandard sheaf.

\subsubsection{Negative sheaves}

Denote $D^b_{(T \times T, \sL' \boxtimes \sL^{-1}),m}(\buuw)$ as $_{\sL'}D(w)^-_{\sL}$. We can define the constant sheaf $C(\wdo)_\sL^- \in \pha _{\sL'}D(w)^-_{\sL}$, whose stalk at $\wdo$ is $\ql \langle -\ell(w) \rangle$. Note that $_{\sL'}D(w)^-_{\sL}$ is zero unless $\sL'= w\sL$.  

Since $!$-extension, $*$-extension and IC-extension from open subsets are well-behaved with respect to open restriction, we can define the \textbf{negative standard sheaf} $\Delta(\wdo)_\sL^-:=j_{w,!}C(\wdo)_\sL^-$, the \textbf{negative costandard sheaf} $\nabla(\wdo)_\sL^-:=j_{w,*}C(\wdo)_\sL^-$, and the \textbf{negative IC sheaf} $\IC(\wdo)_\sL^-:=j_{w,!*}C(\wdo)_\sL^-$ in $\lpdlm$, where $j_w$ is the embedding from $\buuw$ to $\buu$. Similar to positive sheaves, we denote the isomorphism classes of images of above sheaves under $\omega$ by $\underline{\Delta}(w)_\sL^-$, $\underline{\IC}(w)_\sL^-$ and $\underline{\nabla}(w)_\sL^-$ respectively.

%To define standard objects $\Delta(\wdo)_\sL^-$, costandard objects $\nabla(\wdo)_\sL^-$ and IC objects $\IC(\wdo)_\sL^-$ in $\lpdlm$, we only need the fact that $!$-extension, $*$-extension and IC-extension from open subset are well-behaved with respect to open restriction. 

As an illustration, we construct $\IC(\wdo)_\sL^-$ explicitly and show that it is well-defined. For any $w'' \ge w' \ge w$, we have a chain of embeddings of stacks of finite type, $\buuw \xrightarrow{j_{w}^-} \buuswp \xrightarrow{j_{w',w''}}  \buuswpp$. Since $j_{w',w''}$ is an open embedding, $(j_{w',w''})^*((j_{w',w''} \circ j_{w}^-)_{!*}(C(\wdo)_\sL^-))$ is canonically isomorphic to $(j_{w}^-)_{!*}(C(\wdo)_\sL^-))$ via the restriction map. Therefore the projective system $(j_{w}^-)_{!*}(C(\wdo)_\sL^-)$ yields a well-defined complex supported on the connected component of $\buu$ containing the $w$-stratum. We define $\IC(\wdo)_\sL^-$ as the extension by zero of $(j_{w}^-)_{!*}(C(\wdo)_\sL^-)$ to $\buu$. 

\subsection{Basic operations}
The construction of sheaf hom $\RuHom$ in \cite[Section 2.8]{LY} can be easily generalized to affine case. By viewing negative sheaves as a projective system of complexes on $\gk/ \biu$ (see the discussion below Claim \ref{welldef}), $\RuHom(\sF, \sG)$ is also well-defined for $\sF \in \pha \lpdlm$, $\sG \in \pha \lpdl$. 

We define the renormalized Verdier dualities $\D^+$ and $\D^-$ on $\lpdl$ and $\lpdlm$ respectively as follows. 

For $\sF \in \pha \lpdl$, suppose its support is contained in a finite type subvariety $X \subset \gk / \biu$. Define $\D^+(\sF) :=\D_{X}(\sF)[2 \dim(T)]$, where $\D_{X}$ is the usual Verdier duality for the scheme $X$. It is independent of the choice of $X$. Moreover, it is constructible along $\biu$-orbit and its monodromy is reversed. Hence $\D^+$ is a contravariant functor from $\lpdl$ to $\lpmdlm$. We write $\lpdl^{\le 0}$ (resp. $\lpdl^{\ge 0}$) for the full subcategory of $\lpdl$ generated by $\sF \in \pha \lpdl$ such that $\sF [\dim(T)]$ is in $^p D^{ \le 0}(\gk / \biu)$ (resp. $^p D^{ \ge 0}(\gk / \biu)$). Since pushforward under closed embedding is perverse exact, it is justified to define $^p D^{ \le 0}(\gk / \biu)$ and $^p D^{ \ge 0}(\gk / \biu)$. 

For $\sF \in \pha \lpdlm$, define $\D^-(\sF):=\D_{\buu}(\sF)[-2 \dim(T)]$, where $\D_{\buu}$ is the standard Verdier duality on $\buu$.
Note that it is well-defined because open restriction commutes with the standard Verdier duality. 

Let $\lpdl^{-, \le 0}$ (resp. $\lpdl^{-, \ge 0}$) be the full subcategory of $\lpdlm$ generated by objects $\sF \in \pha \lpdlm$ such that $\sF [\dim(T)]$ is in $^pD^{ \le 0}(\buu)$ (resp. $^p D^{ \ge 0}(\buu)$). Since taking restriction to an open subset is perverse exact, it makes sense to define $^p D^{ \le 0}(\buu)$ and $^p D^{ \ge 0}(\buu)$. More precisely, $(\sF_w)_{w \in \extw} \in \pha ^p D^{ \le 0}(\buu)$ if and only if $\sF_w \in  \pha ^p D^{ \le 0}(\buusw)$. Similarly, $(\sF_w)_{w \in \extw} \in \pha ^p D^{ \ge 0}(\buu)$ if and only if $\sF_w \in  \pha ^p D^{ \ge 0}(\buusw)$.

The renormalized Verdier duality $\D^+$ (resp. $\D^-$) is an involutive anti-equivalence of categories between $\lpdl$ and $\lpmdlm$ (resp. $\lpdlm$ and $\lpmdlmm$).

Starting from this point, we define perverse sheaves to be the objects in 
$\lpdl^{\le 0} \cap \pha \lpdl^{\ge 0}$ or $\lpdl^{-, \le 0} \cap \pha \lpdl^{-, \ge 0}$. By our degree shift convention, $\D^+$ and $\D^-$ preserve perverse sheaves. Both positive and negative IC sheaves are perverse in this sense. 

%Suppose $\sF, \sG \in \pha \lpdlm$. Let $\sF_{\le w}$, $\sG_{\le w}$ be their restrictions in $ \lpdswlm$ respectively. A homomorphism between $\sF$ and $\sG$ is a compatible system of homomorphisms between $\sF_{\le w}$ and $\sG_{\le w}$. %Note that $\sF \ra \sG \ra \sH \xra{[1]}$ is a distinguished triangle if $\sF_{\le w} \ra \sG_{\le w} \ra \sH_{\le w} \xra{[1]}$ is a distinguished triangle in $ \lpdswlm$ for any $w \in \extw$. 

\section{Convolution product}\label{sect4}
In this section, we give the definition of convolution products between the sheaves of different types. 

\subsection{Convolution product of two positive sheaves}

The convolution product of the affine monodromic Hecke categories is defined in a similar way to \cite[Section 3]{LY} and \cite[Section 4]{MV} by taking both infinite-dimensional varieties and monodromies into account. 

%For any $w, w' \in \extw$, denote $M(w,w')$ to be the subset $\{x \in \extw: x \ge uv$ for any $u\le w$ and $v \le w'\}$. 

Let $\sF \in \pha_{\sL''} D(\le w )_{\sL'}$ and $\sG \in \pha_{\sL'} D(\le w')_{\sL}$. Choose $w'' \in M(w,w')$. Consider the following maps \[
\gksw / \biu \times \gkswp /  \biu \xla{p} \gksw /J_{w'} \times  \gkswp /  \biu \xra{q} \gksw /J_{w'} \times^{\biu / J_{w'}} \gkswp /  \biu . \]
\[\gksw /J_{w'} \times^{\biu / J_{w'}} \gkswp /  \biu \xra{\phi_{w,w'}^+}\gksw /J_{w'} \times^{\bi / J_{w'}} \gkswp /  \biu \xra{\mu_{w,w'}^+} \gkswpp / \biu \]

Here $p, q$ are the canonical quotient maps and $\mu_{w,w'}^+$ is the multiplication map. Since $J_{w'}$ is normal in $\bi$, there is a natural action of $T$ on $\gksw /J_{w'} \times^{\biu / J_{w'}} \gkswp /  \biu$ by setting $t(g_1, g_2)= (g_1t^{-1}, tg_2)$. Let $\phi_{w,w'}^+$ be the quotient map of this action. 

Since $p^*(\sF \boxtimes \sG)$ is $\biu / J_{w'}$-equivariant, we can find a sheaf $ \sF \widetilde{\boxtimes} \sG$ on $\gksw /J_{w'} \times^{\biu / J_{w'}} \gkswp /  \biu$ such that $p^*(\sF \boxtimes \sG) \cong q^*(\sF \widetilde{\boxtimes} \sG)$. Similarly, $\sF \widetilde{\boxtimes} \sG$ has a natural $T$-equivariant structure under the aforementioned $T$ action, hence it descends to a sheaf $\sF {\odot} \sG$ on  $\gksw /J_{w'} \times^{\bi / J_{w'}} \gkswp /  \biu$.

$(\mu_{w,w'}^+)_*(\sF \odot \sG)$ is $\biu$-equivariant and hence can be viewed as an object in $\lppdl$. One can check that it is independent of the choice of $w, w', w'', J_{w'}$. Observe that both leftmost and rightmost $T$-monodromic structures are preserved under the above operations. Hence we can define a functor $\conp: \pha_{\sL''} D(\le w)_{\sL'} \times  \pha_{\sL'} D(\le w')_{\sL} \ra \pha_{\sL''} D_{\sL}$ by sending $(\sF, \sG)$ to $\sF \star_+ \sG:=(\mu_{w,w'}^+)_*(\sF \odot \sG)$.

It is straightforward to generalize the above construction to $\sF$ and $\sG$ with general supports by replacing $\gksw / \biu$ with the finite union $\bigcup_{w \in A}\gksw / \biu$ for some finite subset $A \subset \extw$. Therefore we can define $\conp: \pha_{\sL''} D_{\sL'} \times  \pha_{\sL'} D_{\sL} \ra \pha_{\sL''} D_{\sL}$ using a similar formula.

Since all the above maps are of finite type, the base change theorem shows that $\conp$ is associative. It is also clear that $\delta_\sL$ is the monoidal unit for $\conp$. 

\begin{rem}\label{disconsup}
	As the referee pointed out, there may not exist a $w \in \extw$ such that $\sF \in \pha_{\sL'} D(\le w)_{\sL}$ for $\sF \in \pha \lpdl$. This is due to the fact that $(\extw, \le)$ is not a directed set. Nevertheless, $\sF$ can be written as a finite direct sum of $\sF_w$ for some $\sF_w \in \pha_{\sL'} D(\le w)_{\sL}$ and $w \in \extw$. All constructions and proofs in the paper can be reduced to the case when $\sF \in \pha_{\sL'} D(\le w)_{\sL}$ for some $w \in \extw$. A similar phenomenon occurs in other categories. To simplify our notation, we may focus on cases where the sheaf support lies in a connected component of the affine flag variety (and its analog) throughout the paper.
\end{rem}

\subsection{Convolution product of negative sheaves and positive sheaves}

We also need to define the convolution product of negative sheaves and positive sheaves. In the case of trivial monodromies, it is known as the Hecke modification in the literature (see \cite[Section 2]{G2}). Most properties of this special case can be generalized to monodromic cases without much difficulty. 

Let $\hec$ be the Hecke stack of $\buu$ with a modification at zero. That is, $\hec$ classifies the data of $(\mathcal{E}^1, F_{0}^1, F_{\infty}^1, \mathcal{E}^2, F_{0}^2, F_{\infty}^2, \beta)$, where $(\mathcal{E}^i, F_{0}^i, F_{\infty}^i) \in \buu$ for $i=1,2$ and $\beta$ is an isomorphism $\mathcal{E}^1|_{\p^1-\{0\}} \simeq \mathcal{E}^2|_{\p^1-\{0\}}$ sending $F_{\infty}^1$ to $F_{\infty}^2$.

Let $pr_1, pr_2$ be the two projections from $\hec$ to $\buu$. The fibers of $pr_1$ (or $pr_2$) are isomorphic to $\gk / \biu$. Hence we can view $\hec$ as a twisted product of $\buu$ and $\gk / \biu$. Let $\hecswa$ (resp. $\hecwa$) be the inverse image $pr_1^{-1}(\buusw)$ (resp. $pr_1^{-1}(\buuw)$). 

The stratification of $\gk / \biu$ induces a stratification of $\hec$. Denote the substack corresponding to $\gksw / \biu$ (resp. $\gkw / \biu$) as $\hecsw$ (resp. $\hecw$). $\hecsw$ can be realized as a twisted product of $\buu$ and $\gksw / \biu$. For any $w,w' \in \extw$, let $\hecswswpa$ be the intersection of $\hecsw$ and $\hecswpa$. Similarly, we can define $\hecwwpa$, $\hecwswpa$ and $\hecswwpa$.

Let $\hecb$ be the quotient of $\hec$ by the $T$-action on the $U$-reduction at zero of the first $G$-bundle. That is, $\hecb$ classifies the data of $(\mathcal{E}^1, F_{0'}^1, F_{\infty}^1, \mathcal{E}^2, F_{0}^2, F_{\infty}^2, \beta)$, where $F_{0'}^1$ is a $B$-reduction of $\mathcal{E}^1$ and the rest is same as $\hec$. Since $\hecswswpa$ is stable under the $T$-action, we can define $\hecbswswpa$ as the quotient of $\hecswswpa$ by $T$, which is a substack of $\hecb$. 

By abuse of notation, we denote the projections from $\hecb$ to $\bbu$ (resp. $\buu$) by $pr_1$ (resp. $pr_2$) respectively. We also use the same notation when restricted to a substack. 

The definition of convolution product of a negative sheaf $\sF$ and a positive sheaf $\sG$ is as follows.

%\begin{lemma}
%Let $\mu^-: \bium \bsl \gk \times \gk / \biu \ra \bium \bsl \gk / \biu$ be the multiplication map and $s$ be a simple reflection corresponds to root $\alpha$. Then $\mu(\bium \bsl \gkwm \times \gkss / \biu)$ is the union of $ \bium \bsl \gkwm / \biu$ and $ \bium \bsl \gkwsm / \biu$.
%\end{lemma}

%\begin{proof}
%$\gkss=\sdo \bi \cup U_{-\alpha}(a) \bi$??
%\end{proof}

%Let $\mu^-: \bium \bsl \gk / \biu \times \biu \bsl \gk / \biu \ra \bium \bsl \gk / \biu$ be the multiplication map. For any $w, w' \in \extw$, by above lemma?? in geometry, we know there exists $w'' \in \extw$ such that $pr_1((\mu^-)^{-1}(\bium \bsl \gkswm / \biu) \cap pr_2^{-1}(\biu \bsl \gkswp / \biu))$ is contained in $\bium \bsl \gkswppm / \biu$ where $pr_1, pr_2$ are the projections from $\bium \bsl \gk / \biu \times \biu \bsl \gk / \biu$ to $\bium \bsl \gk / \biu$ and $\biu \bsl \gk / \biu$ respectively.%By the construction of $w''$, it is obvious that $\mu^{-1}(\bium \bsl \gkswm / \biu) \subset \bium \bsl \gkswppm / \biu \times \biu \bsl \gkswp / \biu$. Call this subset $X_{w,w',w''}$.

Let $\sF=(\sF_w)_{w \in \extw} \in \pha_{\sL''}D^-_{\sL'}$, $\sG \in  \pha_{\sL'}D(\le w')_{\sL}$. For each $w \in \extw$, we choose $w'' \in M(w,w'^{-1})$ and $w''' \in M(w'',w')$. Consider the following maps

\[\buuswpp \xleftarrow{pr_1} \hecswpswppa \xrightarrow{q} \hecbswpswppa \xrightarrow{pr_2} \buuswppp.\]

Since $\sG$ is $\biu$-equivariant, we can form the twisted product $\sF_{w''} \widetilde{\boxtimes} \sG$ on $\hecswpswppa$. It descends to a sheaf $\sF_{w''} \odot \sG$ on $\hecbswpswppa$. Hence we obtain a sheaf $pr_{2*} (\sF_{w''} \odot \sG) \in \pha_{\sL''}D(\le w''')^-_{\sL}$. It is not hard to check that its restriction to $\buusw$ is independent to the choice of $w', w'', w'''$. They form a projective system when $w$ runs through $\extw$. From this, we obtain a negative sheaf and denote it as $\sF \star_- \sG \in \pha_{\sL''}D^-_{\sL}$. We can similarly define $\sF \star_- \sG$ when both $\sF$ and $\sG$ have disconnected support (see Remark \ref{disconsup}). It is clear that $\star_-$ is a functor from $_{\sL''}D^-_{\sL'} \times \pha _{\sL'}D_{\sL}$ to $_{\sL''}D^-_{\sL}$.

By the base change theorem, we have $\sF \conn (\sG \conp \sH) \simeq (\sF \conn \sG) \conn \sH$ for $\sF \in \pha\lpppdlppm$, $\sG \in \pha\lppdlp$, $\sH \in \pha\lpdl$.

%$\sF \in {\ulad}^b_{(T \times T, \sL \boxtimes \sL'),m}(\gk / \biu)$ to ${\ulad}^b_{(T \times T, \sL \boxtimes \sL'),m}(\biu \bsl \gk / \biu)$. 

%\begin{tikzcd}[row sep=large]
%& \biu \bsl \gk \times^ \biu \gk / \biu \arrow[dl, "\pi_1^+"'] \arrow[dr, "\pi_2^+"] \arrow[r, "\phi^+"] & \biu \bsl \gk \times^ \bi \gk / \biu \arrow[r, "\mu^+"] &\biu \bsl \gk / \biu\\
%\biu \bsl \gk / \biu   & & \biu \bsl \gk / \biu
%\end{tikzcd}

%\begin{tikzcd}%[column sep=small]
%& \bium \bsl \gk \times^ \biu \gk / \biu \arrow[dl, "\pi_1^-"'] \arrow[dr, "\pi_2^-"] \arrow[r, "\phi^-"] & \bium \bsl \gk \times^ \bi \gk / \biu \arrow[r, "\mu^-"] &\bium \bsl \gk / \biu\\
%\bium \bsl \gk / \biu   & & \biu \bsl \gk / \biu
%\end{tikzcd}

%For $\sF \in D^b_{(T \times T, \sL \boxtimes \sL'),m}(\bium \bsl \gk / \biu)$, $\sG \in D^b_{(T \times T, \sL' \boxtimes \sL''),m}(\biu \bsl \gk / \biu)$

\subsection{Definition of $\extw_{\sL}$ and $\extw_{\sL}^\circ$}\label{endodef}

Lusztig defined $\extw_{\sL}^\circ$ attached to any $\sL\in \Ch(T)$ in \cite[Section 2]{L}.\footnote{Lusztig used $W^x$ instead of $\extw_{\sL}^\circ$. Here we follow the notation in \cite{LY}.} Since this group plays an important role in this paper, we recall its construction. First, we construct the root system attached to $\sL\in \Ch(T)$. 

Let $Q^\vee$ be the coroot lattice of $G$, $\Phi ^\vee_\sL$ is the subset of coroots in $\Phi ^\vee(G, T)$ such that the restriction of $\sL$ to those coroots are trivial. Let $\Phi_\sL$ be the corresponding subset of roots. We form the endoscopic group $H=H_{\sL}$ using the root datum $(\X^*(T), \X_*(T), \Phi_\sL, \Phi^\vee_\sL)$. Note that $H$ and $G$ share the same maximal torus, but $H$ is not a subgroup of $G$ in general. $W_\sL^\circ$ is the Weyl group of this root system. It is a normal subgroup of $W_\sL$, the stabilizer of $\sL$ in the finite Weyl group $W$ of $G$.

To define (real) affine coroots of $G$, we shall fix a non-trivial central extension $\egk$ of the loop group $\gk$ (see Section \ref{cent} and \cite[Section 6.2.2]{Y1}). The adjoint map $\Ad: G \ra GL(\lig)$ provides us a natural choice of a central extension. The rotation torus $\gmr$ naturally acts on $\egk$. The semidirect product $\gmr \ltimes \egk$ is an affine Kac-Moody group of $G$ and $\gmr \ltimes \eT$ is a maximal torus, where $\eT$ is the central extension of $T$ in $\egk$. 

The affine roots are characters of the maximal torus $\gmr \ltimes \eT$ acting on the Lie algebra of $\gmr \ltimes \egk$. We call an affine root to be real if the corresponding eigenspace is not contained in the loop group of torus. Each real affine root $\alpha$ is associated with a one-dimensional unipotent group $U_{\alpha}$ in $\gmr \ltimes \egk$. There is a central isogeny $\phi_{\alpha}$ from $SL_2$ to $G_{\alpha}$, where $G_{\alpha}$ is the subgroup of $\gmr \ltimes \egk$ generated by $U_{\alpha}$ and $U_{-\alpha}$. The real affine coroot $\alpha^\vee$ is obtained by restricting $\phi_{\alpha}$ to the standard torus $\gm$ of $SL_2$. One can verify that $\alpha^\vee$ is inside the subspace $\X_*(\eT) \subset \X_*(\gmr \ltimes \eT)$. Moreover, each $\alpha^\vee$ gives rise to a reflection $s_{\alpha^\vee}$ in $\extw$. 	We denote $\widetilde{\Phi} ^\vee$ to be set of all real affine coroots. 

The character sheaf $\sL \in \Ch(T)$ can be considered as a character sheaf on $\eT$. We also denote it by $\sL$. We define $\widetilde{\Phi} ^\vee_\sL$ to be the subset of real affine coroots $\{\alpha^\vee \in \widetilde{\Phi} ^\vee: \alpha^{\vee*}\sL$ is trivial$\}$. Note that $\alpha^\vee$ is in $\widetilde{\Phi} ^\vee_\sL$ if and only if its finite part is in $\Phi^\vee_\sL$. In particular, $\widetilde{\Phi} ^\vee_\sL$ is independent of the choice of the central extension $\egk$. 

Let $\ewlo$ be the subgroup of $\extw$ generated by the reflections corresponding to the affine coroots in $\widetilde{\Phi} ^\vee_\sL$. The results in \cite[Chapter V, \textsection 3]{Bo} show that $\ewlo$ is a Coxeter group. On the other hand, the affine Weyl group of $H$, $\ewho=W_H \ltimes Q_H^\vee$, is a subgroup of $\widetilde{W}=W \ltimes \X_*(T)$. By the construction, we know that $\ewlo=\ewho$ and they are generated by the same set of reflections.

It is not hard to prove that $\extw_{\sL}^\circ$ is a normal subgroup in $\extw_{\sL}$, the stabilizer of $\sL$ in $\extw$. Here the action of $\extw$ on $\Ch(T)$ is defined to factor through $W$. Note that $\extw_{\sL}$ is isomorphic to $W_\sL \ltimes \X_*(T)$. Henceforth we fix $\lio$, a $W$-orbit of $\Ch(T)$. 
\begin{defn}
Let $\sL, \sL' \in \lio$. Define $_{\sL'} \extw_\sL :=\{ w \in \extw: w(\sL)=\sL'\}= \pha _{\sL'} W_\sL \ltimes \X_*(T)$ and define $_{\sL'} \ulextw_\sL := \pha _{\sL'} \extw_\sL / \extw_{\sL}^\circ = \extw_{\sL'}^\circ \bsl \pha _{\sL'} \extw_\sL$. Elements in this coset space are called the \textbf{blocks} of $\lio$. 
\end{defn}

\begin{exmp}\label{example}
Let $G=Sp_{2n}$. The roots of $G$ are $\{\pm L_i \pm L_j\} \cup \{ \pm 2L_i \}$ and the coroots are $\{\pm e_i \pm e_j\} \cup \{ \pm e_i \}$, where $L_i$'s are orthonormal and $e_i$ is the dual basis. The cocharacter lattice is spanned by the coroots and the character lattice is dual to the cocharacter lattice. There exists an order $2$ character sheaf $\sL$, which is fixed by the Weyl group of $G$, with $\Phi^\vee_\sL $ equal to the set of all long coroots $\{\pm e_i \pm e_j\}$. Hence $\Phi_\sL $ is the set of all short roots  $\{\pm L_i \pm L_j\}$. In this case, the endoscopic group $H_\sL$ is $SO_{2n}$. Moreover $\extw_{\sL}^\circ$ is an index four normal subgroup of $\extw_{\sL} = \extw$, that is there are four blocks of $\lio$. 
\end{exmp}

\section{Basic properties of Hecke categories}\label{sect5}
Most results in Sections $3$ and $4$ of \cite{LY} remain true for positive sheaves by replacing the groups with their affine analogs. Since the arguments are similar to that in \cite{LY}, we do not reproduce the arguments for positive sheaves and only give a sketch of the proof for negative sheaves. Let us first introduce the analogs of the lemmas used in \cite{LY} for our cases. 

\begin{lem}\label{Verdier}
For $\sF \in \pha \lppdlp$, $\sG \in \pha \lppdlpm$ and $\sH \in \pha \lpdl$,   there are canonical isomorphisms 
\[ \D^+(\sF \conp \sH) \cong \D^+ (\sF) \conp \D^+(\sH), \qquad \D^-(\sG \conn \sH) \cong \D^- (\sG) \conn \D^+(\sH).\] 
\end{lem}
\begin{proof}
 Note that $\gksw / \bi$ is proper. 
\end{proof}

\begin{lem}
\begin{enumerate}[{$(1)$}]
\item If $\sF \in  \pha\pha \lppdlp$, $\sG \in \pha \lppdlpm$ and $\sH \in \pha \lpdl$ are semisimple complexes, that is their images under $\omega$ are direct sums of shifted simple non-mixed perverse sheaves, then $\sF \conp \sH$ and $\sG \conn \sH$ are also semisimple.
\item  If $\sF \in  \pha \lppdlp$, $\sG \in \pha \lppdlpm$ and $\sH \in \pha \lpdl$ are pure of weight zero, then so are $\sF \conp \sH$ and $\sG \conn \sH$.
\end{enumerate}
\end{lem}
\begin{proof}
The same reason as the previous lemma. 
\end{proof}

\begin{lem}\label{stdstd}
Let $w_1, w_2, w_3, w_4 \in \extw$ such that $\ell(w_1)+\ell(w_2)= \ell(w_1w_2)$ and $\ell(w_3)-\ell(w_4)=\ell(w_3w_4)$.

Then we have canonical isomorphisms 
\[ \Delta (\wdo_1)^+_{w_2\sL} \conp \Delta (\wdo_2)^+_{\sL} \cong \Delta (\wdo_1\wdo_2)^+_{\sL},  \qquad  \nabla (\wdo_1)^+_{w_2\sL} \conp \nabla (\wdo_2)^+_{\sL} \cong \nabla (\wdo_1\wdo_2)^+_{\sL},\] 
\[\Delta (\wdo_3)^-_{w_4\sL} \conn \Delta (\wdo_4)^+_{\sL} \cong \Delta (\wdo_3\wdo_4)^-_{\sL},  \qquad  \nabla (\wdo_3)^-_{w_4\sL} \conn \nabla (\wdo_4)^+_{\sL} \cong \nabla (\wdo_3\wdo_4)^-_{\sL}.\]

\end{lem}

\begin{proof}
It follows from the fact that $\hecbwf^{w_3} \ra \buu^{w_3w_4}$ is a biregular morphism if $\ell(w_3)-\ell(w_4)=\ell(w_3w_4)$. To prove the fact, we write $w_4$ in a reduced word expression $s_{i_1}s_{i_2}\cdots s_{i_n}$. By induction, it suffices to show that $\hecbs^{w} \ra \buu^{ws}$ is a biregular morphism if $ws<w$. It follows from Birkhoff decomposition (c.f. \cite[Section 3]{F}).
\end{proof}

\begin{lem}\label{inverse}
For $w \in \extw$, there are isomorphisms 
\[ \Delta (\wdo^{-1})^+_{w\sL} \conp \nabla (\wdo)^+_{\sL} \cong \Delta (e)^+_{\sL} \cong  \nabla (\wdo^{-1})^+_{w\sL} \conp \Delta (\wdo)^+_{\sL}.\] 
When $w$ is a simple reflection of $\extw$, these isomorphisms can be chosen canonically. Therefore the functor $(-) \conp \Delta (\wdo)^+_{\sL}$ is an equivalence from $\lpdwl$ to $ \lpdl$ with inverse given by $(-) \conp  \nabla (\wdo^{-1})^+_{w\sL}$. Similarly, the functor $(-) \conn \Delta (\wdo)^+_{\sL}$ is an equivalence from $\lpdwlm$ to $ \lpdlm$ with the inverse given by $(-) \conn  \nabla (\wdo^{-1})^+_{w\sL}$.
\end{lem}
\begin{proof}
See Lemma $3.5$ in \cite{LY}.
\end{proof}

\begin{lem} \label{equiv}
Let $s \in \extw$ be a simple reflection and $s \notin \ewlo$. \begin{enumerate}[{$(1)$}]
\item The natural maps $\Delta (\sdo)^+_{\sL} \ra \IC (\sdo)^+_{\sL} \ra \nabla (\sdo)^+_{\sL}$ are isomorphisms. 
\item The functor $(-) \conp \IC (\sdo)^+_{\sL}: \pha \lpdsl \ra  \pha \lpdl$ is an equivalence and the inverse is given by  $(-) \conp \IC (\sdo^{-1})^+_{\sL}$.
\item The functor $(-) \conn \IC (\sdo)^+_{\sL}:\pha  \lpdslm \ra  \pha \lpdlm$ is an equivalence and the inverse is given by  $(-) \conn \IC (\sdo^{-1})^+_{\sL}$.
\item For any $w \in \pha _{\sL'} \extw _{s\sL}$, the equivalence $(-) \conp \IC (\sdo)^+_{\sL}$ sends $\Delta (\wdo)^+_{s\sL}, \IC (\wdo)^+_{s\sL}$, $\nabla (\wdo)^+_{s\sL}$ to the isomorphism classes $\Delta (\wdo\sdo)^+_{\sL}, \IC (\wdo\sdo)^+_{\sL}, \nabla (\wdo\sdo)^+_{\sL}$ respectively. 
\item For any $w \in \pha _{\sL'} \extw _{s\sL}$, the equivalence $(-) \conn \IC (\sdo)^+_{\sL}$ sends $\Delta (\wdo)^-_{s\sL}, \IC (\wdo)^-_{s\sL}$, $\nabla (\wdo)^-_{s\sL}$ to the isomorphism classes $\Delta (\wdo\sdo)^-_{\sL}, \IC (\wdo\sdo)^-_{\sL}, \nabla (\wdo\sdo)^-_{\sL}$ respectively. 

\end{enumerate}
\end{lem}
\begin{proof}
The proof of statements $(1),(2)$, and $(4)$ follows from Proposition $5.2$ in \cite{LY}. It also follows from the block decomposition (see Proposition \ref{blockdec} below). 
Statement $(3)$ follows from statement $(1)$ and Lemma \ref{inverse}. For statement $(5)$, if $\ell(w)= \ell(ws)+1$, then $\Delta (\wdo)^-_{s\sL} \conn \IC (\sdo)^+_{\sL} \cong\Delta (\wdo)^-_{s\sL} \conn \Delta (\sdo)^+_{\sL} \cong \Delta (\wdo\sdo)^-_{\sL}$ by Lemma \ref{stdstd}. If $\ell(w)= \ell(ws)-1$, then $\Delta (\wdo)^-_{s\sL} \conn \IC (\sdo)^+_{\sL} \cong \Delta (\wdo)^-_{s\sL} \conn \nabla (\sdo)^+_{\sL} \cong \Delta (\wdo\sdo)^-_{\sL}$ by the same lemma. The same argument works for costandard sheaves. It remains to prove that the functor sends $\IC (\wdo)^-_{s\sL}$ to $\IC (\wdo\sdo)^-_{\sL}$. 

By Lemma \ref{Verdier}, we know that $\sF:= \IC (\wdo)^-_{s\sL} \conn \IC (\sdo)^+_{\sL}$ is Verdier self-dual. To prove it is an IC sheaf, it suffices to check its stalk on each stratum satisfying a certain cohomology degree bound. The stalk can be computed via taking homomorphism between $\sF$ and costandard sheaves. Since the equivalence preserves costandard sheaves, we know that $\sF$ is an IC sheaf. From the fact that there is a nonzero map from $\Delta (\wdo\sdo)^-_{\sL} \cong \Delta (\wdo)^-_{s\sL} \conn \IC (\sdo)^+_{\sL}$ to $\IC (\wdo)^-_{s\sL} \conn \IC (\sdo)^+_{\sL}$, the latter sheaf has to be $\IC (\wdo\sdo)^-_{\sL}$.
\end{proof}

%Let $\busu$ be the moduli stack of $G$-bundles on $\p^1$ with a $U$-reduction at infinity and a $\bius$-level structure at zero. That is, $\busu$ classifies the data of $()$ 

Let $s$ be a simple reflection in $\extw$ such that $s \in \ewlo$. There is a \textit{canonical} object $\IC(s)^+_\sL \in \pha \ldl$ with an isomorphism from its stalk at $e$ and $\ql \langle 1 \rangle$ (see \cite[Section 3.7]{LY}). Moreover, $\sL$ can be extended to $\extsL$ on $L_s$. Similar to the construction of negative sheaves, we define $\lpdelm$ to be $D^b_{(T \times L_s, \sL' \boxtimes \extsL ^{-1}), m}(\bsu)$. Let $\pi_s$ be the natural map from $\buu$ to $\bsu$. It can be regarded as a $\p^1$-fibration. The construction in \cite[Section 2.6]{LY} gives two adjoint pairs $(\pi_s^*, \pi_{s*})$ and $(\pi_{s*}, \pi_s^!)$, where $\pi_{s*}: \pha \lpdlm \ra \pha \lpdelm$ and $\pi_s^! \cong \pi_s^* \langle 2 \rangle : \pha \lpdelm \ra \pha \lpdlm$. 

\begin{lem} \label{proj}
Let $ \sL, \sL' \in \lio$ and $s$ be a simple reflection in $\extw$ such that $s \in \ewlo$. There is a canonical isomorphism of endo-functors
\[(-) \conn \IC(s)^+_\sL \cong \pi_s^* \pi_{s*}(-) \langle 1 \rangle \cong \pi_s^! \pi_{s*}(-) \langle -1 \rangle \in \End (\lpdlm).\]
As a result, $(-) \conn \IC(s)^+_\sL$ is right adjoint to itself. Similar statement also holds for positive sheaves. 
\end{lem}

\begin{proof}
See Lemma $3.8$ in \cite{LY}.
\end{proof}

Fix $w$ such that $ws <w$ and let $\wdo$ be a representative of $N_{\gk}(\bold{T_K})(\fq) / \bold{T_\co}(\fq)$. As in Section $3.10$ in \cite{LY}, we define $C(\wdo)^-_{\extsL} \in D^b_{(T \times L_s, \sL' \boxtimes \extsL ^{-1}), m}(\bsuw)$ such that its stalk at $\wdo$ is $\ql \langle-\ell(w) \rangle$. Then we can define $\Delta(\wdo)^-_{\extsL}$, $\IC(\wdo)^-_{\extsL}$, and $\nabla(\wdo)^-_{\extsL}$ similarly. The isomorphism classes of their images under $\omega$ are denoted by $\underline{\Delta}(w)_{\extsL}^-$, $\underline{\IC}(w)_{\extsL}^-$, and $\underline{\nabla}(w)_{\extsL}^-$ respectively.

\begin{lem} \label{pbic}
Let $s$ be a simple reflection in $\extw$ such that $s \in \ewlo$. Suppose $\ell(w) > \ell(ws)$, then $\pi_s^*\IC(\wdo)_{\extsL}^- \cong \IC(\wdo)^-_\sL.$
\end{lem}
\begin{proof}
See Lemma $3.10$ in \cite{LY}. Note that $\pi_s$ is a smooth $\p^1$-fibration, hence $\pi_s^*$ sends (simple) perverse sheaves to (simple) perverse sheaves up to a shift. 
\end{proof}

\begin{prop}\label{parity}
Let $w \in \extw$, $\sL \in \lio$ and $v \in \pha \wlewl$. Let $i_v$ be the embedding of $\biu \bsl \gkv / \biu$ into $\biu \bsl \gk / \biu$.
\begin{enumerate}[{$(1)$}]
\item The complexes $i_v^*\IC(\wdo)^+_\sL$ and $i_v^!\IC(\wdo)^+_\sL$ are pure of weight zero in $\wldvl$. 
\item $i_v^*\ICu(\wdo)^+_\sL$ and $i_v^!\ICu(\wdo)^+_\sL$ are isomorphic to direct sums of $\underline{C}(v)^+_\sL[n]$ for $n \equiv \ell(w)-\ell(v) \mod 2$. 
\end{enumerate}
\end{prop}

\begin{proof}
See Proposition $3.11$ in \cite{LY}. 
\end{proof}

For a block $\beta \in \pha_{\sL'} \ulextw_\sL$, the restriction of the Bruhat order of $\extw$ gives a partial order $\le$ in $\beta$. The following proposition is proved by Lusztig in \cite[Section 2]{L}. 

\begin{prop}\label{Lusztig}[Lusztig]
There is a unique minimal element $w^\beta$ in each block $\beta \in \pha_{\sL'} \ulextw_\sL$ and it is characterized by sending every positive affine coroot in $\widetilde{\Phi}^\vee_\sL$ to positive affine coroots in $\widetilde{\Phi}^\vee$. 
\end{prop}

\begin{lem}\label{minelt}
Let $\beta \in \pha _{\sL'} \ulextw_\sL$ and $\gamma \in \pha _{\sL''} \ulextw_{\sL'}$, then $w^\gamma w^\beta= w^{\gamma \beta}$.
\end{lem}

\begin{proof}
See Corollary $4.3$ in \cite{LY}. 
\end{proof}

\begin{lem}\label{conjmin}
	Let $\beta \in \pha_{\sL'} \ulextw_\sL$. Then $w \mapsto w^\beta w w^{\beta,-1}$ gives an isomorphism of Coxeter groups $\ewlo$ to $\ewlpo$. 
\end{lem}

\begin{proof}
	See Corollary $4.4$ in \cite{LY}.
\end{proof}

\begin{defn}
Let $\beta \in \pha _{\sL'} \ulextw_\sL$ and $w \in \beta$. There is a unique $v \in \ewlo$ such that $w= w^\beta v$. Define $\ell_\beta(w) := \ell_{\sL}(v)$, the length of $v$ in the Coxeter group $\ewlo$. Define the partial order $\le_\beta$ on the block $\beta$ as follows. Suppose $w' \in \beta$ and $v' \in \ewlo$ be its corresponding element as above, define $w \le_\beta w'$ if and only if $v \le_{\ewlo} v'$.
\end{defn}

\begin{rem}\label{wellconj}
	Thanks to Lemma \ref{conjmin}, we observe that the partial order $\le_{\beta}$ and the length function defined by right multiplication of $w^\beta$ coincide with those defined by left multiplication of $w^\beta$.
\end{rem}

\begin{prop}\label{order}
For $w,w' \in \beta$, if $w \le_\beta w'$, then  $w \le w'$.
\end{prop}

\begin{proof}
See Lemma $4.8$ in \cite{LY}. 
\end{proof}

\begin{rem}\label{orderdif}
	The converse of Proposition \ref{order} is false. The restriction of Bruhat order $\le_{\extw}$ to $\ewlo$ could be different from the Bruhat order $\le_{\ewlo}$ of $\ewlo$. For instance, this phenomenon is observed in Example \ref{example}.
\end{rem}

\begin{lem}\label{parityrel}
	Let $\beta \in \pha \lpewul$ and $x \in \beta$. We have $\ell_\beta(x) \equiv \ell(x)+ \ell(w^\beta) \mod 2$. 
\end{lem}

\begin{proof}
	By the definition of $\ell_\beta$, it suffices to check for the case when $\beta$ is the neutral block. The statement boils down to proving that the two length functions $\ell_{\sL}(x)$ and $\ell(x)$ have same parity for $x \in \ewlo$. When $x$ is a simple reflection in $\ewlo$, it acts as a (possibly affine) reflection on the root space. Note that the parity of $\ell(x)$ is determined by the effect of $x$ on the orientation of the root space. Since $x$ is a reflection, $\ell(x)$ has to be odd. For general $x \in \ewlo$, we choose a word expression of simple reflections in $\ewlo$ for $x$. The statement now follows from the fact that $\ell(w_1w_2) \equiv \ell(w_1)+ \ell(w_2) \mod 2$ for any $w_1,w_2 \in \extw$ and induction on the length of the word. 
\end{proof}

\begin{defn} \label{blockdef}
Let $\beta \in \pha _{\sL'} \ulextw_\sL$.

\begin{enumerate}[{$(1)$}]
\item  Denote $ \pha _{\sL'} \ulextd ^\beta _\sL$ to be the full triangulated subcategory generated by $\{ \underline{\Delta}(w)^+_\sL \}_{w \in \beta}$. It consists of sheaves whose cohomology is supported on $w \in \beta$. Denote $\lpdl^\beta \subset \pha\lpdl$ to be the preimage of $\lpudl^\beta$ under $\omega$. $\lpdl^\beta$ (resp. $\lpudl^\beta$) is called a block of $\lpdl$ (resp. $\lpudl$). 

\item If $\beta$ is the unit coset $\ewlo$, we say $\ldl^\beta$  (resp. $\ludl^\beta$) is the \textbf{neutral block} and denote it by $\ldlo$ (resp. $\ludlo$).  
\end{enumerate}
\end{defn}

\begin{prop}\label{blockdec}
There is a direct sum decomposition of the triangulated category \[\lpdl = \bigoplus_{\beta \in \pha \lpewul} \pha \lpdl^\beta.\]
\end{prop}

\begin{proof}
See Proposition $4.11$ in \cite{LY}. 
\end{proof}

\begin{lem}\label{vanstalk}
	Let $\beta \in \pha \lpewul$ and $w \in \beta$. Then $i_v^*\ICu(w)_{\sL}$ and $i_v^!\ICu(w)_{\sL}$ vanish unless $v \in \beta$ and $v \leq_{\beta} w$.
\end{lem}

\begin{proof}
	See Lemma $4.14$ in \cite{LY}. 
\end{proof}

The "filtration" mentioned in the next lemma is in the sense of triangulated categories.

\begin{lem} \label{filter}
	Let $\beta \in \pha \lpewul$ and $w \in \beta$. The IC sheaf $\ICu(w)_\sL$ is a successive extension of $\Delu(v)[n]$ for $v \in \beta$ and $n\in \Z$ such that $v \leq_{\beta} w$. Similarly, the standard sheaf $\Delu(w)$ is a successive extension of $\ICu(v)[n]$ for $v \in \beta$ and $n\in \Z$ such that $v \leq_{\beta} w$. 
\end{lem}	

\begin{proof}
	For the first statement, the Schubert stratification induces a "filtration" of $\ICu(w)_{\sL}$ whose subquotients are given by $i_{v!}i_v^*\ICu(w)_{\sL}$ for $v \le w$. Lemma \ref{vanstalk} shows that the subquotient $i_{v!}i_v^*\ICu(w)_{\sL}$ is zero unless $v \in \beta$ and $v \leq_{\beta} w$.
	
	We prove the second statement by induction on $\ell_{\beta}(w)$. From the first statement, we know that $\Delu(w)$ has a "filtration" whose subquotients are given by $\Delu(v)[-]$ and $\ICu(w)$ for $v <_{\beta} w$. By the induction hypothesis, $\Delu(v)[-]$ is a successive extension of $\ICu(u)[-]$ for $u \leq_{\beta} v$. By using the transitivity of $\leq_{\beta}$, we can conclude the proof of the second statement.
\end{proof}	

%\begin{lem}\label{filter2}
%	Let $w \in \beta$. Let $\sF$ be a perverse sheaf such that $i_v^*\sF$ vanishes unless $v \leq_{\beta} w$. Then $\sF$ is a successive extension of $\ICu(v)$ for $v \in \beta$ such that $v \leq_{\beta} w$. In particular, $\Delu(w)$ is a successive extension of $\ICu(v)$ for $v \in \beta$ such that $v \leq_{\beta} w$.
%\end{lem}
%
%\begin{proof}
%	We induct on the length of $\sF$ in the abelian category of perverse sheaves. When the length of $\sF$ is one, $\sF$ is an IC sheaf and the statement is trivial. 
%\end{proof}

\begin{prop}\label{minconv}
	Let $\beta \in \pha _{\sL'} \ulextw_\sL$ and $\dot{w}^{\beta}$ be a lifting of $w^{\beta}$. 
	\begin{enumerate}[{$(1)$}]
		\item The natural maps $\Delta(\dot{w}^{\beta})_{\sL}^+ \ra \IC(\dot{w}^{\beta})_{\sL}^+ \ra \nabla(\dot{w}^{\beta})_{\sL}^+$ are isomorphisms.
		
		\item  Let $\gamma \in \pha _{\sL''} \ulextw_{\sL'}$. Then the functor $(-) \conp \IC(\dot{w}^{\beta})_{\sL}^+: \pha \lppdlp^\gamma \ra \pha \lppdl^{\gamma \beta}$ is an equivalence and the inverse is given by  $(-) \conp \IC(\dot{w}^{\beta,-1})_{\sL'}^+$. The analogous statement is true for left convolution with $\IC(\dot{w}^{\beta})_{\sL}^+$.
		
		\item  The functor $(-) \conn \IC(\dot{w}^{\beta})_{\sL}^+: \pha \lppdlpm \ra  \pha \lppdlm$ is an equivalence and the inverse is given by  $(-) \conn \IC(\dot{w}^{\beta,-1})_{\sL'}^+$. 
		
		\item The equivalence $(-) \conp \IC(\dot{w}^{\beta})_{\sL}^+$ sends $\Delta (\wdo)^+_{\sL'}, \IC (\wdo)^+_{\sL'}$, $\nabla (\wdo)^+_{\sL'}$ to the isomorphism classes $\Delta (\wdo\dot{w}^{\beta})^+_{\sL}, \IC (\wdo\dot{w}^{\beta})^+_{\sL}, \nabla (\wdo\dot{w}^{\beta})^+_{\sL}$ respectively. The analogous statement is true for left convolution with $\IC(\dot{w}^{\beta})_{\sL}^+$.
		
		\item The equivalence $(-) \conn \IC(\dot{w}^{\beta})_{\sL}^+$ sends  $\Delta (\wdo)^-_{\sL'}, \IC (\wdo)^-_{\sL'}$, $\nabla (\wdo)^-_{\sL'}$ to the isomorphism classes $\Delta (\wdo\dot{w}^{\beta})^-_{\sL}, \IC (\wdo\dot{w}^{\beta})^-_{\sL}, \nabla (\wdo\dot{w}^{\beta})^-_{\sL}$ respectively. 

	\end{enumerate}
\end{prop}

\begin{proof}
	The statements are obvious when $\ell(w^\beta)=0$. By Lemma \ref{stdstd}, we can assume $w^{\beta}$ is in the affine Weyl group. For the rest of the proof, see Lemma \ref{equiv} and Proposition 5.2 in \cite{LY}.
\end{proof}

\begin{defn}

For any $\beta \in \pha _{\sL'} \ulextw_\sL$, we call $\sF \in \pha \lpdlm$ (resp. $\lpdl$) a \textbf{maximal} (resp. \textbf{minimal}) \textbf{IC sheaf} if $\omega \sF$ is isomorphic to $\underline{\IC}(w^\beta)_{\sL}^-$ (resp. $\underline{\IC}(w^\beta)_{\sL}^+$).

A \textbf{rigidified (neutral) maximal IC sheaf} is a pair $(\ltlm, \epsilon_\sL)$ where $\ltlm \in \pha \ldlm$ is a maximal IC sheaf and $\epsilon_\sL: \pha \ltlm \ra \delta_\sL$ is a non-zero morphism (see the discussion below Claim \ref{welldef}).
\end{defn}

Here, both $\ltlm$ and $\delta_\sL$ are viewed as projective systems of complexes on $\gk / \biu$. Rigidified maximal sheaves clearly exist and are unique up to unique isomorphism. 

From this point onwards, we fix a rigidified maximal IC sheaf for each neutral block and denote it by $\ltlm$. By abuse of notation, we will use $\lptlm$ to denote a maximal sheaf such that $\omega \pha \lptlm$ is isomorphic to $\underline{\IC}(w^\beta)_{\sL}^-$.

%The proof of the next two propositions is very similar to the reductive group analog, except for using induction instead of backward induction on length in Lemma \ref{stalk}. 

\begin{prop}
Let $\beta \in \pha \lpewul$ and $\gamma \in \pha \lppewulp$. Let $w \in \beta$.
\begin{enumerate}[{$(1)$}]
\item The convolution $\ICu(w^\gamma)^-_{\sL'} \conn \ICu(w)^+_\sL$ is isomorphic to a direct sum of shifts of $\ICu(w^{\gamma \beta})^-_\sL$.
\item The perverse cohomology $^pH^i(\ICu(w^\gamma)^-_{\sL'} \conn \ICu(w)^+_\sL)$ vanishes unless $-\ell_\beta(w) \le i \le \ell_\beta(w)$.
\item The perverse cohomology of $\ICu(w^\gamma)^-_{\sL'} \conn \ICu(w)^+_\sL$ in extremal degrees are isomorphic to $\ICu(w^{\gamma \beta})^-_\sL$. That is \[^pH^{ -\ell_\beta(w)}(\ICu(w^\gamma)^-_{\sL'} \conn \ICu(w)^+_\sL)\cong \ICu(w^{\gamma \beta})^-_\sL \cong \pha ^pH^{\ell_\beta(w)}(\ICu(w^\gamma)^-_{\sL'} \conn \ICu(w)^+_\sL) .\]
\end{enumerate}
\end{prop}

\begin{proof}
We prove the statements simultaneously by induction on $\ell(w)$. If $\ell(w)=0$, the statements follow from Proposition \ref{minconv} as $w$ is a minimal element. If $\ell(w)=1$, there exists a length-zero element $x \in \extw$ such that $w=xs$ is a simple reflection in $\extw$. By Proposition \ref{minconv}, we know that $\ICu(w)^+_\sL$ is isomorphic to $\ICu(x)^+_{s\sL} \star \ICu(s)^+_\sL$. The length zero case implies that we may assume $w=s$ is a simple reflection in $\extw$.  

If $s \notin \ewlo$, then $\ell_\beta(s)=0$ and $w^\gamma s= w^\beta$ by Lemma \ref{minelt}. The statements follow from Lemma \ref{equiv}.

If $s \in \ewlo$, by combining Lemma \ref{proj}, Lemma \ref{pbic} and the projection formula, we get \[\ICu(w^\gamma)^-_{\sL} \conn \ICu(s)^+_\sL \cong \ICu(w^\gamma)^-_{\sL}[1] \oplus \ICu(w^\gamma)^-_{\sL}[-1],\] which implies the above statements. 

For $\ell(w) >1$, write $w=w's$ for some simple reflection $s$ such that $\ell(w)=\ell(w')+1$. The same argument in the proof of Lemma $6.3$ in \cite{LY} (or Lemma \ref{stdicper} below) shows $\ICu(w')^+_{s\sL} \conp \ICu(s)^+_\sL$ is perverse. The decomposition theorem implies $\ICu(w)^+_\sL$ is a direct summand of $\ICu(w')^+_{s\sL} \conp \ICu(s)^+_\sL$. Hence it suffices to show the analogous statements hold for $\ICu(w')^+_{s\sL} \conp \ICu(s)^+_\sL$
and they follow from induction. The key observation is that we have the equality $\ell_{\beta}(w)=\ell_{\beta s }(w')$ when $s \notin \ewlo$, and $\ell_{\beta}(w)=\ell_{\beta}(w')+1$ otherwise. Here $\beta s$ refers to the block that  contains $w'$.
\end{proof}

\begin{prop}\label{stalk}
Let $\beta \in \pha \lpewul$ and $w \in \beta$. For $j_w: \buuw \inj \buu$, we have $j_w^*\ICu(w^\beta)_\sL^- \cong \underline{C}(w)_\sL^-[\ell_\beta(w)]$ and $j_w^!\ICu(w^\beta)_\sL^- \cong \underline{C}(w)_\sL^-[-\ell_\beta(w)]$. 
\end{prop}

\begin{proof}
The second statement can be derived by applying Verdier duality to the first statement.

We prove the first statement for all $\sL$ in $\lio$ by induction on $\ell(w)$. If $\ell(w)=0$, the statement follows from the fact that $w=w^\beta$. Suppose the statement holds for any $w$ with $\ell(w) <n$, we are going to show it also holds for elements of length $n$. Let $w \in \extw$ of length $n$ and $s$ be a simple reflection such that $\ell(w)=\ell(ws)+1$. 

Lemma \ref{minelt} implies that if $s \notin \ewlo$, then $w^\beta s$ is also a minimal element in a block. By Lemma \ref{equiv}, we establish $\ICu(w^\beta)^-_{\sL'} \conn \ICu(s)^+_\sL \cong \ICu(w^\beta s)^-_{\sL}$ and $\underline{\nabla}(w)^-_{\sL'} \conn \ICu(s)^+_\sL \cong \underline{\nabla}(w s)^-_{\sL}$. The same lemma demonstrates that $(-) \conn \ICu(s)^+_\sL$ is an equivalence. Since the stalk can be computed from $\Hom(\ICu(w^\beta)^-_{\sL'},\underline{\nabla}(w)^-_{\sL'})$, the statement follows from the induction hypothesis for $ws$. 

If $s \in \ewlo$, Lemma $\ref{pbic}$ implies the stalks of $\ICu(w^\beta)_\sL^-$ at $\wdo$ and $\wdo \sdo$ are isomorphic. The latter one is known by the induction hypothesis. The fact that $\ell_\beta(w)=\ell_\beta(ws)+1$ finishes the proof. 
\end{proof}

For restrictions of maximal IC sheaves on other strata, we have the following proposition. 

\begin{prop}\label{stalk2}
Let $\beta \in \pha \lpewul$ and $w \notin \beta$. For $j_{w}: \buuw \inj \buu$, we have $j_{w}^*\ICu(w^\beta)_\sL^- \cong j_{w}^!\ICu(w^\beta)_\sL^- \cong 0$. 
\end{prop}

\begin{proof}
The argument in \ref{stalk} applies. 
\end{proof}

\section{Maximal IC sheaves}\label{sect6}

In this section we introduce the averaging functor of maximal IC sheaf $\Av_! \pha \lptlm$ and compute its stalks. 
%The computation allows us to define a coalgebra structure on $\Av_! \pha \lptlm$.

\subsection{Averaging}\label{sectav}

The goal in this subsection is to define an averaging functor, which is similar to that in \cite[Section 4.4]{BY}. %, from $ {\ulad}^b_{(T \times T, \sL \boxtimes \sL'),m}(\gk / \biu)$ to ${\ulad}^b_{(T \times T, \sL \boxtimes \sL'),m}(\biu \bsl \gk / \biu)$. 

%In order to average the maximal IC sheaf (see section ??), we also need the notion of 

Let $J_w$ be a normal subgroup of $\bi$ chosen in Section \ref{positive} and $m_{J_w}: \biu / J_w \times \gksw / \biu \ra \gksw / \biu$ be the left multiplication map. It is a morphism between finite-dimensional varieties. We define the functor $\Av_{J_w,\le w,!}:D^b_m(\gksw / \biu) \ra D^b_m(\gksw / \biu)$ by sending $\sF$ to $m_{J_w,!}(\ql \boxtimes \sF) \langle 2 \dim(\biu/J_w) \rangle$. 

\begin{claim}
	The functor $\Av_{J_w,\le w,!}$ is independent of the choice of $J_w$. Hence we write $\Av_{\biu,\le w,!}$ for the functor. 
\end{claim}

\begin{proof}
Let $J_w'$ be another choice. Without loss of generality, we may assume $J_w' \subset J_w$. We have the following commutative diagram. 
\[
\begin{tikzcd}[row sep=large]
	\biu /J_w' \times  \gksw / \biu \arrow[d, "p \times id"] \arrow[r, "m_{J_w'}"] & \gksw / \biu  \\
	\biu  /J_w \times  \gksw / \biu  \arrow[ru, "m_{J_w}"'] & 
\end{tikzcd}\]
The claim follows from the fact that $p \times id$ is an affine space fibration of dimension $2 \dim(J_w/ J_w')$.
\end{proof}

\begin{claim}\label{avtrans}
The functor $\Av_{\biu,\le w,!}$ commutes with $*$-pullback. In other words, $i_{w,w'}^* \circ \Av_{\biu,\le w',!}$ is canonically isomorphic to $\Av_{\biu,\le w,!} \circ i_{w,w'}^*$ for $w \le w'$, where $i_{w,w'}:\gksw / \biu \ra \gkswp / \biu$ is the closed embedding. Moreover, the isomorphism is compatible with $w \le w' \le w''$. 
\end{claim}

\begin{proof}
For the first statement, we choose $J_w$ to be the same as $J_{w'}$ without loss of generality. Then we apply the proper base change theorem to the following Cartesian diagram,    
%for any $w \le w'$ and $\sF \in D^b_m(\gkswp / \biu)$, $i^*(\Av_{\biu,\le w',!}(\sF))$ is canonically isomorphic to $\Av_{\biu,\le w,!}(i^*(\sF))$ and 
\[
\begin{tikzcd}[row sep=large]
	\biu /J_{w'} \times  \gksw / \biu \arrow[d, "id \times i_{w,w'}"] \arrow[r, "m_{J_{w'}}"] & \gksw / \biu \arrow[d, "i_{w,w'}"]\\
	\biu / J_{w'}  \times \gkswp / \biu \arrow[r, "m_{J_{w'}}"]& \gkswp / \biu.
\end{tikzcd}\]
The second statement can be proved similarly.
\end{proof}

%As a result, ${\ulalim}\Av_{\biu, \le w!}: {\ulad}^b_{(T \times T, \sL \boxtimes \sL'),m}(\gk / \biu) \ra {\ulad}^b_{(T \times T, \sL \boxtimes \sL'),m}( \gk / \biu)$ is well defined. Here 
%\[ {\ulad}^b_{(T \times T, \sL \boxtimes \sL'),m}(\gk / \biu):=2-\ulalim_{\\w\in\extw}D^b_{(T \times T, \sL \boxtimes \sL'),m}(\gksw / \biu)\]
%is a projective system of triangulated categories with pullback functors $i_{w,w'}^*$. 

We recall the definition of $2-$limit of $D^b_{(T \times T, \sL' \boxtimes \sL^{-1}),m}(\gksw / \biu)$ before defining the averaging functor.

%2-\ulalim_{\\w\in\extw}

\begin{defn}\label{twolimit}
Denote the category ${\ulad}^b_{(T \times T, \sL' \boxtimes \sL^{-1}),m}(\gk / \biu)$ as the limit of the projective system of triangulated categories $D^b_{(T \times T, \sL' \boxtimes \sL^{-1}),m}(\gksw / \biu)$ with pullback functors $i_{w,w'}^*$. Concretely, the objects in this category are $(\sF_w, \chi_{w,w'})$, where $\sF_w$ is a complex of sheaves on $\gksw / \biu$ with the prescribed monodromy and $\chi_{w,w'}$ is an isomorphism from $\sF_w$ to $i_{w,w'}^*\sF_{w'}$ such that $i_{w,w'}^*\chi_{w',w''} \circ \chi_{w,w'}= \chi_{w,w''}$ whenever $w \le w' \le w''$. The morphism between $(\sF_w, \chi_{w,w'})$ and $(\sG_w, \psi_{w,w'})$ is a family of morphisms $\phi_w: \sF_w \ra \sG_w$ such that $\psi_{w,w'} \circ \phi_w = \phi_{w'} \circ \chi_{w,w'}$ for any $w \le w'$.
%We will denote this category by ${\ulad}^b_{(T \times T, \sL' \boxtimes \sL^{-1}),m}(\gk / \biu)$.

Similarly, we denote $\lpdul$ as the limit of the projective system of triangulated categories $D^b_{(T \times T, \sL' \boxtimes \sL^{-1}),m}(\biu \bsl \gksw / \biu)$ with $*$-pullback functors. Its objects and morphisms closely resemble those in ${\ulad}^b_{(T \times T, \sL' \boxtimes \sL^{-1}),m}(\gk / \biu)$, with the distinction that all sheaves and morphisms are required to be $\biu$-equivariant. 
\end{defn}

By utilizing Claim \ref{avtrans}, we are able to define the limit of functors ${\ulalim}\Av_{\biu, \le w!}: {\ulad}^b_{(T \times T, \sL' \boxtimes \sL^{-1}),m}(\gk / \biu)$ $\ra {\ulad}^b_{(T \times T, \sL' \boxtimes \sL^{-1}),m}( \gk / \biu)$. Any object in the image of ${\ulalim}\Av_{\biu, \le w!}$ has a natural $\biu$-equivariant structure. Therefore it makes sense to define the averaging functor as follows. 

\begin{defn}
Define $\Av'_!: {\ulad}^b_{(T \times T, \sL' \boxtimes \sL^{-1}),m}(\gk / \biu) \ra \pha \lpdul$ by remembering the $\biu$-equivariant structure. Explicitly, it sends $(\sF_w)_{w \in \extw}$ to $(m_{J_w,!}(\ql \boxtimes \sF_w) \langle 2 \dim(\biu/J_w) \rangle)_{w \in \extw}$. %Here \[ \lpdul:=2-\ulalim_{\\w\in\extw}D^b_{(T \times T, \sL' \boxtimes \sL^{-1}),m}(\biu \bsl \gksw / \biu),\] which can be defined similarly by using construction in Definition \ref{twolimit}.
\end{defn}

\begin{claim}\label{welldef}
For $w \le w' \in \extw$, the image of the composition map $\gksw / \biu \ra \gk / \biu \ra \buu$ is contained in $\buuswp$. 
\end{claim}

\begin{proof}
It reduces to the case when $w=w'$. After multiplying by a lift of a length-zero element, we can assume that $w$ is in the affine Weyl group. Let $w=s_{i_1}s_{i_2}\dots s_{i_n}$ be a reduced word expression. For each point $x \in \gkw$, we can write $x= x_{i_1}x_{i_2}\dots x_{i_n}$, where $x_i \in \gksi$. Suppose the image of $x$ is in the stratum $\buuwpp$ for some $w'' \in \extw$, we know that the image of $x x_{i_n}^{-1}$ is in either $\buuwpp$ or $\buuwppsin$ by the proof of Lemma \ref{stdstd}. Hence, we know the image of $e$ is in one of the strata $\buuwppy$ for some $y \le w$. Therefore $w'' \le w$. 
\end{proof}

Let $i_w^{w'}: \gksw / \biu \ra \buuswp$ be its restriction map. Since $i_w^{w'}$ is a map between two finite type stacks, it makes sense to $*$-pullback sheaves on $\buuswp$ to $\gksw / \biu$ via $i_w^{w'}$. For $\sF=(\sF_{w'}) \in \pha \lpdlm$, we first $*$-pullback $\sF_{w'}$ to a sheaf on $\gksw / \biu$ whenever $w \le w'$. When $w$ runs through $\extw$, we obtain a compatible system of sheaves on $\gk/ \biu$ and hence an object in the category ${\ulad}^b_{(T \times T, \sL' \boxtimes \sL^{-1}),m}(\gk / \biu)$. Therefore we obtain a functor $i^*$ from $\lpdlm$ to ${\ulad}^b_{(T \times T, \sL' \boxtimes \sL^{-1}),m}(\gk / \biu)$.

\begin{defn}
	We define the averaging functor $\Av''_!: \pha  \lpdlm \ra \pha \lpdul$ to be the composite $\Av'_! \circ i^*$. 
\end{defn}

By abuse of notation, we will use the same symbol $\Av_!$ to represent both averaging functors $\Av'_!$ and $\Av''_!$.

\subsection{Basic properties of the averaging functor} We write down some properties of the averaging functor $\Av_!$.

From the construction in Definition $\ref{twolimit}$, we know that $\lpdl$ can be viewed as a subcategory of $\lpdul$. The inclusion functor can be described explicitly. For any $\sF \in \pha \lpdl$ supported on $\gkswpp / \biu$, the image of $\sF$ is $(\sF_w, \chi_{w,w'})$ where $\sF_w$ is $i_{w,w'*} \sF$ for any $w \ge w''$. We reformulate this as the following lemma.   %Therefore we can average objects in $\lpdl$. 

\begin{lem}
Let $\sG \in \pha \lpdl$. $\Av_!\sG$ is canonically isomorphic to $\sG$ by viewing $\lpdl$ as a subcategory of $\lpdul$. 
\end{lem}

\begin{lem}\label{avcom}
$\Av_!$ commutes with the convolution product. That is, $\Av_!(\sF \conp \sG)$ is canonically isomorphic to $\Av_!\sF \conp \sG$ for $\sF\in \pha \lppdlp$, $\sG \in \pha \lpdl$. 
\end{lem} 

\begin{proof}
This follows from the base change theorem.
\end{proof}

Let $\sF = (\sF_w)_{w \in \extw} \in \pha \lppdulp$ and $\sG \in \pha \lpdl$. We define the convolution product $\sF \conp \sG \in \pha \lppdul$ as follows. It suffices to construct $i_{\le w}^* (\sF \conp \sG) \in \pha \lppdul$ for any $w \in \extw$. Suppose $\sG \in \pha_{\sL'}D(\le w'')_{\sL}$, we choose $w' \in M(w, w''^{-1})$. The complex $i_{\le w}^*(\sF_{w'} \conp \sG)$ is independent of the choice of $w'$. We define $\sF \conp \sG$ to be the projective system $i_{\le w}^*(\sF_{w'} \conp \sG)$ for $w \in \extw$. 

\begin{cor}\label{avass}
Let $\sF  \in \pha \lppdlpm$ and $\sG \in \pha \lpdl$. Then $\Av_!(\sF \conn \sG)$ is canonically isomorphic to $(\Av_!\sF) \conp \sG$. In particular, $\Av_!(\lptlpm \conn \sG)$ is canonically isomorphic to $\avlptlpm \conp \sG$ for any $\sG \in \pha \lpdl$.
\end{cor}

\begin{proof}
Let $\sF = (\sF_w)_{w \in \extw} \in \pha \lppdlpm$ and $\sG \in \pha_{\sL'}D(\le w'')_{\sL}$. It suffices to prove that $i_{\le w}^* \Av_!(\sF \conn \sG)$ is canonically isomorphic to $i_{\le w}^* ((\Av_!\sF) \conp \sG)$ for each $w \in \extw$. Let $w' \in M(w, w''^{-1})$. We use the notation of Section \ref{sectav}. 

We have $i_{\le w}^* \Av_!(\sF \conn \sG)= i_{\le w}^* \Av'_!i^*(\sF \conn \sG)= i_{\le w}^* \Av'_!i_w^{w*}(\sF \conn \sG)$ by the definition of $\Av_!$. The choice of $w'$ implies that $i_w^{w*}(\sF \conn \sG)=i_w^{w*}(\sF_{w'} \conn \sG)$. Hence, $i_{\le w}^* \Av_!(\sF \conn \sG)=i_{\le w}^* \Av'_!i_w^{w*}(\sF_{w'} \conn \sG)$.

On the other hand, we have $i_{\le w}^* ((\Av_!\sF) \conp \sG)=i_{\le w}^* ((\Av'_! i_{w'}^{w'*}\sF) \conp \sG)$. The latter one is canonically isomorphic to $i_{\le w}^* (\Av'_! (i_{w'}^{w'*}\sF \conp \sG))$ by Lemma \ref{avcom}.

We claim that the functor $i^*$ commutes with the Hecke modification. In short, $i^*(\sF \conn \sG) = i^*\sF \conp \sG$. As a result, $i_{\le w}^*(i_{w'}^{w'*}(\sF_{w'}) \conp \sG) = i_w^{w*}(\sF_{w'} \conn \sG)$. Hence, $i_{\le w}^* \Av_!(\sF \conn \sG)$ is canonically isomorphic to $i_{\le w}^* ((\Av_!\sF) \conp \sG)$. 

Let us prove the claim. Let $\mu$ be the multiplication map from $\gk \times^\bi \gkswpp$ to $\gk$. Let $S$ be the preimage $\mu^{-1}(\gksw)$. The following commutative diagram is Cartesian, 

\[
\begin{tikzcd}[row sep=large]
	S / \biu \arrow[d, " i \widetilde{\times} id"] \arrow[r, "\mu"] & \gksw / \biu \arrow[d, "i"]\\
	 \hecbswpp \arrow[r, "pr_2"]& \buu,
\end{tikzcd}\]
where $ i \widetilde{\times} id$ is the restriction of the natural map from $\gk \times^\bi \gkswpp/ \biu$ to $\bium \bsl \gk \times^\bi \gkswpp/ \biu=\hecbswpp$. The diagram is commutative because both convolution and Hecke modification are induced by the multiplication map of $\gk$. Since both $\mu$ and $pr_2$ can be regarded as twisted products with $\gkswpp /\bi$, the diagram is Cartesian.

By applying the base change theorem to the Cartesian diagram above, we know that the functor $i^*$ commutes with the Hecke modification. 
\end{proof}

\begin{claim}\label{avuni}
Let $act:U \times X \ra X$ be an action of a smooth connected unipotent group $U$ on a variety $X$. Then the functor $act_!(\ql \boxtimes -) \langle 2 \dim U \rangle :D^b_m(X) \ra D^b_m(X/U)$ is left adjoint to the forgetful functor. 
\end{claim}	

\begin{proof}
It follows from the fact that $act^!(\sG) \langle -2 \dim U \rangle \cong act^*(\sG)  \cong \ql \boxtimes \sG$ for any $U$-equivariant complex $\sG$.
\end{proof}

\begin{cor}\label{Avfor}
$\Av_!$ is left adjoint to forgetful functor. In particular, we have $\Hom_{\biu \bsl \gk/ \biu}(\avlptlm, \sG)$ $\cong \Hom_{\gk/ \biu}(\pha \lptlm, \sG)$  for $\sG \in \pha \lpdl$.
\end{cor} 

%hat is, for any $\sF \in \pha  \lpdlm$ and $\sG \in \pha \lpdl$, we have a canonical isomorphism $\Hom_{\biu \bsl \gk/ \biu}(\Av_! \sF, \sG) \cong \Hom_{\gk/ \biu}(\sF, \sG)$.

\begin{proof}
	
Let $\sF \in {\ulad}^b_{(T \times T, \sL' \boxtimes \sL^{-1}),m}(\gk / \biu)$. Suppose $\sG$ is supported on $\gksw/ \biu$, we have 
\begin{align*}
	& \Hom_{\biu \bsl \gk/ \biu}(\Av_! \sF, \sG) &&\\
	\cong & \Hom_{\biu \bsl \gk/ \biu}(i_{\le w}^* \Av_! \sF, \sG) && \text{support assumption on } \sG \\
	=& \Hom_{\biu \bsl \gk/ \biu}(i_{\le w}^* \Av_! (i_w^{w*}\sF_w), \sG)  && \text{by construction of} \Av_!\\ 
	=&\Hom_{\biu \bsl \gk/ \biu}(\Av_! (i_w^{w*}\sF_w), \sG)  && \text{adjoint functors} 
\end{align*}
Similarly, we have $\Hom_{\gk/ \biu}(\sF, \sG) \cong \Hom_{\gk/ \biu}(i_w^{w*}\sF_w, \sG)$. Apply Claim \ref{avuni} to the case when $U=\biu/ J_w$ and $X= \gksw/ \biu$. 
\end{proof}

\begin{cor}
Recall that $\delta_\sL$ is the monoidal unit in $\ldl$. 
There is a canonical map $\epsilon_\sL: \avltlm \ra \delta_\sL$ which induces the identity map at stalk $e$. 
\end{cor}

%First, notice that $\gkswm / \biu$ contains $\gksw / \biu$. It can be shown by induction and convolution. Faltings?? Kashiwara??

%For any partially ordered set $I$ and category $\sC$, we could view $I$ as a category whose objects are elements in $I$ and there is a unique morphism (resp. no morphisms) from $i'$ to $i$ if $i \le i'$ (resp. otherwise). Then we can define the category ${\ulalim}_I\sC$ as the category of functors from $I$ to $\sC$. % by viewing $I$ as a category whose morphisms whose objects are $(c_i \in \sC, \phi_{i,i'}:c_i \ra c_{i'}$ for any $i \le i')$ and morphisms from 

%In our situation, we take $I = (\extw , \le)$ and $\sC= D^b_{(T \times T, \sL \boxtimes \sL'),m}(\bium \bsl \gk / \biu) $. Denote the resulting category to be ${\ulad}^b_{(T \times T, \sL \boxtimes \sL'),m}(\bium \bsl \gk / \biu) $ and its objects are called \textbf{pro-negative sheaves}.

%Suppose now we have $\sG'=(\sG'_w)_{w \in \extw} \in {\ulad}^b_{(T \times T, \sL' \boxtimes \sL''),m}(\biu \bsl \gk / \biu)$,  $\sF \star_- \sG'_w$ is a negative sheaf for every $w \in \extw$. Morphism from $\sG'_w$ to $\sG'_{w'}$ induces a morphism from $\sF \star_- \sG'_w$ to $\sF \star_- \sG'_{w'}$. Hence $(\sF \star_- \sG'_w)_{w \in \extw}$ form a projective system and is a pro-negative sheaf. We denote it as $\sF \star_- \sG' \in {\ulad}^b_{(T \times T, \sL \boxtimes \sL''),m}(\bium \bsl \gk / \biu)$.

\subsection{Hyperbolic localization}

The stalks of $\avltlm$ can be computed via hyperbolic localization. To do so, we identify $\buu$ with $[\bium \bsl \gk / \biu]$ as in Section \ref{modsec} and make use of the loop rotation on $\gk$.

\subsubsection{Construction of the $\gm$-action}

We note that $\gm$ acts on $\gk$ by the loop rotation and we denote this action by $\rot$. Explicitly, when $a \in \gm$ and $x \in \gk$, $\rot(a)(x)$ is replacing the uniformizer $t$ in $x$ with $at$. 

%To apply the result in \cite{Br}, $\gk / \biu$ has to be truncated to a normal finite-dimensional variety.  

For any positive integer $n$, let $J^n$ be the unipotent group whose $A$-points are $ker(G(A[t^{-1}]) \ra G(A[t^{-1}]/(t^{-n})))$. As in Section \ref{modsec}, we can define $\bujn$ the moduli stack of $G$-bundles on $\p^1$ with a $J^n$-level structure at infinity and a $U$-reduction at zero. There is a natural map from $\bujn$ to $\bbb$. Let $\bujnw$ (resp. $\bujnsw$) be the preimage of $\bbbw$ (resp. $\bbbsw$). Being a bundle over a smooth and locally of finite type stack $\buu$, the stack $\bujn$ is also smooth and locally of finite type. It can be identified with $[J^n \bsl \gk / \biu]$ (c.f. \cite[Theorem 7]{F} or \cite[Section 2.4]{Y2}).

Fix $w \in \extw$ and $\wdo$ is a representative in the coset $N_{\gk}(\bold{T_K})(\fq) / \bold{T_\co}(\fq)$ corresponding to $w$ throughout this subsection. By \cite[Section 3]{F}, there exists a positive integer $n$ such that $\bujnsw$ is a variety. Fix such an $n$ and denote $J^n$ by $J_\infty$.  We construct a $\gm$-action $\phi_w$ on $\gk$ by declaring $\phi_w(a)(x):=\chi^\vee(a)\rot(a^{2h+1})(x)\lambda(a)$, where $\lambda(a)$ is defined to be $ (\rot(a^{2h+1})(\wdo))^{-1} \chi^\vee (a^{-1}) \wdo \in T$, $\chi^\vee$ is the sum of positive coroots and $h$ is the Coxeter number of $G$. Since both left and right $T$-action leave $J_\infty$, $\biu$, and $\gksw$ invariant, $\phi_w$ induces an action on $\bujsw$. Similarly, $\phi_w$ induces an action on $\bbjsw$, which we denote by $\phi_w'$. 

\begin{lem}\label{bunsm}
	Both $\bujsw$ and $\bbjsw$ are smooth varieties. In particular, they are normal.
\end{lem}

\begin{proof}
	As mentioned in Section \ref{modsec}, we know that $\buu$ is a smooth stack. Therefore the open substack $\buusw$ is also smooth. Since $\bujsw$ is an affine space bundle over $\buusw$, we know that $\bujsw$ is smooth. We can prove the statement for $\bbjsw$ similarly. 
\end{proof}

\subsubsection{Attracting set and repelling set}
It is clear that $\wdo T=[\jinf \bsl \jinf \wdo \bi / \biu]$ is fixed by $\phi_w$. Let $F$ be the connected component of fixed points of $\phi_w$ containing $\wdo T$. We want to find the attracting set $X_w^+:= \{x \in \bujsw: \lim_{a \ra 0} \phi_w(a)(x) \in F\}$ and the repelling set $X_w^-:= \{x \in \bujsw: \lim_{a \ra \infty} \phi_w(a)(x) \in F\}$. Similarly, we define $F'$ as the connected component of fixed points of $\phi_w'$ containing $w=[\jinf \bsl \jinf w \bi / \bi]$. We denote $X_w^{+'}$ as its attracting set and $X_w^{-'}$ as its repelling set.

\begin{lem}\label{hyperset2}
	\begin{enumerate}
	\item The fixed points of $\phi_w'$ are discrete. Moreover, they are indexed by $\{v: v \le w\}$.
	\item $X_w^{-'} = \bbjw$;
	\item $X_w^{+'} =[\jinf \bsl \jinf \gkw / \bi]=\gkw / \bi$.	
\end{enumerate}
\end{lem}

\begin{proof}
	Since $\chi^\vee(a) \lambda(a)$ belongs to $T$,  we can rewrite $\phi'_w(a)(x)$ as  $\Ad(\chi^\vee(a))(\rot(a^{2h+1})(x))$ for $x \in \bbjsw$. It follows that for any $v \le w$, the point $v=[\jinf \bsl \jinf v \bi / \bi]$ is fixed by $\phi_w'$. As $a$ tends to zero, $\Ad(\chi^\vee(a)) \circ  \rot(a^{2h+1})$ contracts $\biu$ to $e$. Therefore, $\phi'_w$ contracts $\bbjv$ to $v$ as $a \ra 0$. As $a^{-1}$ tends to zero, $\Ad(\chi^\vee(a))\circ \rot(a^{2h+1})$ contracts $\bium$ to $e$. Hence, $\phi'_w$ contracts $[\jinf \bsl \jinf \gkv / \bi]$ to $v$ as $a \ra \infty$. The statements now follow from Bruhat decomposition and Birkhoff decomposition. 
\end{proof}

% Notice that $[\jinf \bsl \jinf \gkw / \biu]$ can be naturally identified as $\gkw / \biu$, a $T$-torsor over $\biu \wdo \biu / \biu$. 

%By a dimensional counting argument, we know that the tangent spaces of $\biu \wdo \biu / \biu$ and $\bujw$ at $\wdo$ are linearly disjoint spanning subspaces of the tangent space of $\bujsw$ at $\wdo$. Since the two subsets are $\phi_w$-invariant, we know that 

We have the following results regarding $F$ and its corresponding attracting and repelling sets.

\begin{prop}\label{hyperset}
	\begin{enumerate}
		\item $\wdo T$ is a connected component of fixed points of $\phi_w$, that is $F=\wdo T$;
		\item $X_w^- = \bujw$;
		\item $X_w^+ =[\jinf \bsl \jinf \gkw / \biu]=\gkw / \biu$.	
	\end{enumerate}
\end{prop}

\begin{proof}
	Let $\pi$ be the projection from $\bujsw$ to $\bbjsw$. If a point $y \in \bujsw$ is in $F$ (resp. $X_w^{+}$, $X_w^{-}$), then $\pi(y)$ is in $F'$ (resp. $X_w^{+'}$, $X_w^{-'}$). The proof of Lemma \ref{hyperset2} shows that $X_w^+$ contains $[\jinf \bsl \jinf \gkw / \biu]$ and $X_w^-$ contains $\bujw$. The statements now follow from their analogs in Lemma \ref{hyperset2}.
\end{proof}

\subsubsection{Computation of stalks} We use hyperbolic localization to compute stalks of the averaging IC sheaf in this subsection. 

\begin{lem}\label{prodbase}
	Let $X$ be a variety and $x_0$ be a point in $X$. Let $A$ be a torus and $\sL$ be a rank one character sheaf on $A$. Let $e$ be the identity element of $A$. Suppose $\sF$ is a complex of sheaves on $A \times X$ such that it is $\sL$-equivariant with respect to the action of $A$ on the first coordinate, then the product of restriction maps of $\sF$ on $A \times x_0$ and $e \times X$ gives an isomorphism from $\sF$ to $\sF|_{A \times x_0} \boxtimes \sF|_{e \times X} = \sL \boxtimes \sF|_{e \times X} $.
\end{lem}

\begin{proof}%[Proof of Claim \ref{prodbase}]
	Let $act, pr: A \times (A \times X) \ra A \times X$ be the action and projection map respectively. From the assumption, we know that $a^*\sF \cong \sL \boxtimes pr^*\sF$ on $A \times A \times X$. The lemma follows from restricting this isomorphism to $A \times e \times X$. 
\end{proof} 

The following lemma is well-known and we will skip the proof.
\begin{lem}[Homotopy lemma] \label{homolem}
Let $act:\gm \times \A^n \ra \A^n$ be the standard scaling action. Let $\sF$ be a $\gm$-equivariant complex of sheaves on $\A^n$. Then $H^*_c(\A^n,\sF)$ is isomorphic to $i^!\sF$, where $i$ is the embedding of the origin into $\A^n$. 
\end{lem}

%Let $\pi: \bujsw \ra \buusw$ be the projection map. Then $\pi^* \pha \ltlm$ is $\phi_w$-equivariant. Hence hyperbolic localization is applicable to $\pi^* \pha \ltlm$. It states that $i_{\wdo T}^!i_{X^-_w}^*(\pi^* \pha \ltlm) \cong i_{\wdo T}^*i_{X^+_w}^!(\pi^* \pha \ltlm)$, where $i_{\wdo T}, i_{X^-_w}, i_{X^+_w}$ are the embeddings of $\wdo T, X^-_w, X^+_w$ into $\bujsw$ respectively.

\begin{prop}\label{avstalk}
Let $i_w$ be the embedding of $\gkw / \biu$ into $\gk / \biu$. When $w \in \ewlo$, $i_w^*(\Av_! \pha \ltulm)$ is isomorphic to $\underline{C}(w)^+_\sL[-\ell_{\sL}(w)]$.  When $w \notin \ewlo$, $i_w^*(\Av_! \pha \ltulm)$ vanishes. 
\end{prop}

\begin{proof}
	It suffices to show that the stalk at $\wdo$ is $\ql[-\ell_{\sL}(w)]$ when $w \in \ewlo$ and zero otherwise. We consider the following commutative diagram, 

\[\begin{tikzcd}
	\wdo \arrow[d, "d"] \arrow[r, "a"] & \wdo T  \arrow[d, "e"] \arrow[r,"b"] &  X_w^- \arrow[d, "f"] \arrow[r,"c"]  &  \buuw \arrow[d, "g"]\\
	\biu \wdo \biu \arrow[r, "h"]&X_w^+ \arrow[r, "i"] & \bujsw \arrow[r,"j"] & \buusw.
\end{tikzcd}\]
Here $c$ and $j$ are two torsors for the unipotent group $\biu / J_\infty$, and the other morphisms are closed embeddings. We have

%\begin{lem}
%Let $i:Z \ra X$ be the closed embedding of a variety $Z$ into a variety $X$. Let $G$ be a smooth algebraic group acting on $X$ such that $Z$ is $G$-invariant. Consider the following Cartesian diagram, 
%\[\begin{tikzcd}
%	Z \arrow[d, "p_Z"] \arrow[r, "i"] & X  \arrow[d,"p_X"] \\
%	\left [G\bsl Z \right] \arrow[r, "\bar{i}"]& \left [G \bsl X \right ].
%\end{tikzcd}\]
%For any sheaf $\sF$ on $[G \bsl X]$, there is a natural isomorphism from $p_Z^*\bar{i}^!\sF$ to $i^!p_X^*\sF$.
%\end{lem}
%\begin{proof}
%	By applying base change theorem to $p_X^*\sF$, we have a natural isomorphism $\bar{i}^!p_{X*}p_X^*\sF \xra{\sim} p_{Z*}i^!p_X^*\sF$. Composing with the natural transformation $\sF \ra p_{X*}p_X^*\sF$, we obtain a morphism $\bar{i}^! \sF \ra p_{Z*}i^!p_X^*\sF$, hence a morphism from $p_Z^*\bar{i}^!\sF$ to $i^!p_X^*\sF$. We claim that this morphism is an isomorphism. It suffices to check that on the stalk level. 
%\end{proof}
\begin{align*}
	i_{\wdo}^*(\Av_! \pha \ltulm) & = \tH^*_c(\biu \wdo \biu / \biu, \pha \ltulm) && \text{proper base change theorem} \\
	&= d^!h^*i^*j^*(\ltulm) && \text{Lemma } \ref{homolem} \\
	&= a^*e^!i^*j^*(\ltulm) && \text{Lemma } \ref{prodbase} \\
	&= a^*b^*f^!j^*(\ltulm) && \text{hyperbolic localization} \\
	&= a^*b^*f^*j^*(\ltulm)[2 \dim(\biu / J_\infty)] && f \text{ is smooth}  \\
	&= a^*b^*c^*g^*(\ltulm)[2 \dim(\biu / J_\infty)] \\
	&= a^*b^*c^*g^! (\ltulm) && g \text{ is smooth}  
	%&= i_{\wdo}^*i_{\wdo T}^*(\pi')^*\underline{C}(w)^-_\sL[-\ell_{\sL}(w)] && \text{Proposition \ref{stalk}} \\
	%&= \ql[-\ell_{\sL}(w)].
\end{align*}

The last term is $\ql[-\ell_{\sL}(w)]$ when $w \in \ewlo$ by Proposition \ref{stalk} and vanishes when $w \notin \ewlo$ by Proposition \ref{stalk2}. By Lemma \ref{bunsm}, we know that the variety $\bujsw$ is normal. Therefore, hyperbolic localization is applicable to the situation in Proposition \ref{hyperset}. 
\end{proof}

The same argument or Lemma \ref{avcom} shows that we can generalize the proposition to other maximal IC sheaves. 

\begin{prop} \label{stalkav}
Let $\beta \in \pha \lpewul$. Let $i_w$ be the embedding of $\gkw / \biu$ into $\gk / \biu$. When $w \in \beta$, $i_w^*(\Av_! \pha \ICu(w^\beta)_{\sL}^-)$ is isomorphic to $\underline{C}(w)^+_\sL[-\ell_{\sL}(w)]$.  When $w \notin \beta$, $i_w^*(\Av_! \pha \ICu(w^\beta)_{\sL}^-)$ vanishes. 
\end{prop}

\begin{defn}
For $w \in \ewlo$, define $C(w)_\sL^\dagger := i_w^* \pha (\Av_! \pha \ltlm) \langle \ell_{\sL}(w) \rangle$. Proposition \ref{stalkav} shows that $\omega C(w)_\sL^\dagger$ is isomorphic to $\underline{C}(w)_\sL^+$. Define $\Delta(w)_\sL^\dagger := i_{w!}(C(w)_\sL^\dagger)$, $\nabla(w)_\sL^\dagger := i_{w*}(C(w)_\sL^\dagger)$, and $\IC(w)_\sL^\dagger := i_{w!*}(C(w)_\sL^\dagger)$. 
\end{defn}

\begin{lem}\label{chess}
	For $w \in \ewlo$, there is a unique map $\theta_w^\dagger: \Av_! \pha \ltlm \ra \IC(w)_\sL^\dagger \langle -\ell_{\sL}(w) \rangle$, whose restriction to $\gkw / \biu$ is the identity map of $C(w)_\sL^\dagger \langle -\ell_{\sL}(w) \rangle$.
\end{lem}

\begin{proof}
	We mimic the proof of Theorem $3.4.1$ in \cite{BGS}. For any $n \in \N_{\ge 0}$, let $Y_n$ be the union of $\gkwp$ for all $w'$ such that $w' \le w$ and $\ell(w')=n$. With respect to this stratification, $\IC(w)_\sL^\dagger$ is filtered with the $n$-th subquotient $\bigoplus_{\ell(w')=n} i_{w'*}i_{w'}^! \IC(w)_\sL^\dagger$. This gives a spectral sequence whose $E_1$-page is $\bigoplus_{\ell(w')=n} \Hom^{\bullet}(i_{w'}^*\avltlm, i_{w'}^! \IC(w)_\sL^\dagger)$. We claim that the spectral sequence degenerates "like a chessboard". It suffices to show the degeneracy after base change to $k$. By Proposition \ref{parity} and Proposition \ref{avstalk}, $\Hom^{\bullet}(i_{w'}^*\avltulm, i_{w'}^! \underline{\IC}(w)_\sL^+)$ is the direct sum of $\Hom^{\bullet}(\underline{C}(w')_\sL^+[-\ell_\sL(w')], \underline{C}(w')_\sL^+[m])$, where $m \equiv \ell(w)-\ell(w') \mod 2$. Based on Lemma \ref{parityrel}, it is known that $\ell_\sL(w') \equiv \ell(w') \mod 2$, which implies the degeneracy of the spectral sequence. As a result, $M:=\Hom^{\bullet}(\avltlm, \IC(w)_\sL^\dagger)$ admits an increasing filtration $F_{\le n}$ by Fr-submodules whose associated graded piece $\Gr^F_nM$ is $\bigoplus_{\ell(w')=n}\Hom^{\bullet}(i_{w'}^*\avltlm, i_{w'}^! \IC(w)_\sL^\dagger)$. 
	
	After showing the existence of filtration, the argument in \cite[Lemma 6.8]{LY} implies both the existence and uniqueness of $\theta_w^\dagger$ by showing the quotient map $M \ra \Gr^F_nM$ is an isomorphism in degrees $\le 1$. 
	%shows that $Gr^F_nM$ is concentrated in degrees $ \ge 2$ for any $n < \ell(w)$. Therefore the
\end{proof}

\begin{lem} \label{daggerplus}
	 \begin{enumerate}[$(1)$]
		\item $\IC(e)_\sL^\dagger$ can be identified with $\delta_\sL$ canonically so that $\theta_e^\dagger$ coincides with the rigidification $\epsilon_\sL$. 
		\item When $s \in \extw$ is a simple reflection such that $s \in \ewlo$, there is a unique isomorphism $\iota_s : \IC(s)_\sL^\dagger \xra{\sim} \IC(s)^+_\sL$ so that $\iota_s \circ \theta_s^\dagger: \avltlm \ra \IC(s)^+_\sL \langle -1 \rangle$ restricts to the identity map on the stalks at $e$. Recall that $\IC(s)^+_\sL$ is introduced in Lemma \ref{proj}, and both stalks of $\avltlm$ and $\IC(s)^+_\sL \langle -1 \rangle$ at $e$ are equipped with an isomorphism with trivial $\Fr$-module $\ql$. 
	\end{enumerate}
\end{lem}

\begin{proof}
	See Lemma 6.9 in \cite{LY}.
\end{proof}

Hence, with the choice of a rigidified maximal IC sheaf, we can identify $\IC(s)_{\sL}^\dagger$ with $\IC(s)^+_{\sL}$.

\section{Characterization of averaging maximal IC sheaves}

In this section, we provide a characterization of maximal IC sheaves under averaging. For the sake of simplicity in notation, we omit the plus sign in positive sheaves throughout this section.

\subsection{Facts about general Coxeter group}

Let us review three facts about general Coxeter groups. Detailed proofs can be found in \cite{BB}. Let $(\sW,\sS)$ be a Coxeter group and $\ell$ be the length function. 

\begin{lem} \label{two}
	Let $u > v$ be two elements in $\mathcal{W}$ such that $\ell(u)=\ell(v)+2$. There exist exactly two elements $x_1$ and $x_2$ in $\mathcal{W}$ such that $u>x_i >v$ for $i=1,2$. 
\end{lem}

\begin{proof}
	See Lemma $2.7.3$ in \cite{BB}.
\end{proof}

\begin{lem}\label{ref}
	For any reflection $t$ in $\mathcal{W}$, there exists a palindromic reduced expression for $t$. 
\end{lem}

\begin{proof}
	The statement can be proved by applying Theorem 1.4.3 in \cite{BB} to the case when $w=t$. Let $t=s_1s_2 \cdots s_k$ be a reduced word expression. The theorem states that there exists an index $i$ such that $s_1s_2 \cdots s_{i-1} s_{i+1} \cdots s_k$ equals to the identity. Therefore, we have $s_1s_2 \cdots s_{i-1}= s_k s_{k-1} \cdots s_{i+1}$ and $i-1=k-i$. As a result, $t=s_1s_2 \cdots s_{i-1} s_i s_{i-1 }\cdots s_2 s_1$ is a palindromic reduced expression. 
\end{proof}

\begin{defn} \label{defanti}
	Let $\sV$ be the set $\{(a,b)\in \sW^2: \ell(b)-\ell(a)=1$ and $a<b\}$.
	A $\ql$-valued function $f$ defined on $\sV$ is said to be \emph{anti-commutative} if it satisfies the following condition. For any $(v,u) \in \sW^2$ such that $\ell(u)=\ell(v)+2$ and $u>v$, we have \[f(v,x_1)f(x_1,u)+f(v,x_2)f(x_2,u)=0,\] where $x_1$ and $x_2$ are the two elements given in Lemma \ref{two}. 
\end{defn}

\begin{defn} \label{defconj}
Let $f$ and $g$ be two anti-commutative functions on $\sV$. We say $f$ is conjugated to $g$ if there exists an everywhere non-zero function $h$ on $\sW$ such that 
\[h(v)^{-1}f(v,u)h(u)=g(v,u)\]
for any $(v,u) \in \sV$.
\end{defn}

\begin{lem} \label{everynonzero}
	Let $f$ be an anti-commutative function on $\sV$. If $f(v,vs) \neq 0$ for any $v \in \sW$ and simple reflection $s$ in $\sW$ such that $v<vs$, then $f(v,u) \neq 0$ for any $(v,u) \in \sV$. 
\end{lem}

\begin{proof}
	We induct on the length $\ell(v)$. When $\ell(v)=0$, $u$ has to be a simple reflection in $\sW$.  Let $s_{i_1}s_{i_2} \dots s_{i_n}$ be a reduced word expression of $u$ in $\sW$. There exists an index $j$ such that $1 \le j \le n$ and $s_{i_1}s_{i_2} \dots s_{i_{j-1}} s_{i_{j+1}} \dots s_{i_n}$ is a reduced word expression of $v$. If $j$ equals to $n$, the assumption implies that $f(v,u)$ is non-zero. Hence we will assume $j<n$. We write $t:=s_{i_n}$. Then we have $ut< u$ and $vt< v$. 
	
	The function $f$ being anti-commutative implies $f(vt,v)f(v,u)+f(vt,ut)f(ut,u)=0$. By our induction hypothesis, we know that $f(vt,ut)$ is non-zero. Our assumption implies that $f(ut,u)$ is non-zero. Therefore $f(v,u)$ is non-zero. 
\end{proof}

\begin{prop} \label{conjfcn}
	Any two anti-commutative $\ql^*$-valued functions $f, g$ on $\sV$ are conjugate.
\end{prop}

\begin{proof}
	Let $\sW_{\le n}$ (resp. $\sW_n$) be the subset of $\sW$ containing elements of length not greater than $n$ (resp. equal to $n$). Let $\sV_{\le n}$ be the intersection of $\sW_{\le n-1} \times \sW_{\le n}$ and $\sV$.
	
	For every non-negative integer $n$, we will inductively construct a  function $h_{n}$ on $\sW_{\le n}$ such that $h_n(v)^{-1}f(v,u)h_n(u)=g(v,u)$ for any $(v,u) \in \sV_{\le n}$ and the restriction of $h_n$ on $\sW_{\le n-1}$ equals to $h_{n-1}$. 
	
	We set $h_0(e)=1$, where $e$ is the identity element in $\sW$. When $n$ equals to $1$, we set $h_1(s)=g(e,s)/f(e,s)$ for any simple reflection $s$ and $h_1(e)=1$. Now we assume that $n$ is greater than one and $h_{n-1}$ has already been constructed.
	
	For every $u \in \sW_n$, we select a reduced word expression $u:=s_{i_1}s_{i_2}\dots s_{i_n}$. Let $y=us_{i_n} \in \sW_{n-1}$. We set $h_n(u)$ to be $g(y,u)/f(y,u)h_{n-1}(y)$, and set $h_n(v)=h_{n-1}(v)$ for $v \in \sW_{n-1}$. We will show that $h_n$ satisfies the condition. 
	
	For any $(v,u) \in \sV_{\le n}$, there exists an index $j$ such that $v=s_{i_1}s_{i_2}\dots s_{i_{j-1}} s_{i_{j+1}} \dots s_{i_n}$. If $j=n$, then $v$ and $y$ coincide. By our assignment, we know that $h_n(v)^{-1}f(v,u)h_n(u)=g(v,u)$. When $j$ and $n$ are distinct, let $t$ be $s_{i_n}$. Since $vt<v$, we have the equations $f(vt,v)f(v,u)=-f(vt,y)f(y,u)$ and $g(vt,v)g(v,u)=-g(vt,y)g(y,u)$. We divide the first equation by the second equation and obtain \[\frac{h_n(vt)}{h_n(v)} \times \frac{f(v,u)}{g(v,u)}= \frac{h_n(vt)}{h_n(y)} \times \frac{h_n(y)}{h_n(u)}.\] Therefore $h_n$ satisfies the condition. Finally, we take $h$ to be the union of the functions $h_n$ and it satisfies the condition in Definition \ref{defconj}. 
\end{proof}

\subsection{Homomorphisms between standard sheaves}\label{homstds} In this subsection, we study homomorphisms between standard sheaves.

\begin{lem}\label{conj}
	Let $s$ be a simple reflection in $\ewlo$, there exists a minimal element $x$ in $_{\sL'}\extw_\sL$ such that $x^{-1} s x$ is a simple reflection in $\extw$ and $x^{-1} s x \in \extw^\circ_{x \sL}$.
\end{lem}

\begin{proof}
	We note that $s$ is in the affine Weyl group. By Lemma \ref{ref}, write $s=s_{i_1}s_{i_2} \cdots s_{i_k} \cdots s_{i_2}s_{i_1}$ in a palindromic reduced expression in $\extw$. Let $\sL_j=s_{i_j}\cdots s_{i_1}(\sL)$ for $j \ge 1$.  Since $\ell_{\sL}(s)=1$, we know that $s_{i_j} \notin \extw^\circ_{\sL_{j-1}}$ for $j<k$ and $s_{i_k} \in \extw^\circ_{\sL_{k-1}}$ by Lemma $4.6$ in \cite{LY}. We can choose $x$ to be $s_{i_{k-1}}\cdots s_{i_2}s_{i_1}$, which is a minimal element by Lemma \ref{minelt}. 
\end{proof}

Lemma \ref{conj} is essential to the proofs presented in this section as it enables us to prove statements involving simple reflections in $\ewlo$ by reducing them to statements involving simple reflections in $\extw$. For instance, we have the following lemma. 

\begin{lem}\label{stdicper}
	Let $s$ be a simple reflection in $\ewlo$. Let $u,v \in \beta$ such that $u<_{\beta}us$ and $v<_{\beta}sv$. Then both $\Delu(u) \star \ICu(s)$ and $\ICu(s) \star \Delu(v)$ are perverse. Dually, both $\nabu(u) \star \ICu(s)$ and $\ICu(s) \star \nabu(v)$ are perverse. 
\end{lem}

\begin{proof}
	The second statement can be obtained by applying Lemma \ref{Verdier} to the first statement. When $\sL$ is trivial, the first statement follows from Lemma $3(b)$ in \cite{AB}. Their proof can be generalized directly to the monodromic case. We will only prove the perversity of $\Delu(u) \star \ICu(s)$; the other case can be proved similarly. 
	
	Let $x$ be a minimal element in Lemma \ref{conj} such that $x^{-1}sx$ is a simple reflection in $\extw$ and $x^{-1}sx \in \extw^\circ_{x \sL}$. Proposition \ref{minconv} states that the equivalences $(-)\star \ICu(x)$ and $\ICu(x^{-1}) \star (-)$ preserve the $t$-structure. Moreover, $\ICu(x^{-1}) \star (-) \star \ICu(x)$ sends $\Delu(u)$ to $\Delu(x^{-1}ux)$ and $\ICu(s)$ to $\ICu(x^{-1}sx)$. Hence $\Delu(u) \star \ICu(s)$ is perverse if and only if $\Delu(x^{-1}ux) \star \ICu(x^{-1}sx)$ is perverse. Combining Lemma \ref{minelt} and Lemma \ref{conjmin}, we conclude that $x^{-1}ux<_{x^{-1}\beta x}x^{-1}usx$. Therefore we may assume $s$ is a simple reflection in $\extw$. %Proposition \ref{minconv} states that the equivalence $\ICu(w^{\beta,-1})\star (-)$ sends $\Delu(u)$ to $\Delu(w^{\beta,-1}u)$. Therefore we may further assume that 
	
	Note that $(-)\star \ICu(s)$ is isomorphic to $\pi_s^{+*}\pi^{+}_{s*}(-)[1]$, where $\pi_s^{+}$ is the projection map from the affine flag variety $\gk/ \bi$ to the partial flag variety $\gk / P_s$ (See \cite[Section 3.7]{LY} and Lemma \ref{proj}). The composition $\pi_s^+ \circ i_u$ is a locally closed affine embedding, where $i_u$ is the embedding of the $u$-stratum to the affine flag variety $\gk/ \biu$. Since $\pi_s^+$ is smooth, projective, and of relative dimension one, we know that $\pi_s^{+*}\pi^{+}_{s*} \Delu(u)$ is perverse.
\end{proof}

%\begin{cor}
%	Let $u \in \extw$
%\end{cor}

\begin{lem}\label{std3}
	Let $u,v$ be two elements in $\ewlo$ such that $\ell_{\sL}(u)+\ell_{\sL}(v)=\ell_{\sL}(uv)$. Then $\Delu(u) \star \Delu(v)$ and $\Delu(uv)$ are isomorphic.
\end{lem}

\begin{proof}
	By using induction on $\ell_{\sL}(v)$, we can assume that $v=s$ is a simple reflection in $\ewlo$. Let $x$ be a minimal element in Lemma \ref{conj} such that $x^{-1}sx$ is a simple reflection in $\extw$ and $x^{-1}sx \in \extw^\circ_{x \sL}$. Lemma \ref{conjmin} implies that $x^{-1}ux<_{\ewlo}x^{-1}usx$. By Proposition \ref{order}, we have $x^{-1}ux<x^{-1}usx$. Consequently, $\Delu(x^{-1}ux) \star \Delu(x^{-1}sx)$ and $\Delu(x^{-1}usx)$ are isomorphic by Lemma \ref{stdstd}.
	
	Proposition \ref{minconv} asserts that the equivalence $\ICu(x^{-1}) \star (-) \star \ICu(x)$ sends $\Delu(u)$, $\Delu(s)$, and $\Delu(us)$ to $\Delu(x^{-1}ux)$, $\Delu(x^{-1}sx)$, and $\Delu(x^{-1}usx)$, respectively. Hence, $\Delu(u) \star \Delu(s)$ and $\Delu(us)$ are isomorphic.
\end{proof}

\begin{prop}\label{stdleft}
	For $w \in \extw$, both functors $(-) \star \Delu(w)$ and $\Delu(w) \star (-)$ are left exact. Dually, $(-) \star \nabu(w)$ and $\nabu(w) \star (-)$ are right exact.
\end{prop}	

\begin{proof}
	We will only prove that the functor $(-) \star \Delu(w)$ is left exact. The rest of the statement can be proved similarly. When $\ell(w)=0$, the statements follow from Proposition \ref{minconv}. Hence, we can assume $w$ is in the affine Weyl group. By Lemma \ref{stdstd}, it suffices to consider the case when $w=s$ is a simple reflection in $\extw$. When $s \notin \ewlo$, the statements follow from Lemma \ref{equiv}. 
	
	Now, assuming that $s \in \ewlo$, we first prove that $\nabu(u) \star \Delu(s)$ is perverse for $u \in \extw$ (compare with \cite[Lemma 4]{Be}). If $u<us$, we have the following distinguished triangle, $\nabu(u) \star \Delu(e) \ra \nabu(u) \star \Delu(s) \ra \nabu(u) \star \ICu(s)$. By Lemma \ref{stdicper}, we know that $\nabu(u) \star \Delu(s)$ is perverse. If $u>us$, $\nabu(u) \star \Delu(s)$ is isomorphic to $\nabu(us)$ by Lemma \ref{stdstd} and Lemma \ref{inverse}.
	
	To show that the functor $(-) \star \Delu(s)$ is left exact, it suffices to prove that the perverse cohomology of $\ICu(u) \star \Delu(s)$ concentrates on non-negative degrees for any $u \in \extw$. We prove this by inducting on the length of $u$. Let $Q$ be the cokernel of the natural morphism $ \ICu(u) \ra \nabu(u)$. Note that $Q$ has a filtration whose subquotients are given by $\ICu(v)$ for $v<u$ and $\ICu(u) \ra \nabu(u) \ra Q$ is a distinguished triangle. By the induction hypothesis, the perverse cohomology of $Q \star \Delu(s)$ concentrates on non-negative degrees. Since $\nabu(u) \star \Delu(s)$ is perverse, we know that the perverse cohomology of $\ICu(u) \star \Delu(s)$ concentrates on non-negative degrees by applying the functor $(-) \star \Delu(s)$ to the above distinguished triangle. 
\end{proof}	

\begin{lem}\label{ext}
	Let $\sF$ and $\sG$ be two perverse sheaves in $\lpdl$. The group $\Ext^j(\sF, \sG)$ vanishes if $j<0$. In particular, the group $\Ext^j(\Delu(w_1), \Delu(w_2))=0$ if $j <0$. 
\end{lem}

\begin{proof}
	By definition, the category of perverse sheaves is the heart of the perverse t-structure of $\lpdl$. The first statement follows from the definition of the t-structure of triangulated categories. For the second statement, we note that both $\Delu(w_1)$ and $\Delu(w_2)$ are perverse sheaves. 
\end{proof}

%\begin{prop}\label{parityhom}
%	Let $w_1,w_2 \in \beta$ and let $k=\ell_\beta(w_2)-\ell_\beta(w_1)$. The group $\Ext^j(\Delu(w_1),\Delu(w_2))$ vanishes when $j$ and $k$ have different parity. 
%\end{prop}
%
%\begin{proof}
%	Corollary \ref{filter} implies that $\Delu(w_2)$ is a successive extension of $\ICu(y)[n]$ for $y \leq_{\beta} w_2$ and $n \in \Z$ such that $n$ and $\ell_\beta(w_2)-\ell_\beta(y)$ have same parity. It suffices to show $\Ext^j(\Delu(w_1), \ICu(y)[n])$ vanishes for any such $y$ and $n$.  We observe that $\Ext^j(\Delu(w_1), \ICu(y)[n])$ is isomorphic to the $!$-restriction $i_{w_1}^{!}\ICu(y)[n+j]$. The $!$-restriction vanishes whenever $n+j \not\equiv \ell(y)-\ell(w_1) \mod 2$ because of Proposition \ref{parity}. The statement now follows from Lemma \ref{parityrel}, which states $\ell_\beta(y)-\ell_\beta(w_1) \equiv \ell(y)-\ell(w_1) \mod 2$. 
%\end{proof}

\begin{defn}\label{defv}
Let $V_n^{\beta}$ be the set of $(w_1,w_2) \in \beta^2$ such that $w_1 <_\beta w_2$ and $\ell_\beta(w_1)+1= \ell_\beta(w_2)=n$. Let $V^{\beta}$ be the union of $V_n^{\beta}$ for $n \in \N$. When $\beta$ is the neutral block, we write $V_n^\circ$ and $V^\circ$ instead of $V_n^{\beta}$ and $V^{\beta}$. 
\end{defn}

%Notice that there is a natural bijection from $V_n^\circ$ to $V_{n, \beta}$ by left multiplication of $w^\beta$ (which is different from right multiplication of $w^\beta$). 

\begin{prop}\label{reflect}
	Let $ w_1, w_2 \in \beta$. If $\Hom^0(\Delu(w_1)[-\ell_{\beta}(w_1)], \Delu(w_2)[-\ell_{\beta}(w_2)+1])$ is non-zero, then $(w_1,w_2) \in V^{\beta}$. 
\end{prop}

\begin{proof}
	Lemma \ref{filter} states that the sheaf $\Delu(w_2)$ is a successive "extension" of $\ICu(y)[n]$ for $y \leq_{\beta} w_2$ and $n \in \Z$. Therefore $\Hom^0(\Delu(w_1)[-\ell_{\beta}(w_1)], \ICu[y][n])$ has to be non-zero for some $y \leq_\beta w_2$ and $n \in \Z$. Lemma \ref{vanstalk} implies that $w_1 \leq_\beta y$ for any such $y$. The transitivity of $\le_\beta$ implies $w_1 \le_\beta w_2$. The length relationship follows from Lemma \ref{ext}.
\end{proof}

The converse of Proposition \ref{reflect} is true.

\begin{lem} \label{onedim}
	Let $(w_1,w_2) \in V^{\beta}$. Then $\Hom^0(\Delu(w_1), \Delu(w_2))=\ql$. In particular, the dimension of $\hom^0(\Delta(\dot{w_1})_\sL, \Delta(\dot{w_2 })_\sL)$ is either $0$ or $1$. 
	
	%The cone of a non-zero homomorphism is isomorphic to the pullback of a standard sheaf on $\gk / P_s$, where $P_s$ is the parahoric subgroup of $\gk$ corresponding to $s$. 
\end{lem}

\begin{proof}
	The second statement follows from Lemma \ref{mixnonmix} and the first statement. Let $w_i':= (w^{\beta})^{-1}w_i$ for $i=1,2$. By Proposition \ref{minconv}, we know that the group $\Hom^0(\Delu(w_1), \Delu(w_2))$ is isomorphic to $\Hom^0(\Delu(w_1'), \Delu(w_2'))$. Let $w_2'=s_{i_1}s_{i_2}\cdots s_{i_k}$ be a reduced expression in terms of simple reflections in $\ewlo$. Then there exists an index $j$ such that $w_1'=s_{i_1}s_{i_2}\cdots s_{i_{j-1}} s_{i_{j+1}} \cdots s_{i_k}$. Since left (right) convolution with $\Delu(s_{i_l})$ is an equivalence for $1 \le l \le k$ by Lemma \ref{inverse}, it suffices to show that $\Hom^0(\Delu(e)_\sL, \Delu(s_{i_j})_\sL)=\ql$. Let $x$ be a minimal element in Lemma \ref{conj} such that $x^{-1}s_{i_j}x$ is a simple reflection in $\extw$ and $x^{-1}s_{i_j}x \in \extw^\circ_{x \sL}$. By Proposition \ref{minconv}, we know that $\Hom^0(\Delu(e)_\sL, \Delu(s_{i_j})_\sL)=\Hom^0(\Delu(e)_{x \sL}, \Delu(x^{-1}s_{i_j}x)_{x \sL})$. This latter group is $\ql$, by a standard argument as in the usual affine Hecke category. 
\end{proof}

\begin{prop} \label{Vermainj}
	Let $(w_1,w_2) \in V^{\beta}$. Any non-zero map from the perverse sheaf $\Delu(w_1)$ to the perverse sheaf $\Delu(w_2)$ is injective. 
\end{prop}

\begin{proof}
	Based on Lemma \ref{onedim}, it suffices to show a non-zero morphism is injective. Similar to the proof in Lemma \ref{onedim}, we can assume $\beta$ is the neutral block. 
	
	Let $w_2=s_{i_1}s_{i_2}\cdots s_{i_k}$ be a reduced expression in terms of simple reflections in $\ewlo$. Then $w_1=s_{i_1}s_{i_2}\cdots s_{i_{j-1}} s_{i_{j+1}} \cdots s_{i_k}$  for some $j$. For simplicity, we write $u:=s_{i_1}s_{i_2}\cdots s_{i_{j-1}}$, $v:=s_{i_{j+1}} \cdots s_{i_k}$, and $s:=s_{i_j}$. Then $w_1=uv$ and $w_2=usv$. 
	
	We claim that a (every) non-zero map from $\Delu(u)$ to $\Delu(us)$ is injective. Let $x$ be a minimal element in Lemma \ref{conj} such that $x^{-1}sx$ is a simple reflection in $\extw$ and $x^{-1}sx \in \extw^\circ_{x \sL}$. Proposition \ref{minconv} states that the functors $(-)\star \ICu(x)$ and $\ICu(x^{-1}) \star (-)$ preserve $t$-structure. Moreover, $\ICu(x^{-1}) \star \Delu(u) \star \ICu(x) \cong \Delu(x^{-1}ux)$ and $\ICu(x^{-1}) \star \Delu(us) \star \ICu(x) \cong \Delu(x^{-1}uxx^{-1}sx)$. Therefore we may and will assume $s$ is a simple reflection in $\extw$ and $s \in \ewlo$. In this case, there is a distinguished triangle $\Delu(u) \ra \Delu(us) \ra \Delu(u) \star \ICu(s) \xra{[1]}$. Lemma \ref{stdicper} implies that $\Delu(u) \star \ICu(s)$ is perverse. Hence, the morphism from $\Delu(u)$ to $\Delu(us)$ is injective. Similarly, a (every) non-zero map from $\Delu(v)$ to $\Delu(sv)$ is injective.

	By applying $\Delu(u) \star (-) \star \Delu(v)$ to $\Delu(e) \ra \Delu(s) \ra \ICu(s) \xra{[1]}$, we obtain a distinguished triangle $\Delu(u) \star \Delu(v) \xra{\iota} \Delu(u) \star \Delu(s) \star \Delu(v) \ra \Delu(u) \star \ICu(s) \star \Delu(v) \xra{[1]} $. The first map $\iota$ can be viewed as a morphism from $\Delu(w_1)$ to $\Delu(w_2)$ by Lemma \ref{std3}. The goal is to show that $\iota$ is injective. Given that the first two terms are perverse, the injectivity of $\iota$ is equivalent to the perversity of $\Delu(u) \star \ICu(s) \star \Delu(v)$. From the distinguished triangle, we infer that the perverse cohomology of $\Delu(u) \star \ICu(s) \star \Delu(v)$ is concentrated in degrees $-1$ and $0$. Proposition \ref{stdleft} implies that $\Delu(u) \star \ICu(s) \star \Delu(v)$ is perverse. 
%	
%	
%		
%	Therefore, it is sufficient to demonstrate that $\Delu(u) \star \ICu(s) \star \Delu(v)$ remains invariant under the Verdier duality $\D$. This is true because 
%	\begin{align*}
%		\D( \Delu(u) \star \ICu(s) \star \Delu(v)) & \cong \D( \Delu(u) \star \ICu(s)) \star \nabu(v) && \text{Lemma } \ref{Verdier}\\
%		&\cong \Delu(u) \star \ICu(s) \star \nabu(v) && \Delu(u) \star \ICu(s)\text{ is perverse}  \\
%		&\cong \Delu(u) \star \D(\ICu(s) \star \Delu(v)) &&  \\
%		&\cong \Delu(u) \star \ICu(s) \star \Delu(v) && \ICu(s) \star \Delu(v) \text{ is perverse.}
%	\end{align*}
\end{proof}

\begin{rem}
	This proposition can be seen as the sheaf analogue of the result stating that any non-zero homomorphism between Verma modules is injective (c.f. \cite[Theorem 4.2]{Hu}).
\end{rem}

\begin{prop}\label{Vermavan}
	Let $w_1,w_2 \in \beta$ such that $\ell_\beta(w_1) \geq \ell_\beta(w_2)$ and $w_1 \neq w_2$. The group $\Ext^k(\Delu(w_1),\Delu(w_2))$ vanishes for any integer $k$.
\end{prop}

\begin{proof}
	When $k$ is negative, it follows from Lemma \ref{ext}. Now, we assume $k$ is non-negative. Given that convolving with the minimal sheaf $\ICu(w^\beta)$ is an equivalence, as stated in Proposition \ref{minconv}, it suffices to prove the case when $\beta$ is the neutral block.
	
	We induct on the length $n:=\ell_\beta(w_2)$. When $n=0$, the statement is trivial. Let $s_{i_1}s_{i_2} \dots s_{i_m}$ (resp. $s_{j_1}s_{j_2} \dots s_{j_n}$) be a reduced word expression of $w_1$ (resp. $w_2$) in $\ewlo$. Let $x$ be a minimal element in Lemma \ref{conj} such that $x^{-1}s_{j_n}x$ is a simple reflection in $\extw$ and $x^{-1}s_{j_n}x \in \extw^\circ_{x \sL}$. Since $\IC(x^{-1}) \star (-) \star \IC(x)$ is an equivalence, as stated in Proposition \ref{minconv}, $\Ext^k(\Delu(w_1),\Delu(w_2))$ is isomorphic to $\Ext^k(\Delu(x^{-1}w_1x),\Delu(x^{-1}w_2x))$. Because of Lemma \ref{conjmin}, we may and will assume $s_{j_n}$ is a simple reflection in $\extw$.
	
	If $\ell_{\beta}(w_1)>\ell_{\beta}(w_1s_{j_n})$, then $\Delu(w_1) \cong \Delu(w_1s_{j_n}) \star \Delu(s_{j_n})$ and $\Delu(w_2) \cong \Delu(w_2s_{j_n}) \star \Delu(s_{j_n})$. Since $(-) \star \Delu(s_{j_n})$ is an equivalence by Lemma \ref{inverse}, we know that the group $\Ext^k(\Delu(w_1),\Delu(w_2))$ is isomorphic to $\Ext^k(\Delu(w_1s_{j_n}),\Delu(w_2s_{j_n}))$. The latter group is zero by the induction hypothesis. Hence, we can assume $\ell_{\beta}(w_1)<\ell_{\beta}(w_1s_{j_n})$.
	
	In the proof of Proposition \ref{Vermainj}, we show that there are two exact sequences of perverse sheaves, 
	\[0 \ra \Delu(w_2s_{j_n}) \ra \Delu(w_2) \ra \Delu(w_2) \star \ICu(s_{j_n}) \ra 0, \label{eq:1} \tag{$1$}\]
	\[0 \ra \Delu(w_1) \ra \Delu(w_1s_{j_n}) \ra \Delu(w_1) \star \ICu(s_{j_n}) \ra 0. \label{eq:2} \tag{$2$}\]
	
	By applying $\Hom(\Delu(w_1),-)$ to \eqref{eq:1}, we get a long exact sequence \[0 \ra \Hom^0(\Delu(w_1),\Delu(w_2s_{j_n})) \ra \Hom^0(\Delu(w_1),\Delu(w_2)) \ra \Hom^0(\Delu(w_1),\Delu(w_2)\star \ICu(s_{j_n})) \ra \] 
	\[\cdots \ra \Ext^k(\Delu(w_1),\Delu(w_2s_{j_n})) \ra \Ext^k(\Delu(w_1),\Delu(w_2)) \ra \Ext^k(\Delu(w_1),\Delu(w_2)\star \ICu(s_{j_n})) \ra \cdots.\]
	
	Since $\ell_{\beta}(w_1) > \ell_{\beta}(w_2s_{j_n})$, we know that $\Ext^i(\Delu(w_1),\Delu(w_2s_{j_n}))=0$ for any $i \in \Z$ by the induction hypothesis. Hence, $\Ext^i(\Delu(w_1),\Delu(w_2))$ is isomorphic to  $\Ext^i(\Delu(w_1),\Delu(w_2)\star \ICu(s_{j_n}))$.
	
	By applying $\Hom(-,\Delu(w_2))$ to \eqref{eq:2}, we get a long exact sequence 
	\[0 \ra \Hom^0(\Delu(w_1)\star \ICu(s_{j_n}),\Delu(w_2)) \ra \Hom^0(\Delu(w_1s_{j_n}),\Delu(w_2)) \ra \Hom^0(\Delu(w_1),\Delu(w_2)) \ra \] 
	\[\cdots \ra \Ext^k(\Delu(w_1)\star \ICu(s_{j_n}),\Delu(w_2)) \ra \Ext^k(\Delu(w_1s_{j_n}),\Delu(w_2)) \ra \Ext^k(\Delu(w_1),\Delu(w_2)) \ra \cdots .\] 
	
	By the induction hypothesis, we know that $\Ext^i(\Delu(w_1s_{j_n}),\Delu(w_2))=0$ for any $i \in \Z$. Hence, $\Ext^i(\Delu(w_1)\star \ICu(s_{j_n}),\Delu(w_2))$ is isomorphic to $\Ext^{i-1}(\Delu(w_1),\Delu(w_2))$. 
	
	By Lemma \ref{proj}, $\Ext^i(\Delu(w_1),\Delu(w_2)\star \ICu(s_{j_n}))$ is isomorphic to $\Ext^i(\Delu(w_1)\star \ICu(s_{j_n}),\Delu(w_2))$. Therefore we conclude that $\Ext^i(\Delu(w_1),\Delu(w_2))$ is isomorphic to $\Ext^{i-1}(\Delu(w_1),\Delu(w_2))$ for any $i \in \Z$. Since $\Ext^{-1}(\Delu(w_1),\Delu(w_2))=0$, the proposition is proved.
\end{proof}

\begin{prop} \label{difftwo}
	Let $v,u \in \beta$ such that $v <_\beta u$ and $\ell_\beta(v)+2=\ell_\beta(u)$. The group $\Hom^0(\Delu(v),\Delu(u))$ is a one-dimensional $\ql$-vector space. 
\end{prop}

\begin{proof}
	Without loss of generality, we may assume $\beta$ is the neutral block. Let $s_{i_1}s_{i_2} \dots s_{i_n}$ be a reduced word expression for $u$ in $\ewlo$. Then there exist two indices $j$ and $k$ such that $1 \le j < k \le n$ and $v=s_{i_1}s_{i_2} \dots s_{i_{j-1}}s_{i_{j+1}} \dots s_{i_{k-1}}s_{i_{k+1}} \dots s_{i_n}$. For simplicity, let $w_1:=s_{i_1}s_{i_2} \dots s_{i_{j-1}}$,  $w_2:=s_{i_{j+1}} \dots s_{i_{k-1}}$, $w_3:=s_{i_{k+1}} \dots s_{i_n}$, $s:=s_{i_j}$, and $t:=s_{i_k}$. Then $u=w_1sw_2tw_3$ and $v=w_1w_2w_3$. Since both $\Delu(w_1) \star (-)$ and $(-) \star \Delu(w_3)$ are equivalences by Lemma \ref{inverse}, the group $\Hom^0(\Delu(v),\Delu(u))$ is isomorphic to $\Hom^0(\Delu(w_2),\Delu(sw_2t))$. Similar to Proposition \ref{Vermavan}, we may assume $t$ is a simple reflection in $\extw$. 
	
	In the proof of Proposition \ref{Vermainj}, we show that there are two exact sequences of perverse sheaves, 
	\[0 \ra \Delu(sw_2) \ra \Delu(sw_2t) \ra \Delu(sw_2) \star \ICu(t) \ra 0, \label{eq:3} \tag{$3$}\]
	\[0 \ra \Delu(w_2) \ra \Delu(w_2t) \ra \Delu(w_2) \star \ICu(t) \ra 0. \label{eq:4} \tag{$4$}\]

	Lemma \ref{onedim} implies that $\Hom(\Delu(w_2),\Delu(sw_2))$ is one-dimensional. By applying the functor $\Hom(\Delu(w_2),-)$ to the exact sequence \eqref{eq:3}, we have the following exact sequence
	\[0 \ra \Hom(\Delu(w_2),\Delu(sw_2)) \ra \Hom(\Delu(w_2),\Delu(sw_2t)) \ra \Hom(\Delu(w_2),\Delu(sw_2) \star \ICu(t)). \]
	Hence, it suffices to show that $\Hom(\Delu(w_2),\Delu(sw_2) \star \ICu(t))$ vanishes.
	
	By Lemma \ref{proj}, we know that $\Hom(\Delu(w_2),\Delu(sw_2) \star \ICu(t))$ is isomorphic to $\Hom(\Delu(w_2) \star \ICu(t), \Delu(sw_2))$. By applying the left-exact functor $\Hom(-,\Delu(sw_2))$ to the exact sequence \eqref{eq:4}, we know that $\Hom(\Delu(w_2) \star \ICu(t), \Delu(sw_2))$ embeds into $\Hom(\Delu(w_2t), \Delu(sw_2))$. The last group vanishes because of Proposition \ref{Vermavan}. 
\end{proof}

\begin{lem} \label{twobeta}
	Let $v,u \in \beta$ such that $v <_\beta u$ and $\ell_\beta(v)+2=\ell_\beta(u)$. There exist exactly two elements $x_1$ and $x_2$ such that $(v,x_i) \in V_\beta$ and $(x_i,u) \in V_\beta$ for $i=1,2$.
\end{lem}

\begin{proof}
 	Let $u':=(w^\beta)^{-1}u$ and $v':=(w^\beta)^{-1}v$. From Lemma \ref{two}, we know that there exist exactly two elements $y_1,y_2 \in \ewlo$ such that $v'<y_i<u'$ for $i=1,2$. Hence $x_i:=w^{\beta}y_i$ satisfy the condition. 
\end{proof}

\begin{lem}\label{cupinj}
	Let $x_1$ and $x_2$ be the two elements given in Lemma \ref{twobeta}. The composition map $\circ_i: \Hom^0(\Delu(v),\Delu(x_i)) \times \Hom^0(\Delu(x_i),\Delu(u)) \ra \Hom^0(\Delu(v),\Delu(u))$ is an isomorphism for $i=1,2$. 
\end{lem}

\begin{proof}
	For $i = 1,2$, Proposition \ref{Vermainj} implies that any morphism in $\Hom^0(\Delu(v),\Delu(x_i))$ is injective. Likewise, any morphism in $\Hom^0(\Delu(x_i),\Delu(u))$ is injective. By Proposition \ref{difftwo}, both the source and target spaces have dimension one. The statement follows from the fact that the composition of two injective maps is non-zero.
\end{proof}

\subsection{Monostalk complexes}
	Let $\sF \in \pha \lpdul$ and let $S$ be a subset of $\extw$. We denote $\sF_{S}$ as the extension by zero of $\sF$ on ${\bigcup_{w \in S}{\gkw /\biu}}$. When $S$ equals to $\{w\}$ (resp. $\{w' \in \extw: w' \le w\}$), we write $\sF_w$ (resp. $\sF_{\le w}$) instead of $\sF_S$.
	
	Let $\beta \in \pha \lpewul$ and $\extw^\beta_{\le n}$ be the elements $w$ in $\beta$ such that $\ell_\beta(w) \le n$. Let $cl(\extw^\beta_{\le n})$ be the subset of $w' \in \extw$ such that $w' \le w$ for some $w \in \extw^\beta_{\le n}$. Let $\extw^\beta_{n}$ be $cl(\extw^\beta_{\le n}) \bsl cl(\extw^\beta_{\le n-1})$.
	
	For any sheaf $\sF$ with support contained in $\beta$ and any subset $T$ in $\N_{\ge 0}$, let $\sF_T$ be $\sF_{\bigcup _{t \in T} \extw^\beta_{t}}$. When $T$ equals to the set $\{n\}$ (resp. $\{m \in \N_{\ge 0} : m \le n\}$, $\{m \in \N_{\ge 0} : n'< m \le n\}$), we write $\sF_n$ (resp. $\sF_{\le n}$, $\sF_{(n',n]}$) instead of $\sF_T$. 

We define $\lpdulk$ to be the non-mixed analog of $\lpdul$. There is a natural functor $\omega: \pha \lpdul \ra \pha \lpdulk$ (see Section \ref{sectFrob}).

\begin{defn}\label{definemono}
	Let $\sF$ be a complex in $\lpdulk$. We call $\sF$ to be a \emph{monostalk complex} with respect to $\beta$ if it satisfies the following properties:
	\begin{enumerate}
		\item $\sF_w = 0$ for $w \notin \beta$;
		\item $\sF_w \cong \Delu (w)_\sL[-\ell_\beta(w)]$ for $w \in \beta$;
	\end{enumerate}
\end{defn}

We say that $\sF \in \pha \lpdul$ is a monostalk complex if $\omega \sF$ is a monostalk complex in the non-mixed category $\lpdulk$. 

%\begin{itemize}
%	\item [(2')] $\sF_w \cong \Delu (w)_\sL[-\ell_\beta(w)]$ for $w \in \beta$;
%\end{itemize}

%It is clear that $\omega \sF$ is a monostalk complex when $\sF$ is a monostalk complex. 

\begin{prop}\label{dissupp}
	Let $\sF$ be a monostalk complex with respect to $\beta$. There is an isomorphism between $\sF_n$ and $\oplus_{w \in \extw^\beta_{n}} \sF_{w}$. Equivalently, $\sF_n$ is isomorphic to the direct sum $\oplus_{w \in \beta : \ell_\beta(w)=n} \sF_{w}$. 
\end{prop}

\begin{proof}
	First, we consider the case when $\sF$ is non-mixed. By successively applying the open-closed distinguished triangle  on $\sF_n$ with respect to the Schubert stratification, we establish the existence of a natural "filtration" of $\sF_n$ whose subquotients are given by $i_{w!}i_w^*\sF$ for $w \in \extw^\beta_{n}$, where $i_w$ is the embedding of $\gkw / \biu$ into $\gk /\biu$. It suffices to show that there are no non-trivial extensions between any $i_{w!}i_w^*\sF=\sF_w$ for $w \in \extw^\beta_{n}$, that is $\Ext^1(\sF_{w_1}, \sF_{w_2})=0$ for $w_1 \neq w_2 \in \extw^\beta_{n}$. It follows from Proposition \ref{Vermavan} and property (2) in Definition \ref{definemono}. The proof is similar for the mixed case. Because of Lemma \ref{mixnonmix}and Proposition \ref{Vermavan}, we know that $\ext^1(\sF_{w_1}, \sF_{w_2})=0$ for $w_1 \neq w_2$. 
\end{proof}	

\begin{prop}\label{deccan}
	Let $\sF$ be a monostalk complex with respect to $\beta$. There exists a unique decomposition of $\sF_n$ into a direct sum of shifted perverse sheaves $\oplus_{w \in \extw^\beta_{n}} \sF^{can}_w$ such that each $\sF^{can}_w$ is isomorphic to $\sF_w$ for any $w \in \extw^\beta_{n}$.
\end{prop}

\begin{proof}
	We first note that property $(2)$ in Definition \ref{definemono} implies that $\sF_n$ is a shifted perverse sheaf. By Proposition \ref{dissupp}, we know that $\sF_n$ is isomorphic to the direct sum of $\sF_w$ for $w \in \extw^\beta_{n}$. 
	
	We recall the following fact about splittings in a general abelian category: for any objects $A$ and $B$ in an abelian category, the set of splittings of a trivial extension $0 \ra A \ra E \ra B \ra 0$ is a $\Hom(B,A)$-torsor, where $E$ is isomorphic to the trivial extension $A \oplus B$. 
	
	To prove that there is a unique splitting of $\sF_n$ as the desired direct sum, it suffices to prove that $\Hom^0(\sF_{w_1},\sF_{w_2})$ vanishes for any distinct $w_1, w_2 \in \extw^\beta_{n}$. This follows from Proposition \ref{Vermavan}. 
	
	The proof is similar for the mixed case.
\end{proof}

\begin{lem}\label{injproj}
	Let $w \in \beta$ such that $\ell_{\beta}(w)=n$. Let $S$ be the set $\{w' \in cl(\extw^\beta_{\le n}): w' \ge w\}$. Let $j_{\ge w}$ be the open embedding of $\bigcup_{w' \in S}\gkwp /\biu$  into $\bigcup_{w' \in cl(\extw^\beta_{\le n})}\gkwp /\biu$. The natural morphism  $j_{\ge w !}j_{\ge w}^*\sF_n \ra \sF_n$ can be regarded as the embedding of $\oplus_{w' \in S \cap \extw^\beta_{n}} \sF^{can}_{w'}$ into $\sF_n$. Similarly, let $T$ be the set $\{w' \in cl(\extw^\beta_{\le n}): w' \le w\}$ and let $i_{\le w}$ be the closed embedding of $\bigcup_{w' \in T}\gkwp /\biu$  into $\bigcup_{w' \in cl(\extw^\beta_{\le n})}\gkwp /\biu$. The natural morphism $\sF_n \ra i_{\le w!} i_{\le w}^*\sF_n$ can be regarded as the projection map of $\sF_n$ onto $\oplus_{w' \in T \cap \extw^\beta_{n}} \sF^{can}_{w'}$.

\end{lem}

\begin{proof}
	For $w' \in \beta$ such that $\ell_{\beta}(w')=n$, the shifted perverse sheaf $j_{\ge w !}j_{\ge w}^*\sF_{w'}$ is isomorphic to $\sF_{w'}$ if $w' \ge w$, and it is zero otherwise. Similarly, the shifted perverse sheaf $i_{\le w!} i_{\le w}^*\sF_{w'}$ is isomorphic to $\sF_{w'}$ if $w' \le w$, and it is zero otherwise.
\end{proof}

\begin{lem}\label{extinj}
	Let $\sF$ be a monostalk complex with respect to $\beta$. For $j \le n-1$, the group $\Ext^{j}(\sF_{\le n-j-1},\sF_n)$ vanishes. Hence, the natural map $\Ext^j(\sF_{\le n-j}, \sF_n) \ra \Ext^j(\sF_{n-j}, \sF_n)$ is injective. The analogous statement for the mixed case holds, that is $\ext^{j}(\sF_{\le n-j-1},\sF_n)$ vanishes.
\end{lem}

\begin{proof}
	For every positive integer $m$, there exists an open-closed distinguished triangle $\sF_m \ra \sF_{\le m} \ra \sF_{\le m-1} \xra{[1]}$. It implies that $\sF_{\le m}$ is a successive extension of $\sF_k$ for $k \le m$. Since $\sF_m$ is of perverse degree $m$, the first statement follows from Lemma \ref{ext}. By applying $\Hom(-,\sF_n)$ to the distinguished triangle $\sF_{n-j} \ra \sF_{\le n-j} \ra \sF_{\le n-j-1} \xra{[1]}$, we obtain an exact sequence $\Ext^j(\sF_{\le n-j-1}, \sF_n) \ra \Ext^j(\sF_{\le n-j}, \sF_n) \ra \Ext^j(\sF_{ n-j}, \sF_n)$. The second statement follows from the vanishing of the group $\Ext^j(\sF_{\le n-j-1}, \sF_n)$. The mixed case follows from Lemma \ref{mixnonmix}. 
\end{proof}	

The distinguished triangle $\sF_n \ra \sF_{\le n} \ra \sF_{\le n-1} \xra{[1]}$ gives an element $\sE_n:={[\sF_{\le n}]}$ in the group $\Ext^1(\sF_{\le n-1}, \sF_n)$. By applying the functor $\Hom(-,\sF_n)$ to the triangle corresponding to $\sE_{n-1}$, we obtain the first row of the following diagram, which is an exact sequence.
\[\begin{tikzcd}[column sep = small]
	\ \Ext^1(\sF_{\le n-2}, \sF_n) \arrow[r] \arrow[d, equal,"Lemma \ref{extinj}"]
	& \Ext^1(\sF_{\le n-1}, \sF_n) \arrow[r, "\iota"]
	& \Ext^1(\sF_{n-1}, \sF_n) \arrow[r, "\partial"] \arrow[d,equal, "Proposition \ref{deccan}"] 
	& \Ext^2(\sF_{\le n-2}, \sF_n) \arrow[d, hook, "\phi"',"Lemma \ref{extinj}" ]  \\
	0 
	& \sE_n={[\sF_{\le n}]} \arrow[u, phantom, sloped, "\in"]
	& \bigoplus \Ext^1(\sF^{can}_{w'}, \sF^{can}_{w}) 
	& \Ext^2(\sF_{n-2}, \sF_n)
\end{tikzcd}\]

Here, the direct sum runs through all pairs $(w',w)$ in $V_n^{\beta}$. The first term vanishes because of Lemma \ref{extinj}. The third term is the direct sum of extension groups of individual shifted standard sheaves because of Proposition \ref{deccan}. By Lemma \ref{onedim}, each summand is a  one-dimensional $\ql$-vector space. The injectivity of the last vertical map $\phi$ follows directly from the second statement in Lemma \ref{extinj}. The boundary map $\partial$ is the cup product with the extension class $\sE_{n-1}$ in $\Ext^1(\sF_{\le n-2}, \sF_{n-1})$.

The injectivity of $\phi$ implies the following sequence is exact. 
\[0 \ra \Ext^1(\sF_{\le n-1}, \sF_n) \xra{\iota} \Ext^1(\sF_{n-1}, \sF_n) \xra{\phi \circ \partial} \Ext^2(\sF_{n-2}, \sF_n). \label{eq:fundamental} \tag{$\spadesuit$}\]

\begin{lem} \label{cup2}
	Let $\sE_{n-2,n-1}:=[ \sF_{\{n-2,n-1\}}] \in \Ext^1(\sF_{n-2}, \sF_{n-1})$ be the extension class corresponding to the open-closed distinguished triangle $\sF_{n-1} \ra \sF_{\{n-2,n-1\}} \ra \sF_{n-2} \xra{[1]}$. The composition map $\phi \circ \partial$ in \eqref{eq:fundamental} is the cup product with $\sE_{n-2,n-1}$. 
\end{lem}

\begin{proof}
	Let $j$ be the open embedding of ${\bigcup_{w \in \extw^\beta_{n-2} \cup \extw^\beta_{n-1}}{\gkw /\biu}}$ into ${\bigcup_{w \in  cl(\extw^\beta_{\le n-1})}{\gkw /\biu}}$. There is a natural transformation $T$ from the functor $j_!j^*$ to the identity functor $Id$. We observe that $j_!j^*\sF_{\le n-1}$ is canonically isomorphic to $\sF_{\{n-2,n-1\}}$. Likewise, we know that $j_!j^*\sF_{\le n-2}$ is canonically isomorphic to $\sF_{\le n-2}$. Hence, we have the following distinguished triangle morphism, 
	\[\begin{tikzcd}
		\	
		\sF_{n-1} \arrow[r] \arrow[d, equal]
		& \sF_{\{n-2,n-1\}} \arrow[r] \arrow[d,"T"]
		& \sF_{n-2} \arrow[r, "\sE_{n-2,n-1}"] \arrow[d, "T"] 
		& \sF_{n-1}{[1]} \arrow[d, equal]  \\
		\sF_{n-1} \arrow[r]
		& \sF_{\le n-1} \arrow[r]
		& \sF_{\le n-2} \arrow[r, "\sE_{n-1}"]
		& \sF_{n-1}{[1]}.
	\end{tikzcd}\]

The map $\phi$ is induced by $T$ in our construction, while the map $\partial$ is induced by $\sE_{n-1}$. The lemma is a consequence of the commutativity of the rightmost square in the diagram.
\end{proof}

\begin{defn}\label{decompcan2}
Let $\sF$ be a monostalk complex with respect to $\beta$. According to Proposition \ref{deccan}, the group $\Ext^1(\sF_{n-1},\sF_{n})$ is canonically isomorphic to $\bigoplus_{(w',w) \in V^{\beta}_{n}} \Ext^1(\sF_{w'}^{can},\sF_{w}^{can})$. We define $\pi_{w',w}=\pi_{w',w}^{\sF}$ as the projection map from $\Ext^1(\sF_{n-1},\sF_{n})$ onto the $(w',w)$-component $\Ext^1(\sF_{w'}^{can},\sF_{w}^{can})$ under this isomorphism. 
\end{defn}

\begin{lem}\label{restsimple}
	Let $w \in \beta$ and $s$ be a simple reflection in $\extw$ such that $s \in \ewlo$ and $ws>_\beta w$. Let $n:=\ell_{\beta}(ws)$. We denote $\sE_{w,ws}$ as the extension class in $\Ext^1(\sF_w,\sF_{ws})$ representing the open closed distinguished triangle $\sF_{ws} \ra \sF_{\{w,ws\}} \ra \sF_{w}$. Then $\pi_{w,ws}\sE_{n-1,n}$ is non-zero if and only if $\sE_{w,ws}$ is non-zero.
\end{lem}	

\begin{proof}
We use the notation of Lemma \ref{injproj}. Let $i$ be the closed embedding $i_{\le ws}$ and $j$ be the open embedding $j_{\ge w}$.
We have the following distinguished triangle morphisms, 
\[\begin{tikzcd}
	\	
		\sF_{n} \arrow[r]
	& \sF_{\{n-1,n\}} \arrow[r]
	& \sF_{n-1} \arrow[r, "\sE_{n-1,n}"]
	& \sF_{n}{[1]} \\
		j_!j^*\sF_{n} \arrow[r] \arrow[u] \arrow[d]
	& j_!j^*\sF_{\{n-1,n\}}  \arrow[r] \arrow[u] \arrow[d]
	& j_!j^*\sF_{n-1} \arrow[r]  \arrow[u] \arrow[d]
	& j_!j^*\sF_{n}{[1]} \arrow[u] \arrow[d] \\
		i_!i^*j_!j^*\sF_{n} \arrow[r]
	& i_!i^*j_!j^*\sF_{\{n-1,n\}} \arrow[r]
	& i_!i^*j_!j^*\sF_{n-1} \arrow[r]
	& i_!i^*j_!j^*\sF_{n}{[1]}.
\end{tikzcd}\]
We observe that the last row is canonically isomorphic to the open closed distinguished triangle $\sF_{ws} \ra \sF_{\{w,ws\}} \ra \sF_{w}$ because there does not exist an element $u$ in $\extw$ such that $w<u<ws$. 
The morphisms induce a correspondence \[\Hom(\sF_{n-1},\sF_n[1]) \la \Hom(j_!j^*\sF_{n-1}, j_!j^*\sF_n[1]) \ra \Hom(i_!i^*j_!j^*\sF_{n-1},i_!i^*j_!j^*\sF_n[1]).\] By Proposition \ref{deccan}, we can identify $\sF_w$ and $\sF_{ws}$ with $\sF_w^{can}$ and $\sF_{ws}^{can}$ up to a non-zero scalar. The correspondence and Lemma \ref{injproj} imply $\pi_{w,ws} \sE_{n-1,n}$ and $\sE_{w,ws}$ can be identified up to a non-zero scalar.
\end{proof}

The following lemma about the open closed distinguished triangle is well-known.

\begin{lem}\label{openclo}
	Let $X$ be a variety. Let $j:U \inj X$ be an open embedding and let $i:Z \inj X$ be the closed embedding of the complement of $U$ in $X$. Suppose $\sF$, $\sG$, and $\sH$ are three complexes on $X$ such that $i^*\sG=0$, $j^!\sH=0$, and $\sG \ra \sF \ra \sH \xra{[1]}$ is a distinguished triangle. Then the distinguished triangle is canonically isomorphic to $j_!j^*\sF \ra \sF \ra i_!i^* \sF \xra{[1]}$. The statement also holds for equivariant complexes and mixed complexes.
\end{lem}

\begin{proof}
	By applying the natural transformations $j_!j^* \ra Id \ra i_!i^*$ to the triangle $\sG \ra \sF \ra \sH \xra{[1]}$, we have the following distinguished triangle morphisms, 
	\[\begin{tikzcd}
		\	
		j_!j^*\sG \arrow[r,"\iota"]\arrow[d]
		& j_!j^*\sF \arrow[r]\arrow[d]
		& 0 \arrow[r] \arrow[d]
		& j_!j^*\sG{[1]} \arrow[d] \\
		\sG \arrow[r]  \arrow[d]
		&\sF \arrow[r]  \arrow[d]
		& \sH\arrow[r]   \arrow[d]
		& \sG{[1]} \arrow[d] \\
		0 \arrow[r]
		& i_!i^*\sF \arrow[r, "\pi"]
		& i_!i^*\sH \arrow[r]
		& 0.
	\end{tikzcd}\]
	We obtain a natural map from $j_!j^*\sF$ to $\sG$ by inverting $\iota$. Similarly, we obtain a natural map from $\sH$ to $i_!i^*\sF$ by inverting $\pi$. It is clear that these two maps are isomorphisms. Hence we get a canonical isomorphism between the triangles. 
\end{proof}

\begin{prop}\label{minfilt}
	Let $\sF$ be a monostalk complex with respect to $\beta$. Let $\xi$ be a minimal IC sheaf of the block $\gamma$. There exists a canonical isomorphism $\phi$ between the following distinguished triangles for any $n \ge 0$.
		\[\begin{tikzcd}
		\	
		\sF_{n}\star \xi \arrow[r] \arrow[d,"\phi_{n}"]
		& \sF_{\le n} \star \xi \arrow[r] \arrow[d,"\phi_{\le n}"]
		& \sF_{\le n-1} \star \xi \arrow[r, "\sE_{n} \star \xi"] \arrow[d,,"\phi_{\le n-1}"] 
		& \sF_{n}\star \xi {[1]}  \arrow[d]  \\
		(\sF \star \xi)_{n} \arrow[r]
		& (\sF \star \xi)_{\le n} \arrow[r]
		& (\sF \star \xi)_{\le n-1} \arrow[r, "(\sE \star \xi)_{n}"]
		& (\sF \star \xi)_{n}{[1]}.
	\end{tikzcd}\]
\end{prop}

\begin{proof}
	%For any integer $t$, let $\extw_t$ be the union $\bigcup_{\beta} \extw_t^{\beta}$ and $\extw_{\le t}$be the union $\bigcup_{t' \le t} \extw_t$. We observe that $\extw_t$ is stable under right multiplication of any minimal elements.
	For every integer $m$ greater than $n$, we have the distinguished triangle $\sF_{(n,m]} \ra \sF_{\le m} \ra \sF_{\le n}\xra{[1]}$. Hence $\sF_{(n,m]}\star \xi \ra \sF_{\le m}\star \xi \ra \sF_{\le n}\star \xi \xra{[1]}$ is a distinguished triangle. There is a "filtration" of $\sF_{(n,m]}\star \xi$ with subquotient $\sF_u \star \xi$ for $u \in \bigcup_{t \in (n,m]}\extw_t^{\beta}$. Since $\omega (\sF_u \star \xi)$ is isomorphic to $\Delu(uw^{\gamma})$, the $*$-stalk of complex $\sF_{(n,m]}\star \xi$ at the $w$-stratum is zero for any $w \in cl(\extw_{\le n}^{\beta\gamma})$. Similarly, the $!$-stalk of $\sF_{\le n}\star \xi $ at the $w$-stratum vanishes for any $w \in cl(\extw_{\le n}^{\beta\gamma})$. From Lemma \ref{openclo}, we know that $(\sF_{\le m}\star \xi)_{\le n}$ is canonically isomorphic to $(\sF_{\le n}\star \xi)$. Because of their canonicity, these isomorphisms are compatible for different $m$. As a result, we obtain that an canonical isomorphism from $ \sF_{\le n} \star \xi$ to $(\sF \star \xi)_{\le n}$. The statement follows from Lemma \ref{openclo}.	
\end{proof}

\begin{cor}\label{transport}
	We use the notation of Proposition \ref{minfilt}. Let $\sG$ be $\sF \star \xi$. The composite $\phi \circ \star \xi: \Ext^1(\sF_{n-1},\sF_n) \ra \Ext^1(\sF_{n-1} \star \xi, \sF_n \star \xi) \ra \Ext^1(\sG_{n-1}, \sG_{n})$ is an isomorphism. Moreover, it sends $\Ext^1(\sF_{w'}^{can},\sF_w^{can})$ to $\Ext^1(\sG _{w'w^{\gamma}}^{can},\sG_{ww^{\gamma}}^{can})$. The  statement holds when replacing $\ext^1$ with $\Ext^1$.
\end{cor}

\begin{proof}
	It follows from the argument in Proposition \ref{minfilt}.
\end{proof}

\subsection{Definition of glueable complexes}

\begin{defn}\label{defineglue}
	Let $\sF$ be a non-mixed monostalk complex with respect to $\beta$. We call $\sF$ to be a \emph{glueable complex} with respect to $\beta$ if it satisfies the following property:
	\begin{enumerate}
		\item[(3)] For any $w \in \beta$ and simple reflection $s \in \ewlo$ (not necessarily a simple reflection in $\extw$) such that $ws>_\beta w$, the composite $\pi_{w,ws} \circ \iota$ sends $\sE_{n}$ to a non-zero element in $\Ext^1(\sF_{w}^{can},\sF_{ws}^{can})$, where $\iota$ is the map stated in \eqref{eq:fundamental}.
	\end{enumerate}
\end{defn}

We say that $\sF$ is a glueable complex if $\omega \sF$ is a glueable complex in the non-mixed category.

\begin{exmp}
	When $\sL$ is trivial, the constant sheaf on a connected component of affine flag variety is glueable with respect to the block corresponding to that component. 
\end{exmp}

\subsection{Characterization of averaging maximal IC sheaves via glueable complexes}

\begin{thm}\label{avglue}
	$\Av_!\Xi$ is the unique (up to tensoring with a one-dimensional Fr-module) glueable complex and $\omega \Av_!\Xi$ is the unique glueable non-mixed complex with respect to $\beta$, where $\Xi$ is a maximal IC sheaf for $\beta$.
\end{thm}

The proof of this theorem follows from Sections \ref{sectavglue} and \ref{sectglueav}.

\subsection{Averaging maximal IC sheaves are glueable} \label{sectavglue}

\begin{prop}\label{ICglue}
	For any maximal IC sheaf $\Xi$ for $\beta$,  $\Av_!\Xi$ is glueable with respect to $\beta$.
\end{prop}

\begin{proof}
	By the definition of glueable complexes, it is sufficient to consider the non-mixed case. We know that $\Av_!\Xi$ is a monostalk complex by Proposition \ref{avstalk}. Let us now check $\sF$ satisfies property $(3)$ in Definition \ref{defineglue}. Let $w$ and $s$ be as in property (3). From Lemma \ref{conj}, there exists $u \in \extw$ such that $u$ is a minimal element and $u^{-1}su$ is a simple reflection in $\extw$. We write $u=s_{i_1}s_{i_2} \dots s_{i_m}$ as a reduced word expression. Lemma \ref{equiv} implies $\Xi \star \ICu(s_{i_1})$ is a maximal IC sheaf. Inductively, we know that $\Xi \star \ICu(s_{i_1}) \star \ICu(s_{i_2}) \dots \star \ICu(s_{i_m})$ is a maximal IC sheaf. We can identify $(\Av_! \Xi) \star \ICu(u)$ with $\Av_! (\Xi \star \ICu(u))$ canonically by Corollary \ref{avass}. By Corollary \ref{transport}, we know that property (3) holds for the tuple $(\sF=\Av_! \Xi, w, s)$ if and only if it holds for $(\sF=\Av_! (\Xi \star \ICu(u)), wu, u^{-1}su)$. Therefore we may and will assume $s$ is a simple reflection in $\extw$. Thanks to Lemma \ref{restsimple}, it suffices to show that $(\Av_! \Xi)_{\{w,ws\}}$ is a non-trivial extension of $(\Av_! \Xi)_{ws}$ and $(\Av_! \Xi)_{w}$.
	
	We define $\lpduel$ to be the limit of the projective system of triangulated categories $D^b_{(T \times L_s, \sL' \boxtimes \extsL^{-1}),m}(\biu \bsl \gksw / \bius)$ with $*$-pullback functors (see Definition \ref{twolimit}). Then we can define an averaging functor $\Av_{!,s}$ from $\lpdelm$ to $\lpduel$ by the construction similar to $\Av_!$. Let $\pi_s^{*}$ (resp. $\pi_s^{+,*}$) be the pullback functor from $\lpdelm$ to $\lpdlm$ (resp. $\lpduel$ to $\lpdul$).
	
	Lemma \ref{pbic} implies that $\Xi \cong \omega \pi_{s}^* \Xi'$ for some $\Xi' \in \pha \lpdelm$. By the base change theorem, we know that $\Av_! \pi_s^* \Xi'$ is canonically isomorphic to $\pi_s^{+,*} \Av_{!,s} \Xi'$. Therefore $\Av_!\Xi$ is a pullback from a (projective limit of) complex on $\gk /P_s$. Therefore $(\Av_! \Xi)_{\{w,ws\}}$ is a non-trivial extension of $(\Av_! \Xi)_{ws}$ and $(\Av_! \Xi)_{w}$.
\end{proof}

\subsection{Glueable objects are averaging maximal IC sheaves}\label{sectglueav}

We apply Definition \ref{defanti} to the Coxeter group $\ewlo$ and obtain the notion of anti-commutative functions on $V^\circ$. The concept can be generalized to non-neutral blocks. The definition of $V^\circ$ and $V^{\beta}$ are stated in Definition \ref{defv}.

\begin{defn} \label{defanti2}
	The left multiplication of $w^{\beta}$ induces a bijection between $V^\circ$ and $V^{\beta}$. We refer to a $\ql$-valued function $f$ on $V^{\beta}$ as \emph{anti-commutative} if it exhibits anti-commutativity under this bijection.  In other words, let $g$ be the function on $V^\circ$ such that $g(v,u):=f(w^{\beta,-1}v,w^{\beta,-1}u)$. The function $f$ is considered anti-commutative if and only if $g$ is.
\end{defn}

\begin{lem}
	We obtain the same definition of anti-commutative function on $V^{\beta}$ if we use the bijection induced by the right multiplication of $w^{\beta}$.
\end{lem}

\begin{proof}
	It follows from Lemma \ref{conjmin}.
\end{proof}

Let $\sF$ be a non-mixed monostalk sheaf with respect to $\beta$. Therefore $\sF_w$ is isomorphic to $\Delu (w)_\sL[-\ell_\beta(w)]$ for $w \in \beta$. 

For $w \in \beta$, the automorphism group of $\Delu(w)$ is $\ql^*$. We choose an identification of the stalk of $\Delu(w)$ at any point in $\gkw /\biu$ with $\ql$. Then we identify $\Ext^1(\sF_{n-1}, \sF_n) \cong \bigoplus_{(w',w) \in V^{\beta}_{n}} \Ext^1(\sF_{w'}^{can},\sF_{w}^{can})$ as the space of $\ql$-valued functions on $V_n^{\beta}$. Different identifications of stalks change the function by conjugation, that is $f(v,u)$ transforms to $h(v)^{-1}f(v,u)h(u)$ for some everywhere non-zero function $h$. Similarly, $\Ext^2(\sF_{n-2},\sF_n)$ can be viewed as the space of functions on the set $\{(v,u): \ell_{\beta}(u)=\ell_{\beta}(v)+2 $ and $u >_{\beta} v\}$. 

Via the exact sequence \eqref{eq:fundamental}, we can regard the extension class $\sE_{n} \in \Ext^1(\sF_{\le n-1}, \sF_n)$ as a function $f_n$ on $V_n^{\beta}$. We define $f$ to be the function on $V^{\beta}$ whose restriction on $V_n^{\beta}$ coincide with $f_n$. 

\begin{prop}\label{monoanti}
	Let $\sF$ be a non-mixed monostalk sheaf with respect to $\beta$. The corresponding function $f$ is anti-commutative. 
\end{prop}

\begin{proof}
	Lemma \ref{cup2} states that the map $\phi \circ \partial$ in \eqref{eq:fundamental} is the cup product with $\sE_{n-2,n-1} \in \Ext^1(\sF_{n-2}, \sF_{n-1})$. Since the cup product map can be regarded as homomorphism between perverse sheaves, it is well-behaved with respect to the direct summand. Let $u,v \in \beta$ such that $\ell_{\beta}(u)=\ell_{\beta}(v)+2$ and $u >_{\beta} v$. Let $x_1,x_2$ be the two elements such that $u>_{\beta}x_i>_{\beta}v$ for $i=1,2$. The exactness of \eqref{eq:fundamental} implies that 	$\pi_{v,x_1}(\sE_{n-2,n-1})\pi_{x_1,u}(\sE_{n-1,n})+\pi_{v,x_2}(\sE_{n-2,n-1})\pi_{x_2,u}(\sE_{n-1,n})$ vanishes, where $\pi_{\cdot, \cdot}$ is defined in Definition \ref{decompcan2}. Because of Lemma \ref{cupinj}, we know that the function $f_{n-1}(v,x_1)f_n(x_1,u)+f_{n-1}(v,x_2)f_n(x_2,u)$ vanishes. Hence the function $f$ is anti-commutative. 	
\end{proof}

\begin{lem}\label{everynonzero2}
	Let $\sF$ be a non-mixed glueable complex. The corresponding function $f$ is everywhere non-zero.
\end{lem}

\begin{proof}
	By definition of glueable complex, we know that $f(w,ws)$ is non-zero when $s$ is a simple reflection in $\ewlo$ and $w<_{\beta}ws$. Proposition \ref{monoanti} implies $f$ is anti-commutative. The statement follows from Lemma \ref{everynonzero}.
\end{proof}

\begin{prop}
Any two non-mixed glueable complexes $\sF$ and $\sG$ in the block $\beta$ are isomorphic. In particular, $\sF$ is isomorphic to $\omega \Av_! \Xi$, where $\Xi$ is a maximal IC sheaf in $\beta$. 
\end{prop}

\begin{proof}
We construct a family of isomorphisms $\phi_{\le n}: \sF_{\le n} \ra \sG_{\le n}$ by induction. For $n=0$, we choose an isomorphism from $\sF_0$ to $\sG_0$. Suppose now we have an isomorphism $\phi_{\le n-1}: \sF_{\le n-1} \ra \sG_{\le n-1}$. Consider the following diagram of distinguished triangles, 

\[\begin{tikzcd}
	\sF_{\le n-1}[-1] \arrow[d, "\phi_{\le n-1} \text{[-1]}"] \arrow[r,"\sE_n^{\sF}"] & \sF_n \arrow[d, dashed] \arrow[r] & \sF_{\le n} \arrow[d, dashed] \arrow[r,"{[1]}"] &  \pha\\
	\sG_{\le n-1}[-1] \arrow[r,"\sE_n^{\sG}"]& \sG_n \arrow[r] & \sG_{\le n} \arrow[r,"{[1]}"] & \pha.
\end{tikzcd}\]

%The isomorphism class of $\sF_{\le n}$ is determined by the map $\sE_n^{\sF}$.

Notice that $\sF_n$ and $\sG_n$ are canonically isomorphic to direct sums $\oplus_{w \in \extw_n^{\beta}} \sF_w^{can}$ and $\oplus_{w \in \extw_n^{\beta}} \sG_w^{can}$ respectively. In particular, they are both isomorphic to $\oplus_{w \in \extw_n^{\beta}} \Delu(w)$. 

Let $f$ and $g$ be the functions on $V^{\beta}$ corresponding to $\sF$ and $\sG$ respectively. The extension class of $\sF_{\le n}$ is determined by the map $\sE_n^{\sF}$, which, in turn, is determined by the function $f$ up to conjugacy. The function $f$ is not uniquely determined due to the presence of the non-trivial automorphism group of $\sF_w^{can}$, which is isomorphic to $\ql^*$. Likewise, the extension class of $\sG_{\le n}$ is determined by the function $g$ up to conjugacy. Lemma \ref{conjfcn} and Lemma \ref{everynonzero2} imply that $f$ is conjugated to $g$. Hence, we can construct a family of maps $\phi_w$ from $\sF_w^{can}$ to $\sG_w^{can}$ such that their direct sum $\phi := \oplus \phi_w: \sF_n \ra \sG_n$ ensures the commutativity of the first square. By the axioms of triangulated categories, we can construct a (non-canonical) isomorphism $\phi_n: \sF_{\le n} \xra{\sim} \sG_{\le n}$ extending $\phi_{\le n-1}$. 
\end{proof}

The proof for glueable sheaves in the mixed categories is similar, involving applying Lemma \ref{mixnonmix} multiple times. 

%\begin{lem}
%	Let $w_1, w_2 \in \ewlo$ such that $\ell_\sL(w_2)=\ell_\sL(w_1)+1$ and $w_1 <_{\ewlo} w_2$ (the Bruhat order for $\ewlo$). Let $M$ be a one-dimensinoal Fr-module. The dimension of  $\hom(\Delta(\dot{w_1})_\sL, \Delta(\dot{w_2})_\sL \otimes M)$ is at most one. Moreover, there exists a unique (up to isomorphism) $M$ such that the dimension of $\hom^0(\Delta(\dot{w_1})_\sL, \Delta(\dot{w_2})_\sL \otimes M)$ is one. 
%\end{lem}

%\begin{lem}
%	Suppose $\sF \in \pha \lpdl$ such that its Frobenius pullback $\omega \sF$ is isomorphic to $\underline{\Delta}(w)_\sL$, there exists a one-dimensinoal Fr-module $M$ such that $\sF \cong \Delta(\wdo)_\sL \otimes M$. 
%\end{lem}

%\begin{proof}
%	From the exact sequence (5.1.2.5) in \cite{BBD}, we have $\hom(\sF, \Delta(\wdo)_\sL)=\Hom(\omega \sF, \Delu(w)_\sL)^{\Fr}=\ql^{\Fr}$ . Take $M$ to be the dual of $\hom(\sF, \Delta(\wdo)_\sL)$.
%\end{proof}

\section{Convolution products of maximal sheaves}\label{conmax}

In this section, we construct morphisms between convolution products of maximal sheaves.

\begin{lem}\label{constd}
	Let $\beta \in \pha _{\sL} \extw _{\sL''}$ and $w \in \beta$. The convolution product $\lptulm \conn \pha _{\sL}\underline{\Delta}(w)_{\sL''}^+$ is canonically isomorphic to $\lptulmpp [-\ell_{\beta}(w)]$. The convolution product $\lptulm \conn \pha _{\sL}\nabu(w)_{\sL''}^+$ is canonically isomorphic to $\lptulmpp [\ell_{\beta}(w)]$.
\end{lem}

\begin{proof}
	We obtain the second statement by applying Verdier duality to the first statement. By Proposition \ref{minconv}, we may assume $w$ is in the affine Weyl group. By expressing $_{\sL}\underline{\Delta}(w)_{\sL''}^+$ as the product of positive standard sheaves, it reduces to the case when $w=s$ is a simple reflection in $\extw$. If $s \notin \ewlo$, then $s$ is the minimal element in its block. The first statement follows from Lemma \ref{equiv}. If $s \in \ewlo$, then the first statement follows from Lemma \ref{pbic}, the proper base change theorem and the fact that $\tH^*_c(\A^1)=\ql[2]$. 
\end{proof}

When $s \in \ewlo$, there is a distinguished triangle $ \underline{\Delta}(s)_\sL^+[-1] \ra \underline{\IC}(s)_\sL^+[-1] \ra \underline{\delta}_\sL \xra{[1]}$. To see this, we may assume that $s$ is a simple reflection of $\extw$ by Lemma \ref{conj}. Then the triangle becomes a standard one. Taking the long exact sequence for the convolution of $\lptulm$ with $ \underline{\Delta}(s)_\sL^+[-1] \ra \underline{\IC}(s)_\sL^+[-1] \ra \underline{\delta}_\sL \xra{[1]}$ shows that the restriction map $^pH^0(\lptulm \conn \underline{\IC}(s)_\sL^+[-1]) \ra \pha ^pH^0(\lptulm \conn \underline{\delta}_\sL)$ is an isomorphism.

\begin{lem}\label{unit}
There is a unique map $\alpha$ from $\ltlm$ to $\avltlm$ making the following diagram commute. 
\[\begin{tikzcd}%[row sep=large]
\ltlm \arrow[d, "\epsilon_\sL"] \arrow[r, "\alpha"] & \avltlm \arrow[d, "\epsilon_\sL "]\\
\delta_\sL \arrow[r, "="]& \delta_\sL
\end{tikzcd}\]	
\end{lem}

\begin{proof}
	We observe that $\avltlm$ is isomorphic to the pro-object $\ulalim i_{\le v*}i_{\le v}^*(\avltlm)$. By definition, a map from $\ltlm$ to the pro-object $\avltlm$ is equivalent to a family of compatible maps from $\ltlm$ to $i_{\le v*}i_{\le v}^*(\avltlm)$ for each $v$, which is the same as maps from $\avltlm$ to $i_{\le v*}i_{\le v}^*(\avltlm)$ by Corollary \ref{Avfor}. The adjoint pair $(i_{\le v*}, i_{\le v}^*)$ induces a canonical morphism from $\avltlm$ to $i_{\le v*}i_{\le v}^*(\avltlm)$. These morphisms are compatible and we let $\alpha : \pha \ltlm \ra \avltlm$ to be the corresponding morphism. By construction, $\alpha$ completes the commutative diagram. 

To prove uniqueness, it suffices to show that $\hom^0(\pha \ltlm, \avltlm)=\ql$. For $m \in \Z$ and $w \in \ewlo$, we have
\begin{align*}
	&\Ext^m(\pha \ltulm, i_{w!}i_w^*\avltulm) &\\ 
	=&  \Ext^m(\pha \ltulm, \underline{\Delta}(w)[-\ell_\sL(w)]) && \text{Proposition } \ref{avstalk}\\
	=&\Ext^m(\pha \ltulm \conn  \underline{\nabla}(w^{-1}), \delta_\sL[-\ell_\sL(w)]) && \text{Lemma } \ref{inverse}\\
	=&\Ext^m(\pha \ltulm [\ell_\sL(w)], \underline{\delta}_\sL[-\ell_\sL(w)]) && \text{Lemma } \ref{constd}\\
	=&\Ext^m(\underline{\delta}_\sL , \underline{\delta}_\sL[-2\ell_\sL(w)]). && \text{Proposition } \ref{stalk} 
\end{align*}

Because of Lemma \ref{ext}, we know that the last term vanishes  when $w \neq e$ and $m \le 1$. The Schubert stratification gives rise to a "filtration" on $i_{\le v*}i_{\le v}^*(\avltulm)$, whose subquotients are given by $i_{w!}i_w^*\avltulm$ for $w \le v$ such that $w \in \ewlo$. Therefore, the vanishing statement for $\Ext^m(\pha \ltulm, i_{w!}i_w^*\avltulm)$ implies that $\Hom^0(\pha \ltulm, i_{\le v*}i_{\le v}^*\avltulm)$ $\cong \Hom^0(\pha \ltulm, i_{e!}i_e^*\avltulm)$. Hence, we obtain
\begin{align*}
	\Hom^0(\pha \ltulm, \avltulm) & =  \ulalim \Hom^0(\pha \ltulm, i_{\le v*}i_{\le v}^*\avltulm) \\
	&=\Hom^0(\pha \ltulm, i_{e!}i_e^*\avltulm) \\
	&=\Hom^0(\pha \ltulm, \underline{\delta}_\sL)=\ql. 
\end{align*}

Thanks to Lemma \ref{mixnonmix}, the above vanishing statement holds for mixed sheaves analog, that is $\ext^m(\pha \ltlm, i_{w!}i_w^*\avltlm)$ vanishes for $w \neq e$ and $m \le 1$. The same argument shows that $\hom^0(\pha \ltlm, \avltlm)=\ql$.
\end{proof}

The following lemma involves the object $\ltlm \conn \avltlm$. Since the stalks of $\ltlm \conn \avltlm$ are infinite-dimensional, it does not live in the category $\ldlm$. Therefore we have to regard it as a pro-object of $\ldlm$. 

\begin{lem}\label{comulti}
There is a unique map $\chi$ from $\ltlm$ to $\ltlm \conn \avltlm$ such that the composition $\ltlm \xra{\chi} \pha  \ltlm \conn \avltlm \xra{id \conn \epsilon_\sL} \pha \ltlm \conn \delta_\sL= \pha \ltlm$ is the identity map.  Moreover the composition $\ltlm \xra{\chi} \pha \ltlm \conn \avltlm \xra{\epsilon_\sL \conn id} \delta_\sL \conn \avltlm = \avltlm$ is $\alpha$.
\end{lem}

\begin{proof}
We observe that $\avltlm$ is isomorphic to the pro-object $\ulalim i_{\le v*}i_{\le v}^*(\avltlm)$. To construct a morphism from $\ltlm$ to $\ltlm \conn \avltlm$, it suffices to produce a family of compatible maps from $\ltlm$ to $\ltlm \conn i_{\le v*}i_{\le v}^*\avltlm$. Similar to the proof of Lemma \ref{unit}, we observe that $i_{\le v*}i_{\le v}^*(\avltlm)$ possesses a "filtration" with subquotients $i_{w!}i_w^*\avltlm$ for $w \le v$ such that $w \in \ewlo$. Hence, $\ltlm \conn i_{\le v*}i_{\le v}^*\avltlm$ has a "filtration" whose the subquotients are given by $\ltlm \conn i_{w!}i_w^*\avltlm$ for $w \le v$ such that $w \in \ewlo$.

Since $\omega i_{w!}i_w^*\avltlm$ is isomorphic to $\underline{\Delta}(w)[-\ell_\sL(w)]$, we know that the perverse degree of $\ltlm \conn i_{w!}i_w^*\avltlm$ is concentrated in degree $2 \ell_\sL(w)$ by Lemma \ref{constd}. Therefore the lowest perverse cohomology of $\ltlm \conn i_{\le v*}i_{\le v}^*(\avltlm)$ is $\ltlm \conn  i_{e!}i_e^*\avltlm = \pha \ltlm$ in degree zero. We define $\chi_v$ as the perverse truncation $^p \tau^{\le 0}$ applied to $\ltlm \conn i_{\le v*}i_{\le v}^*\avltlm$ for degrees less than one. By the construction, we know that $\chi_v$ is a family of maps from $\ltlm$ to $\ltlm \conn i_{\le v*}i_{\le v}^*\avltlm$. Hence we obtain a map $\chi$ from $\ltlm$ to $\ltlm \conn \avltlm$. 

Since the composite $(id \conn \epsilon_{\sL}) \circ \chi_{w}$ is the identity map when $*$-restricting at $e$, the morphism $(id \conn \epsilon_{\sL}) \circ \chi_{w}$ has to be the identity map as $\Hom^0(\ltlm, \pha \ltlm) =\ql$. Therefore the composite $(id \conn \epsilon_{\sL}) \circ \chi$ is the identity map. 

Because of Lemma \ref{ext}, we know that $\Ext^m(\ltlm, \pha \ltlm \conn i_{w*}i_{w}^*\avltlm)$ vanishes for $w \neq e$ and $m=0,1$. Therefore $\Hom^0(\ltlm, \ltlm \conn i_{\le v*}i_{\le v}^*\avltlm)$ is canonically isomorphic to $\Hom^0(\ltlm, \pha \ltlm \conn i_{e*}i_{e}^*\avltlm)$, which is a one-dimensional $\ql$-vector space. By requiring the composite morphism to be the identity, the map $\chi$ has to be unique. 

The second statement follows from the fact that $\hom^0(\pha \ltlm, \avltlm)=\ql$, which has been shown in the proof of Lemma \ref{unit}. 
\end{proof}

\begin{prop}\label{triple}
There is a natural isomorphism $\gamma: (\pha \ltlm \conn \avltlm) \conn \avltlm \ra \pha \ltlm \conn \Av_!( \pha \ltlm \conn  \avltlm)$ such that the following diagram commutes. 

\[\begin{tikzcd}[column sep=large]
\ltlm \arrow[d, "id"] \arrow[r, "\chi"] & \ltlm \conn \avltlm \arrow[r, "\chi \conn id "] & (\pha \ltlm \conn \avltlm) \conn \avltlm \arrow[d, "\gamma"] \\
\ltlm \arrow[r, "\chi"] & \ltlm \conn \avltlm \arrow[r, "id \conn \Av_!\chi"] &  \ltlm \conn \Av_!( \pha \ltlm \conn  \avltlm)
\end{tikzcd}\]	

Here we view all objects as pro-objects of $\ldlm$ in $pro_{\extw \times \extw}(\ldlm)$. 
\end{prop}

\begin{proof}

By Corollary \ref{avass} and taking limit, we know that $\avltlm \conp \avltlm$ is canonically isomorphic to $\Av_!( \pha \ltlm \conn \avltlm)$. Since the convolution product is associative, $(\pha \ltlm \conn \avltlm) \conn \avltlm$ is canonically isomorphic to $\pha \ltlm \conn (\avltlm \conp \avltlm)$. Therefore there is a natural isomorphism $\gamma: (\pha \ltlm \conn \avltlm) \conn \avltlm \ra \pha \ltlm \conn \Av_!( \pha \ltlm \conn  \avltlm)$. In particular, both complexes are isomorphic to $(\ltlm \conn \ulalim_{v} i_{\le v*}i_{\le v}^*\avltlm ) \conn \ulalim_{v'} i_{\le v'*}i_{\le v'}^*\avltlm$. For a fixed pair $(v,v')$, the last pro-complex has a "filtration" whose subquotients are given by $(\ltlm \conn i_{w*}i_{w}^*\avltlm) \conn i_{w'*}i_{w'}^*\avltlm$ for $w \le v$ and $w' \le v'$ such that $w,w' \in \ewlo$. By applying Lemma \ref{constd} twice, we know that the subquotient $(\ltlm \conn i_{w*}i_{w}^*\avltlm) \conn i_{w'*}i_{w'}^*\avltlm$ is of perverse degree $2\ell_{\sL} (w)+2\ell_{\sL}(w')$. Therefore the lowest perverse cohomology of $(\ltlm \conn \ulalim_{v} i_{\le v*}i_{\le v}^*\avltlm) \conn \ulalim_{v'} i_{\le v'*}i_{\le v'}^*\avltlm$ is $\ltlm$ in degree zero. Therefore, the diagram must commute up to a scalar. Taking the stalk at $e$ confirms that the diagram commutes.
\end{proof}

%For each pair $(a,b) \in \extw \times \extw$, $\alpha$ restricts to a map from $\ltlm$ to $i_{\le a}^*\avltlm$. Hence we get $\alpha \conn id: \pha \ltlm \conn i_{\le b}^*\avltlm \ra i_{\le a}^*\avltlm \conp i_{\le b}^*\avltlm$. Now choose $a' \in M(a,b)$. $\alpha \conn id$ can be thought as a map from $i_{\le a'}^*(\avltlm \conn i_{\le b}^*\avltlm)$ to $i_{\le a}^*\avltlm \conp i_{\le b}^*\avltlm$. One can check that it is well-behaved with respect to $a'$ and $(a,b)$, hence we get a map $\gamma: (\pha \ltlm \conn \avltlm) \conn \avltlm \ra \pha \ltlm \conn \Av_!( \pha \ltlm \conn  \avltlm)$. 

%The construction of the inverse of $\gamma$ is similar. Fix $(a,b) \in \extw \times \extw$ and choose $a' \in M(a, b^{-1})$. As in the previous paragraph, we have a map from $i_{\le a}^*(i_{\le a'}^* \avltlm \conp i_{\le b}^*\avltlm)$ to $i_{\le a}^*(\avltlm \conp i_{\le b}^*\avltlm)$. It is an isomorphism and its inverse is well-behaved with respect to $a'$ and $(a,b)$. From this, we get a map in the opposite direction. One can check that it is the inverse of $\gamma$.

%The proof is completed by noting that $\gamma$ preserves the lowest perverse cohomology, that is the perverse truncation at degree zero.

After examining the properties of $\avltlm$, we discuss the maximal IC sheaves in non-neutral blocks. Let $\sL, \sL' \in \lio$, $\beta \in \pha \lpewul$, $\xi$ be a minimal IC sheaf for $\beta$. Let $\epsilon (\xi)$ be the canonical map $\lptlpm \conn \xi \ra \delta_{\sL'} \conp \xi= \xi$. By abuse of notation, the same symbol $\epsilon (\xi)$ will also be used for the canonical map $\avlptlpm \conn \xi \ra \delta_{\sL'} \conp \xi= \xi$.

\begin{lem}\label{inter}
Let $\beta \in \pha \lpewul$ be a block and $\xi$ be a minimal IC sheaf for $\beta$. There exists a unique isomorphism $\tau(\xi): \xi \conp \avltlm \xra{\sim} \avlptlpm \conp \xi$ such that the following diagram commutes. 
\[
\begin{tikzcd}[row sep=large]
\xi \conp \avltlm \arrow[d, "id \conp \epsilon_\sL"] \arrow[rr, "\tau(\xi)"] & & \avlptlpm \conp \xi \arrow[d, "\epsilon(\xi)"]\\
\xi \conp \delta_\sL \arrow[r, "="]& \xi \arrow[r, "="] & \delta_{\sL'}\conp \xi.
\end{tikzcd}\]
\end{lem}

\begin{proof}
It is easy to check that both $\xi \conp \avltlm$ and $\avlptlpm \conp \xi$ are glueable complexes with respect to $\beta$. Hence they are isomorphic by Theorem \ref{avglue}. Since right convolution with $\xi$ is an equivalence, we have $\hom^0(\xi \conp \avltlm, \avlptlpm \conp \xi) =\ql$. The existence and uniqueness of $\tau(\xi)$ follow from that. 
\end{proof}

%Similar to the construction of the convolution product of negative sheaves and positive sheaves, we can consider the Hecke stack of $\buu$ with a modification at infinity. 
	
%Let $G_{\fq[t,t^{-1}]}$ be the base change of $G$ to $\fq[t,t^{-1}]$. 
%Fix  $\sL'' \in \lio$. We define the convolution product 
%\[\star_\infty: D^b_{(T \times T, \sL'' \boxtimes (\sL')^{-1}),m}(\bium \bsl G_{\fq[t,t^{-1}]} / \bium) \times \pha \lpdlm \ra \pha \lppdlm .\]

%The double cosets of $\bium \bsl G_{\fq[t,t^{-1}]} / \bium$ are indexed by $\extw$. Define $\IC(\dot{w}^\beta)_\sL^\infty$ to be the intermediate extension of constant sheaf $\ql \langle \ell(w^\beta) \rangle$ on the stratum corresponds to $w^\beta$ and with a rigidfication of stalk at $\dot{w}^\beta$. Then it is easy to show \[\IC(\dot{w}^\beta)_\sL^\infty \star_\infty \pha \ltlm \simeq \pha \lptlm \simeq \pha \lptlpm \conn \xi\] by analogs of Lemma \ref{equiv}. The isomorphism is canonical if we require the rigidfications of stalk at $\dot{w}^\beta$ coincide. Applying $\Av_!$ on both sides, it remains to check $\xi \conp \avltlm$ and $\Av_!(\IC(\dot{w}^\beta)_\sL^\infty \star_\infty \pha \ltlm)$ are canonically isomorphic. 

%Analogs of Lemma \ref{equiv} imply $\Delta(e)_{\sL'}^\infty \ulastar \IC(\dot{w}^\beta)_\sL^\infty$ is canonically isomorphic to $\xi \conp \Delta(e)_{\sL}^\infty$. We obtain the desired isomorphism by right convolution with $\ltlm$ and Lemma \ref{avass}.

Now let $\sL^0, \sL^1, \sL^2, \sL^3 \in \lio$. For $i=1,2,3$, let $\beta_i \in \pha _{\sL^i} \extw_{\sL^{i-1}}$ and $\xi_i$ be a minimal IC sheaf for $\beta_i$. The subscripts $\sL^j$ and identity map will be suppressed for the next lemma. 

\begin{lem}\label{coass}
\begin{enumerate}[{$(1)$}]
\item The composition $\Theta^-  \star \xi_2 \star \xi_1 \xra{\chi} \pha \Theta^- \star \Av_!\Theta^- \star \xi_2 \star \xi_1 \xra{\tau(\xi_2)} \Theta^- \star \xi_2 \star \Av_!\Theta^- \star \xi_1 \xra{\epsilon(\xi_2) \star \epsilon(\xi_1)} \xi_2 \star \xi_1$ is the same as $\epsilon(\xi_2 \star \xi_1)$.

\item The following diagram commutes. 
\[\begin{tikzcd}
\Theta^- \star \xi_2 \star \xi_1  \arrow [d,"\chi"] \arrow [r,"\alpha"] & \Av_!\Theta^- \star \xi_2 \star \xi_1 \arrow[r,"\tau(\xi_2)"]
& \xi_2 \star \Av_!\Theta^- \star \xi_1 \arrow[d, "= "]\\
\Theta^- \star \Av_!\Theta^-\star \xi_2 \star \xi_1 \arrow[r,"\tau(\xi_2)"] & \Theta^- \star \xi_2 \star \Av_!\Theta^-\star  \xi_1 \arrow[r, "\epsilon "] & \xi_2 \star \Av_!\Theta^- \star \xi_1.
\end{tikzcd}\]

\item The following two compositions are the same. \begin{enumerate}[(i)]
\item $\Theta^- \star \xi_3 \star \xi_2 \star \xi_1 \xra{\chi} \Theta^- \star \Av_!\Theta^- \star \xi_3 \star \xi_2 \star \xi_1 \xra{\tau(\xi_3 \star \xi_2)} \Theta^- \star \xi_3 \star \xi_2 \star \Av_!\Theta^- \star \xi_1$ \[ \xra{\chi} \Theta^- \star \Av_!\Theta^- \star \xi_3 \star \xi_2 \star \Av_!\Theta^- \star \xi_1 \xra{\tau(\xi_3)} \Theta^- \star \xi_3 \star \Av_!\Theta^- \star \xi_2 \star \Av_!\Theta^- \star \xi_1.\]
\item $\Theta^- \star \xi_3 \star \xi_2 \star \xi_1 \xra{\chi} \Theta^- \star \Av_!\Theta^- \star \xi_3 \star \xi_2 \star \xi_1 \xra{\tau(\xi_3)} \Theta^- \star \xi_3 \star \Av_!\Theta^- \star \xi_2 \star \xi_1$
\[ \xra{\Av_!\chi} \Theta^- \star \xi_3 \star \Av_!(\Theta^- \star \Av_! \Theta^-) \star \xi_2 \star \xi_1 \ra \Theta^- \star \xi_3 \star \Av_!\Theta^- \star \Av_! \Theta^- \star \xi_2 \star \xi_1 \]
$\xra{\tau(\xi_2)} \Theta^- \star \xi_3 \star \Av_!\Theta^- \star \xi_2 \star \Av_! \Theta^-\star \xi_1$.
\end{enumerate}
The second last map in $(ii)$ is given by Lemma $\ref{avass}$.
\end{enumerate}

\end{lem}

\begin{proof}
$(1)$ follows from the following commutative diagram, 

\[\begin{tikzcd}
\Theta^- \star \xi_2  \arrow [rd, "="'] \arrow [r,"\chi"] & \Theta^- \star \Av_!\Theta^- \star \xi_2 \arrow[d, "\epsilon "] \arrow[r,"\tau(\xi_2)"] & \Theta^- \star \xi_2 \star \Av_!\Theta^- \arrow[d, "\epsilon "] \\
& \Theta^- \star \delta \star \xi_2  \arrow[r, "="] & \Theta^- \star \xi_2 \star \delta.
\end{tikzcd}\]

For $(2)$, we observe that $\xi_2 \star \Av_!\Theta^- \star \xi_1$ has only scalar endomorphisms. It suffices to check the diagram commutes after applying $\epsilon$, which is clear. 

Since the proof for $(3)$ is very similar to Proposition \ref{triple}, we will omit it. The crucial observation is that both maps are equal to the perverse truncation at degree zero. 
\end{proof}

\section{Affine monodromic Soergel functor}\label{sect7}

We construct the affine monodromic Soergel functors $\M$ and $\eM$ in this section by introducing the extended affine monodromic Hecke category. We identify semisimple complexes (resp. in the extended category) with a subcategory of $S$-modules (resp. $\eR$-bimodules) via $\M$ (resp. $\eM$). See Theorem \ref{fufa} for the precise statement. The proof is in parallel with Section $7$ in \cite{LY}. 

\subsection{Central extension} \label{cent}
\subsubsection{Central extension of loop group} \label{centsub}
We recall the construction of a central extension of the loop group $\gk$ in this subsection. We first consider the loop group of the general linear group $\gk=\glnk$. It is well-known that there is a determinant line bundle $\sL_{det}$ on $\glnk$. Let us describe its $k$-valued points $\sL_{det}(k)$. Points for other rings can be described similarly (cf. \cite[Section 2]{F}). 

Let $\co= k[[t]]$ and $K=k((t))$.  For any $\co$-lattices $\Lambda_1, \Lambda_2$ in the $n$-dimensional $K$-vector space $K^n$, there exists a third lattice $\Lambda_3 \subset \Lambda_1 \cap \Lambda_2$. The relative determinant line of $\Lambda_1$ with respect to $\Lambda_2$ is defined to be $\det(\Lambda_1:\Lambda_2)= \det(\Lambda_1/ \Lambda_3) \otimes_k \det(\Lambda_2/ \Lambda_3)^{\otimes -1}$. Different choices of $\Lambda_3$ give canonically isomorphic lines. Let $\Lambda_0$ be the standard lattice $\co^n \subset K^n$. $\sL_{det}(k)$ consists of pairs $(g, \gamma)$, where $g \in GL_n(k((t)))$ and $\gamma$ is a $k$-linear isomorphism $ k \xra{\sim} \det(g\Lambda_0 : \Lambda_0)$. When $g \in GL_n(k[[t]])$, we have $g\Lambda_0= \Lambda_0$. Hence, there is a canonical section of $\sL_{det}(k)$ over $GL_n(k[[t]])$. 

 %When $G$ is almost simple, $\widetilde{\go}$ splits canonically over $\go$. (See \cite[Lemma 4.2.2]{Z} or \cite[Proposition 1]{FL}).

From the above construction, we see that $\sL_{det}$ descends to $\Gr_{GL_n}$ the affine Grassmannian for $GL_n$ naturally and also that the complement of the zero section in the total space of $\ldet$ is a central $\gm$-extension of $\glnk$. In fact, $\ldet$ has a bi-$\go$-equivariant structure. Therefore, it descends to $\go \bsl \gk / \go$. Denote this central extension as $\eglnk$. Let us describe the multiplication of the $k$-valued points of $\eglnk$. The product of $(g_1,\gamma_1)$ and $(g_2,\gamma_2)$ is given by $(g_1g_2, g_1(\gamma_2)\gamma_1)$, where $g_1(\gamma_2) \gamma_1$ is the composition $k =k \otimes_k k \xra{g_1(\gamma_2) \otimes \gamma_1}\det(g_1g_2\Lambda_0 : g_1\Lambda_0)\otimes \det(g_1 \Lambda_0 : \Lambda_0)=\det(g_1g_2 \Lambda_0 : \Lambda_0)$. 

There are several ways to construct central $\gm$-extensions out of existing ones. 

\begin{con} \label{conloop}
	Let $G_i$ be an algebraic group. Let $\widetilde{\bold{G_{i,K}}}$ be a central $\gm$-extension of the loop group $\bold{G_{i,K}}$. 
	\begin{enumerate}
		\item For any homomorphism $\phi: G \ra G_1$, we can pullback $\widetilde{\bold{G_{1,K}}}$ to obtain a central extension $\egk$ via $\phi$.
		\item For any integer $m$, we can obtain another central extension of $\bold{G_{1,K}}$ by pushing out $\widetilde{\bold{G_{1,K}}}$ via the $m$-th power map $[m]: \gm \rightarrow \gm$.
		\item Let $G=\prod_{i=1}^n G_i$. Then $\prod \widetilde{\bold{G_{i,K}}}$ is a $\gm^n$-central extension of $\gk$. We can pushout $\prod \widetilde{\bold{G_{i,K}}}$ via the multiplication map $\prod: \gm^n \ra \gm$ to obtain a central extension of $\gk$. 
	\end{enumerate}
\end{con}

If $H$ is a subgroup of $\gk$, we denote the central extension of $H$ by $\widetilde{H}$. We denote the central torus by $\gmc$. We say a central extension $\egk$ is \emph{positively trivial} if $\ego$ is a trivial extension of $\go$. It is easy to check that the above constructions preserve positively trivial central extensions and that the central extension $\eglnk$ is positively trivial. 

%\begin{rem}
%	When $G$ is almost simple, it is well-known that every central extension is positively trivial. See [F]???
%\end{rem}	

\subsubsection{Affine Weyl group action and bilinear form}

The extended affine Weyl group $\extw=W \ltimes \X_*(T)$ acts on $\eT$. To describe the action, we select $\widetilde{\dot{w}}$ as a lift of $\dot{w}$ in $\widetilde{\gk}$ for every $w \in \extw$. As $T$ remains invariant under the conjugation of $\dot{w}$, it follows that $\eT$ is likewise invariant under the conjugation of $\widetilde{\dot{w}}$. Thus, the desired action is obtained. 

When $\egk$ is positively trivial, we can write down the action more explicitly by making use of the canonical isomorphism between $\eT$ and $T \times \gmc$. 	Let $K_{can}$ be the canonical generator of $\X_*(\gmc)$. 

\begin{lem}\label{bilinear}	
	For a positively trivial central extension $\egk$, there is a unique bilinear form $\langle \cdot, \cdot \rangle: \X_*(T) \times \X_*(T) \ra \Z$ such that the action of $\X_*(T) \subset \extw$ on $\X_*(\eT)$ is given as follows: $\lambda \in \X_*(T)$ sends $\eta \in \X_*(T) \subset \X_*(\eT)$ to $\lambda(\eta)=\eta+\langle \lambda, \eta \rangle K_{can} \in \X_*(\eT)$. Equivalently, we have $t^\lambda \eta(x) t^{-\lambda}=\eta(x)x^{\langle \lambda, \eta \rangle}$ for any $\lambda, \eta \in \X_*(T)$ and $x \in \gm$.
\end{lem}

\begin{proof}
	Since $t^{\lambda}$ commutes with $\eta(x)$ in $\gk$, their commutator in $\egk$ is inside $\gmc$. In other words, $\lambda$ sends $\eta$ to $\eta+c(\lambda, \eta) K_{can}$ for some constant $c(\lambda, \eta) \in \Z$. The bilinearity of $c(\lambda, \eta)$ follows from the fact that $K_{can}$ is invariant under the action of $\lambda$.
\end{proof}

\begin{lem}\label{finiteweyl}
	For a positively trivial central extension $\egk$ and $w \in W$, $w$ fixes $K_{can}$ and acts in the usual way on $\X_*(T)$. Hence, the bilinear form in Lemma \ref{bilinear} is $W$-invariant. 
\end{lem}

\begin{proof}
	Since $\widetilde{\go}$ is a trivial $\gm$-extension over $\go$, the commutator of $\dot{w}$ and $x \in T$ is computed in the usual way. It is clear that $w$ acts trivially on $\gmc$ and hence it fixes $K_{can}$. The second statement follows from the fact that $\extw$ is isomorphic to the semi-direct product $W \ltimes \X_*(T)$. 
\end{proof}

%\begin{lem}\label{biform}
%	There is a unique $W$-invariant bilinear map $\langle \cdot, \cdot \rangle: \X_*(T) \times \X_*(T) \ra \Z$ such that the action of $\extw$ on $\X_*(\eT)$ is as follows. 
%	\begin{enumerate}
%		%\item $\langle \lambda, K_{can} \rangle=0$ for any $\lambda \in \X_*(T)$
%		\item $w \in W$ fixes $K_{can}$, and acts in the usual way on $\X_*(T)$.
%		\item $\lambda \in \X_*(T)$ sends $\eta \in \X_*(T)$ to $\lambda(\eta)=\eta+\langle \lambda, \eta \rangle K_{can} \in \X_*(\eT).$
%	\end{enumerate}
%\end{lem}
%
%The second condition is equivalent to the following: for any $\lambda, \eta \in \X_*(T)$ and $x \in \gm$, we have $t^\lambda \eta(x) t^{-\lambda}=\eta(x)x^{\langle \lambda, \eta \rangle}$.
%
%\begin{proof}
%
%\end{proof}

\begin{lem}\label{bigln}
	Let $T_0$ be the maximal torus consisting of diagonal matrices in $G=GL_n$. The bilinear form on $\X_*(T_0)$ associated to $\eglnk$ is the standard Euclidean form. 
\end{lem}	

\begin{proof}
	Let $w_1, \ldots, w_n$ be the weights of $V$ with respect to $T_0$, where $V$ is the standard representation of $GL_n$. The standard Euclidean form on $X_*(T_0)$ is given by $\langle \lambda, \eta \rangle_{Euc}:= \sum w_i(\lambda)w_i(\eta)$. 
	To show that $\langle \lambda, \eta \rangle=\langle \lambda, \eta \rangle_{Euc}$, we may assume $\lambda$ to be negative, that is $w_i(\lambda) \leq 0$ for all $i$. In this case, we have $t^{\lambda} \Lambda_0 \supset \Lambda_0$. Moreover, $\det(t^{\lambda}\Lambda_0:\Lambda_0)= \bigotimes_{i=1}^{n} \bigotimes_{j=w_i(\lambda)}^{-1}t^{j}e_i$, where $e_i$ is the $i$-th coordinate vector in $V$. 
	
	Let $(x,1) \in \eT_0= T_0 \times \gmc $. We claim that the commutator $[(t^{\lambda},\gamma_1),(x,1)]$ is $\prod_{i=1}^n w_i(x)^{w_i(\lambda)} \in \gmc$. By definition, we have  $(t^{\lambda},\gamma_1)(x,1)=(t^{\lambda}x,t^\lambda(1)\gamma_1)$ and $(x,1)(t^{\lambda},\gamma_1)=(xt^{\lambda},x(\gamma_1)1)$. Note that $x \Lambda_0=\Lambda_0$. Moreover, both $t^{\lambda_1}(1)$ and $1$ are the identity of $\det(t^{\lambda}\Lambda_0:t^{\lambda}\Lambda_0)$ and $\det(\Lambda_0: \Lambda_0)$ respectively. The claim follows from the fact that the difference between $\gamma_1$ and $x(\gamma_1)$ on $t^je_i$ is precisely $w_i(x)^{-1}$. Therefore, $\lambda$ sends $(x,1) \in \eT_0$ to $(x, \prod_{i=1}^n w_i(x)^{w_i(\lambda)})$. Hence the statement is proved. 
\end{proof}

We study how Construction \ref{conloop} affects the bilinear forms. We choose a maximal torus $T_i$ for each group $G_i$ in the constructions and assume any morphism between algebraic groups preserves the maximal torus. Therefore a morphism $\phi: G_1 \ra G_2$ induces a linear map $ \phi_*: \X_*(T_1) \ra \X_*(T_2)$. Hence any bilinear form $\langle \cdot, \cdot \rangle$ on $\X_*(T_2)$ gives rise to a bilinear form $\phi^* \langle \cdot, \cdot \rangle$ on $\X_*(T_1)$ via pullback, that is $\phi^* \langle v, w \rangle= \langle \phi_*v, \phi_*w \rangle$ for $v,w \in \X_*(T_1)$.

\begin{lem}\label{bibase}
	We keep the notation in Construction \ref{conloop}. Let $\langle \cdot, \cdot \rangle_i$ be the bilinear form on $\X_*(T_i)$ associated to the central extension $\widetilde{\bold{G_{i,K}}}$. Then we have
	\begin{enumerate}
		\item $\langle \cdot, \cdot \rangle= \phi^*\langle \cdot, \cdot \rangle_1$.
		\item $\langle \cdot, \cdot \rangle_{new}= m\langle \cdot, \cdot \rangle_{old}$.
		\item $\langle \cdot, \cdot \rangle= \sum \langle \cdot, \cdot \rangle_i$.
	\end{enumerate}
\end{lem}	

\begin{proof}
	For (1), let $\lambda, \mu \in \X_*(T)$ and $x \in \gm$. 
	We notice that the commutator 
	$[t^\lambda,\eta(x)]=[(t^{\phi(\lambda)}, \phi(\eta(x))]=x^{\langle \phi(\lambda), \phi(\eta)\rangle}$. The other statements are clear.
\end{proof}	

Let $\phi: G \ra GL(V)$ be a representation of $G$. Let $wt(V)$ be the set of weights of $V$. For any weight $\alpha \in wt(V)$, we denote $n_\alpha \in \N_{\ge 1}$ to be the multiplicity of $\alpha$. 
We define $\langle \lambda_1 , \lambda_2 \rangle_V := \sum_{\alpha \in wt(V)} n_\alpha \alpha(\lambda_1)  \alpha(\lambda_2)$ for any $\lambda_1, \lambda_2 \in \X_*(T)$. It is clear that $\langle \cdot , \cdot \rangle_V : \X_*(T) \times \X_*(T) \ra \Z$ is a $W$-invariant bilinear form.  

\begin{cor}
	Let $\egk$ be pullback of $\widetilde{\bold{GL(V)_K}}$ via $\phi : G \ra GL(V)$. Then the bilinear form associated to $\egk$ is given by $\langle \cdot , \cdot \rangle_V$.
\end{cor}	

\begin{proof}
	We know the bilinear form associated to $GL(V)$ by Lemma \ref{bigln}. Apply Lemma \ref{bibase}(1) to $\phi$. 
\end{proof}

For a split reductive group $G$ over $\fq$, we consider the isogeny $G \sur G':=G/Z(G) \times G/[G,G]$. Suppose $G$ has $n$ simple factors $G_i$, that is $G/Z(G)= G_1 \times \cdots  \times G_n$, we choose a faithful representation $\phi_i: G_i \ra GL(V_i)$ and a non-zero integer $m_i$ for each $G_i$. Then we obtain a central extension $\widetilde{\bold{G_{i,K}}}$ for each $i$ by Construction \ref{conloop}(1) and (2). Hence, we obtain a central extension for the loop group of $G/Z(G)$ by Construction \ref{conloop}(3). By applying Construction \ref{conloop}(1) to the projection $G \ra G/Z(G)$, we obtain a central extension $\egk$ of $\gk$. 

The isogeny $G \sur G'$ induces an injection $\iota: \X_*(T) \inj \X_*(T_1)\times \X_*(T_2) \times \cdots \times \X_*(T_n) \times \X_*(G/[G,G])$. By using Lemma \ref{bibase}, we can explicitly describe the bilinear form $\langle \cdot, \cdot \rangle$ associated to $\egk$.

\begin{prop}\label{action}
	Let $\lambda, \eta \in \X_*(T)$. The bilinear form associated to $\egk$ is given by $\sum_{i=1}^{n}m_i \langle \iota (\lambda), \iota (\eta) \rangle_{V_i}$. In particular, if $\lambda$ is a coroot of $G$, there exists a cocharacter $\eta \in \X_*(T)$ such that $\langle \lambda, \eta \rangle$ is non-zero. 
\end{prop}

\subsubsection{Compatible $\gm$-torsor over $\Bun_{0,\infty}$} Let $\omega_{Bun}$ be the canonical bundle of $\Bun_G$. Let $\omega_{0,\infty}$ (resp. $\omega_{0',\infty'}$) be the pullback of $\omega_{Bun}$ along the projection map from $\Bun_{G,0,\infty}$ (resp. $\Bun_{0',\infty'}$) to $\Bun_{G}$. Let $\widetilde{\Bun_{0,\infty}^{\otimes n}}$ be the total space of the $\gm$-torsor associated to the pullback of $\omega_{0,\infty}^{\otimes n}$.

When $G$ is almost simple, we know that $\Pic(\Bun_G)$ is isomorphic to $\Z$ (cf. \cite[Theorem 4.2.1]{Z}). Therefore we can choose a faithful representation $V$, a non-zero $m$, and an integer $n$ such that there is a surjective map from $\widetilde{\gk}$ (depending on $V,m$) to $\widetilde{\Bun_{0,\infty}^{\otimes n}}$ making the following diagram Cartesian. 
\[\begin{tikzcd}
\widetilde{\gk} \arrow[d, "proj"] \arrow[r] & \widetilde{\Bun_{0,\infty}^{\otimes n}}  \arrow[d, "proj"]\\
	\gk \arrow[r]& \Bun_{0,\infty}.
\end{tikzcd}\]

We observe that the commutator subgroup $[\widetilde{\bim},\widetilde{\bim}]$ is a natural lift of $\bium$ in $\widetilde{\gk}$. This can be proven at the level of Lie algebra (see \cite[Chapter 7.2]{K}). Since $\Bun_{0,\infty}$ is isomorphic to $[\bium \bsl \gk / \biu]$, we can identify $\widetilde{\Bun_{0,\infty}^{\otimes n}}$ with $[\bium \bsl \widetilde{\gk} / \biu]$ via the above Cartesian diagram. 

For general $G$, we choose $(V_i, m_i, n_i)$ for each simple factor $G_i$ such that the above diagram commutes. Similar to Construction \ref{conloop}(3), we can construct a central $\gm$-extension of $\Bun_{0,\infty}$ from $\widetilde{\Bun_{G_i,0,\infty}^{\otimes n_i}}$. Hence, we obtain the following lemma. 

\begin{lem} \label{gmext}
There exists a central $\gm$-extension $\widetilde{\gk} \ra \gk$ and a $\gm$-torsor $\widetilde{\Bun_{0,\infty}} \ra \Bun_{0, \infty}$ such that the following diagram is Cartesian. 
\[\begin{tikzcd}
	\widetilde{\gk} \arrow[d, "proj"] \arrow[r] & {[\bium \bsl \widetilde{\gk} / \biu]}  \arrow[d, "proj"] \arrow[r,"\simeq"] & \widetilde{\Bun_{0,\infty}} \arrow[d, "proj"]\\
	\gk \arrow[r]& {[\bium \bsl \gk / \biu]} \arrow[r,"\simeq"] & \Bun_{0,\infty}.
\end{tikzcd}\]
\end{lem}

%Consider the $\gm$-extension $\widetilde{\glvk}$ of $\glvk$ constructed using the determinant line bundle, as described in Section \ref{centsub}. Let $ \phi: G \rightarrow GL(V)$ be a representation. Define $\egk$ as the pullback of $\widetilde{\glvk}$ via $\phi$. We will associate a $W$-invariant $\Z$-valued bilinear form on $\X_*(T)$. 

From now on, we always fix a central $\gm$-extension $\widetilde{\gk} \ra \gk$ and a $\gm$-torsor $\widetilde{\Bun_{0,\infty}} \ra \Bun_{0, \infty}$ as in Lemma \ref{gmext}. There are natural stratifications of $\widetilde{\gk} $ and $\widetilde{\Bun_{0,\infty}}$ induced from $\gk$ and $\Bun_{0, \infty}$ respectively.

\subsection{Soergel functor}
Denote $R=\tH^*_{T_k}(\pt_k, \ql)$ and $\eR=\tH^*_{\widetilde{T}_k}(\pt_k, \ql)=\Sym(\X^*(\eT))_{\ql}$.
Let $\Delta \gmc$ be the diagonal embedding of $\gm^{\cen}$ to $\eT \times \eT$ and $A$ be the quotient torus $\frac{\widetilde{T} \times \widetilde{T}}{\Delta \gmc}$. Let $S$ be $ \tH^*_{A_k}(\pt_k, \ql)$. The ring $S$ is naturally a subring of $\eR \otimes \eR$. These are all polynomial algebras with generators placed at degree $2$. Moreover, they carry Frobenius actions which send each generator $x$ to $qx$. 

The action of $\extw$ on $\eT$ gives an action of $\extw \times \extw$ on $\eT \times \eT$. The latter action induces an action of $\extw \times \extw$ on the torus $A$. Hence there is a left and right $\extw$ action on both $S$ and $\eR \otimes \eR$. For $w$ in $\extw$ and a ring $B$ acted by $\extw$ on the right, we denote the $w$-invariant part of $B$ as $B^w$.

%When $G$ is commutative, the extended affine Weyl group is the cocharacter lattice, which contains no reflections. Therefore, $\ewlo$ is trivial. 

One can define several categories associated to $\eR$ and $S$ in a similar way to \cite{LY}. Denote $S$-gmod to be the category of graded $S$-modules; $(S, \Fr)$-gmod to be the category of graded $S$-modules with a compatible Fr action; $\eR \otimes \eR$-gmod to be the category of graded $\eR \otimes \eR$-modules; $(\eR \otimes \eR, \Fr)$-gmod to be the category of graded $(\eR \otimes \eR, \Fr)$-modules with a compatible Fr action.

There is a functor $\omega$ forgetting the Fr action. We denote $[-]$, $(-)$, and $\langle - \rangle$ as degree shift, weight twist, and their composition, respectively. $\Hom^\bullet(-,-)$ represents the inner Hom. These categories naturally carry monoidal structures, as described in Definition \ref{conv} below.

We denote $\Ind$ and $\Res$ as the adjoint pair of induction and restriction functors between $S$-mod and $\eR \otimes \eR$-mod. The same notations are also used for their graded counterparts and those with Frobenius action.

\begin{lem}\label{indiso}
	Let $p: \A_{\ql}^{n+1} \ra \A_{\ql}^n$ be a surjective linear map between affine spaces. Let $\sF$ and $\sG$ be two coherent sheaves on $\A^n$ such that $p^*(\sF)$ and $p^*(\sG)$ are isomorphic. Then $\sF$ and $\sG$ are isomorphic. In particular, if $M$ and $N$ are two $S$-modules such that $\Ind(M) \cong\Ind(N)$, then $M \cong N$. 
\end{lem}

\begin{proof}
	By selecting coordinates for the affine spaces, we let $M$ and $N$ be the corresponding $\ql[x_1,x_2,\dots, x_n]$-module for $\sF$ and $\sG$ respectively. Let $\phi$ be an isomorphism from $M \otimes_{\ql[x_1,x_2,\dots, x_n]}\ql[x_1,x_2,\dots ,x_n,x_{n+1}]$ to $N \otimes_{\ql[x_1,x_2,\dots ,x_n]}\ql[x_1,x_2,\dots, x_n,x_{n+1}]$. The isomorphism $\phi$ must send the submodule $M \otimes (x_{n+1})$ to $N \otimes (x_{n+1})$. By taking the quotient of submodules, $\phi$ induces an isomorphism from $M$ to $N$. 
\end{proof}

\begin{lem}\label{quotsh}
	We retain the notation used in Lemma \ref{indiso}. If $p^*(\sF)$ is a quotient of the structure sheaf $\mathcal{O}_{\A_{\ql}^{n+1}}$, then $\sF$ itself is a quotient of the structure sheaf $\mathcal{O}_{\A_{\ql}^{n}}$.
\end{lem}

\begin{proof}
	It suffices to show that if $M \otimes_{\ql[x_1,x_2,\dots, x_n]}\ql[x_1,x_2,\dots ,x_n,x_{n+1}]$ is generated by a single element as a $\ql[x_1,x_2,\dots ,x_n,x_{n+1}]$-module, then $M$ is also generated by a single element as a $\ql[x_1,x_2,\dots ,x_n]$-module. Let $\sum m_i \otimes x_{n+1}^i$ be a generator for the induced module. It is straightforward to see that $m_0$ is a generator for the module $M$. 
\end{proof}

\begin{defn}
	For each $w \in \extw$, let $\eR(w)$ be the graded $\eR$-bimodule which is the quotient of $\eR \otimes \eR$ by the ideal generated by $w(a) \otimes 1- 1 \otimes a$ for all $a \in \eR$. Let $\barr(w) \in (S, \Fr)$-gmod such that $\Ind(\barr(w)) \cong \eR(w)$. 
\end{defn}

Lemma \ref{indiso} shows that $\barr(w)$ is well-defined up to isomorphism. Lemma \ref{quotsh} shows that $\barr(w)$ can be regarded as a quotient of $S$. Under this identification, we define the element $1$ in $\barr(w)$ to be the image of the identity element in $S$. Similarly, we define $1 \in \eR(w)$ to be the image of $1$ in $\eR \otimes \eR$. Since the degree zero part of $\eR(w)$ is one-dimensional, the degree zero part of $\barr(w)$ is also one-dimensional. Hence, the choices of $1$ rigidify both $\barr(w)$ and $\eR(w)$. Note that Fr naturally acts on both $\barr(w)$ and $\eR(w)$.

Denote $\X_*(\eT) \otimes_{\Z} \ql$ to be $\elit$. By abuse of notation, let $\Delta_{\alge}^{\cen}$ be the diagonal embedding of $\X_*(\gm^{\cen}) \otimes_{\Z} \ql$ into either $\elit^2$ or $\elit^3$. We identify $S$-modules with quasi-coherent sheaves on $\elit \times \elit / \Delta_{\alge}^{\cen}$. For $1 \le i < j \le 3$, let $p_{ij}: \elit^3/ \Delta_{\alge}^{\cen} \ra \elit^2/ \Delta_{\alge}^{\cen}$ be the natural projections.

\begin{defn}\label{conv} Given two quasi-coherent sheaves $\sF$ and $\sG$ on $\elit \times \elit / \Delta_{\alge}^{\cen}$, define their convolution product $\sF \bullet \sG$ as  $p_{13*}(p_{12}^*\sF \otimes p_{23}^*\sG)$. In a similar fashion, we define the convolution product $\esF \bullet \esG$ for quasi-coherent sheaves $\esF$ and $\esG$ on $\elit^2$.
\end{defn}

It is easy to check that the convolution product commutes with $\Ind$. Let $M_1, M_2$ be two $\eR$-bimodules. The convolution product of $M_1 \bullet M_2$ is the tensor product $M_1 \otimes_{\eR} M_2$ of $M_1$ and $M_2$ with respect to the second $\eR$-action on $M_1$ and the first
$\eR$-action on $M_2$.

%An $S$-module is the same as an $R$-bimodule with an action induced by $z$. From this point of view, the convolution product $M \bullet N$ is $M \otimes_R N$ along with the action $z\otimes 1 + 1 \otimes z$. 

\begin{defn}The extended affine monodromic Hecke category $\lpedl$ is the $2$-limit of $D^b_{(\eT \times \eT, \sL' \boxtimes \sL^{-1}),m}(\biu \bsl \egksw / \biu)$, where $\sL' \boxtimes \sL^{-1}$ is viewed as a character sheaf on $\eT \times \eT $ via pullback of the natural projection $\eT \times \eT \ra T \times T$.
\end{defn}

There are no extra difficulties for defining the extended version of other categories and sheaves mentioned before. A tilde is added whenever necessary to indicate the extended analog, with the exception being the convolution product. Statements in previous sections can be easily generalized to the extended case. As an illustration, the first statement of Lemma \ref{stdstd} becomes 
$ \eDe (\wdo_1)^+_{w_2\sL} \conp \eDe (\wdo_2)^+_{\sL}\cong  \eDe (\wdo_1\wdo_2)^+_{\sL}$ if $\ell(w_1)+\ell(w_2)= \ell(w_1w_2)$.

The stack $[\ebi \bsl \egk / \ebi]$ can be regarded as the quotient of  $[\biu \bsl \egk / \biu]$ by the action of $\eT \times \eT$. Therefore the cohomology ring $\tH^*(\ebi \bsl \egk / \ebi)$ is an $\eR \otimes \eR$-algebra. Similarly, the stack $[\bi \bsl \gk / \bi]$ can be viewed as the quotient of $[\biu \bsl \egk / \biu]$ by the action of $A=\frac{\eT \times \eT}{\Delta\gmc}$. Therefore the cohomology ring $\tH^*(\bi \bsl \gk / \bi)$ is an $S$-algebra.

%There is a natural map $\pi: [\biu \bsl \egk / \biu] \ra [\biu \bsl \gk  / \biu]$ and the convolution product commutes with the pullback $\pi^*$. 

For $\esF, \esG \in \pha \lpedl$, $\Hom^\bullet(\esF, \esG)$ is a module over $\tH^*(\ebi \bsl \egk / \ebi)$. Hence it is a $\eR$-bimodule. Similarly, $\Hom^\bullet(\sF, \sG)$ is a $S$-module for $\sF, \sG \in \pha \lpdl$. The following lemma provides basic examples demonstrating the computation on the hom space. 

\begin{lem} \label{homstd}
	For $w \in \extw$, $\eR(w)$ is canonically isomorphic to $\Hom^\bullet(\widetilde{\Delta(\wdo)^+_\sL},\widetilde{\Delta(\wdo)^+_\sL})$ in $(\eR \otimes \eR, \Fr)$-gmod. Similarly, $\barr(w)$ is canonically isomorphic to $\Hom^\bullet(\Delta(\wdo)^+_\sL,\Delta(\wdo)^+_\sL)$ in $(S, \Fr)$-gmod. The analogous statements also hold for costandard sheaves. 
\end{lem}

\begin{proof}
	In the extended case, $\Hom^\bullet(\widetilde{\Delta(\wdo)^+_\sL},\widetilde{\Delta(\wdo)^+_\sL})$ computes the $\eT \times \eT$-equivariant cohomology of $\wdo \eT$, which is the same as $\eR(w)$. There is a unique isomorphism sending $id \in \Hom^{0}(\widetilde{\Delta(\wdo)^+_\sL},\widetilde{\Delta(\wdo)^+_\sL})$ to $1 \in \eR(w)$. The analogous reasoning applies to the other cases as well.
\end{proof}

\begin{defn}
Let $\beta \in \pha \lpewul$, $\xi$ (resp. $\xiu$) be a (resp. non-mixed) minimal IC sheaf for $\beta$.
\begin{enumerate}[$(1)$]
\item  Define $\M_{\xi}$ the \textbf{mixed Soergel functor} associated with $\xi$ to be the functor from $\lpdl^\beta$ to $(S, \Fr)$-$\gmod$, which sends $\sF$ to \[\M_{\xi}(\sF):= \Hom^\bullet(\lptlpm \conn \xi, \sF).\] 

\item Define $\Mu_{\xiu}$ the \textbf{non-mixed Soergel functor} associated with $\xiu$ to be the functor from $\lpdull^\beta$ to $S$-$\gmod$, which sends $\sFu$ to \[\Mu_{\xiu}(\sFu):= \Hom^\bullet(\lptulpm \conn \xiu, \sFu).\] 
\item When $\xi=\delta_\sL$ (resp. $\underline{\xi}=\underline{\delta}_\sL$), the corresponding Soergel functors are denoted as $\M^\circ$ (resp. $\underline{\M}^\circ$). 
\end{enumerate}
\end{defn}

It is straightforward to generalize the above definition to the extended case. Thanks to Corollary \ref{Avfor}, we may replace $\lptlpm$ with the averaging maximal sheaf $\avlptlpm$ in the above definition.

%\begin{lem} \label{equivhom}
%	Let $X$ be a variety and a torus $T$ acts on $X$ trivially. Let $\pi$ be the natural morphism from $[X/T]$ to $X$. Let $\sF$ and $\sG$ be two complexes in $D^b_m(X)$. There is a canonical isomorphism from $\Hom^\bullet(\sF,\sG) \otimes_{H^*(X)} H^*_T(X)$ to $\Hom^\bullet(\pi^*\sF, \pi^*\sG)$.
%\end{lem}

%\begin{proof}
%	Since $T$ acts on $X$ trivially, we know that the stack $[X/T]$ is canonically isomorphic to $X \times [\pt/T]$. The statement is derived from the compatibility of taking hom spaces of complexes with the product operation.
%\end{proof}

Let $\pi: [\ebi \bsl \egk / \ebi] \ra [\bi \bsl \gk / \bi]$ be the natural map of quotient stacks. It is a $\gmc$-gerbe. We have the following two commutative diagrams,
\[\begin{tikzcd}
	{[\ebi \bsl \egk / \ebi]} \arrow[d]  \arrow [r,"\pi"] & {[\bi \bsl \gk  / \bi]} \arrow[d] \\ {[\eT \bsl \pt / \eT]} \arrow[r] & {[\pt /A]},
\end{tikzcd} \qquad \qquad
\begin{tikzcd}
	\tH^*(\ebi \bsl \egk / \ebi)  & 	\tH^*(\bi \bsl \gk  / \bi) \arrow [l,"\pi^*"'] \\ \eR \otimes \eR \arrow[u]& S \arrow[u] \arrow [l] .
\end{tikzcd}\] 

Section 2.6 in \cite{LY} provides a pair of adjoint functors 
\[\pi^*: \pha \lpdl \rightleftarrows \pha  \lpedl : \pi_*.\]
By the construction of the functors, we know that $\pi^*$ preserves positive standard sheaves, IC sheaves, and costandard sheaves respectively. For instance, the functor $\pi^*$ sends $\delta_\sL$ to $\widetilde{\delta_\sL}$. 

Similarly, let $\pi^-: [\ebim \bsl \egk / \ebi] \ra [\bim \bsl \gk / \bi]$ be the natural map of quotient stacks. We have a pair of adjoint functors $(\pi^{-*},\pi^-_*)$ for the category of negative sheaves and the pullback functor $\pi^{-*}$ preserves negative standard sheaves, IC sheaves, and costandard sheaves respectively. In particular, it sends the maximal IC sheaf $\Theta$ to $\eTh$. 

The two Soergel functors $\M$ and $\eM$ are related by the following proposition.  

\begin{prop} \label{Soers}
Let $\sF \in \pha \lpdl^{\beta}$. There exists a natural $(\eR \otimes \eR, \Fr)$-gmod isomorphism from $\Ind(\M(\sF)) \cong  \M(\sF) \otimes _S (\eR \otimes \eR)$ to $\eM(\pi^*\sF)$.
\end{prop}

\begin{proof}
	From the discussion under Claim \ref{welldef}, we can regard $\Theta$ as a $T \times T$-monodromic sheaf on $\gk/ \biu$. Similarly, we can view $\eTh$ as a $\eT \times \eT$-monodromic sheaf on $\egk/ \biu$. Let $\pi': [\eT \bsl \egk / \ebi] \ra [T \bsl \gk / \bi]$. Then we know that  $\pi'^*\Theta$ is canonically isomorphic to $\eTh$. Hence, there is a morphism from $\Hom^{\bullet}(\Theta, \sF)$ to $\Hom^{\bullet}(\eTh,\pi^*\sF)$. This morphism is 	$\tH^*(T \bsl \gk  / \bi)=\tH^*(\bi \bsl \gk  / \bi)$-linear, where the $\tH^*(\bi \bsl \gk  / \bi)$-module structure of $\Hom^{\bullet}(\eTh,\pi^*\sF)$ is given by the transport of structure via the algebra map $\pi^*$. Hence we obtain a $(\eR \otimes \eR, Fr)$-linear morphism from $\M(\sF) \otimes _S (\eR \otimes \eR)$ to $\eM(\pi^*\sF)$. 
	
	From Lemma \ref{homstd} and Proposition \ref{stalkav}, we know that $\M(\nabla(w))$ is isomorphic to $\barr(w)$ as a graded $S$-module and $\eM(\widetilde{\nabla(w)})$ is isomorphic to $\eR(w)$ as a graded $\eR$-bimodule when $w \in \beta$. We observe that the degree zero parts of both $\M(\nabla(w))$ and $\eM(\widetilde{\nabla(w)})$ are one-dimensional $\ql$-vector spaces. The morphism constructed in the previous paragraph induces an isomorphism between the degree zero parts, both spanned by the restriction map. As a result, the morphism is an isomorphism when $\sF$ is a costandard sheaf. The statement follows from the fact that costandard sheaves are generators of $\lpdl$. 
\end{proof}

Here are some simple computations about Soergel functors. 

\begin{lem}\label{conmic}
There is a unique isomorphism in $(S,\Fr)$-$\gmod$, $\M_{\xi}(\xi) \xra{\sim} \barr(w^\beta)$, which sends $\epsilon(\xi)$ to $1$. The extended case holds similarly.  
\end{lem}

\begin{proof}
The argument in the proof of Lemma 7.3 in \cite{LY} still works.
\end{proof}

\begin{lem}\label{conic}
Let $s \in \extw$ be a simple reflection such that $s \in \ewlo$ and let $\sF \in \pha \lpdl ^\beta$. There is a unique isomorphism in  $(S,\Fr)$-$\gmod$, $\M_{\xi}(\sF) \otimes_{S^s} S \xra{\sim}\M_{\xi}(\sF \conp \IC(s)^\dagger_\sL\langle -1 \rangle)$ such that the composition \[\M_{\xi}(\sF) \inj \M_{\xi}(\sF) \otimes_{S^s} S \xra{\sim} \M_{\xi}(\sF \conp \IC(s)^\dagger_\sL \langle -1 \rangle) \xrightarrow{\Lemma \ref{proj} } \M_{\xi}(\pi_s^*\pi_{s*}\sF) \xra{adj} \M_{\xi}(\sF)\] is the identity. 

Moreover, when taking $\sF=\delta_\sL$, this isomorphism sends $1 \otimes 1$ to $\theta^\dagger_s$ after pre-composing the isomorphism in Lemma \ref{conmic}. The extended case holds similarly. 
\end{lem}

\begin{proof}
	The proofs for all cases are similar to that in Lemma $7.4$ in \cite{LY}. We briefly explain the proof for the non-extended case. The key observation is that $\lptlpm \conn \xi$ can be expressed as $\pi_s^*{\overline{\Theta}}$ for some shifted perverse sheaf $\overline{\Theta} \in \pha \lpdelm$ by Lemma \ref{pbic}. For any complex $\sK \in D^b_m(\bi \bsl \gk / P_s)$, there is a natural map in $(\tH^*(\bi \bsl \gk / \bi), \Fr)$-$\gmod$: \[ \tH^*(\bi \bsl \gk / P_s, \sK) \otimes_{\tH^*(\bi \bsl \gk / P_s)} \tH^*(\bi \bsl \gk / \bi) \ra  \tH^*(\bi \bsl \gk / \bi, \pi_{s}^*\sK).\] 
	This map induces a natural map in $(S , \Fr)$-$\gmod$:
	\[ \tH^*(\bi \bsl \gk / P_s, \sK) \otimes_{S^s} S \ra  \tH^*(\bi \bsl \gk / \bi, \pi_{s}^*\sK).\] One can check that this map is an isomorphism. We get the required isomorphism by applying this isomorphism to $\sK=\RuHom(\overline{\Theta}, \pi_{s*}\sF)$.
\end{proof}

\begin{cor}
	Let $s \in \extw$ be a simple reflection such that $s \in \ewlo$. There is a canonical isomorphism in $(\eR \otimes \eR, \Fr)$-$\gmod$ from 
	$\eM^\circ(\eIC(s)_\sL^{\dagger} \langle -1 \rangle)$ to $\eR \otimes_{\eR^s} \eR$. Hence $\eM_{\xi}(\sF) \bullet \eM^\circ(\eIC(s)_\sL^{\dagger} \langle -1 \rangle)$ is canonically isomorphism to $\eM_{\xi}(\sF \conp \eIC(s)^\dagger_\sL\langle -1 \rangle)$. The analogous statement holds for the non-extended case. There is a canonical isomorphism in $(S, \Fr)$-$\gmod$ from $\M^\circ(\IC(s)_\sL^{\dagger} \langle -1 \rangle )$ to $R \otimes_{\eR^s} \eR= R \otimes_{S^s} S$. 
\end{cor}

\begin{proof}
	For the extended case, the statement is proved by taking $\sF$ and $\xi$ to be the identity $\widetilde{\delta_\sL}$ in Lemma \ref{conic}. The non-extended case can be proved similarly. 
\end{proof}	

\subsection{Monoidal structure}
We construct the monoidal structure of $\M$.

Let $\beta_1 \in \pha \lpewul$ and $\beta_2 \in \pha \lppewulp$. Let $\xi_1, \xi_2$ be two minimal sheaves for $\beta_1, \beta_2$ respectively. Suppose $\sF \in \pha \lpdl, \sG \in \pha \lppdlp$ and we are given maps $g: \pha \lpptlppm \conn \xi_2 \ra \sG[i]$ and  $f: \pha \lptlpm \conn \xi_1 \ra \sF[j]$. We can define $\lambda(g, f)$, the product of $g$ and $f$ in $\Hom^{i+j}(\lpptlppm \conn \xi_2 \conn \xi_1, \sG \conp \sF)$, as follows. $\lambda(g,f)$ is the composition
\[\Theta \star \xi_2 \star \xi_1 \xra{\chi} \pha \Theta \star \Av_! \pha \Theta \star \xi_2 \star \xi_1 \xra{\tau(\xi_2)} \pha \Theta  \star \xi_2 \star \Av_! \pha \Theta \star \xi_1 \xra{g\star f} \sG \star \sF [i+j].\] We have identified $f$  as a homomorphism from $\Av_! \pha \Theta \star \xi_1$ to $\sF[j]$ via the isomorphism in Corollary \ref{Avfor}. We have omitted certain superscripts and subscripts. We would do so henceforth as long as it causes no confusion. 

Taking the direct sum over $i,j \in \Z$ gives a pairing \[(\cdot, \cdot): \M_{\xi_2}(\sG) \times \M_{\xi_1}(\sF) \ra  \M_{\xi_2 \star \xi_1}(\sG \star \sF),\]
which descends to \[\M_{\xi_2}(\sG) \otimes_R \M_{\xi_1}(\sF) \ra  \M_{\xi_2 \star \xi_1}(\sG \star \sF).\] Furthermore, it gives a functor \[c_{\xi_2,\xi_1}(\sG, \sF):\M_{\xi_2}(\sG) \bullet \M_{\xi_1}(\sF) \ra  \M_{\xi_2 \star \xi_1}(\sG \star \sF).\]

Lemma \ref{coass}(3) gives an associativity constraint. The two ways of composing from $\M_{\xi_3}(\sH) \bullet \M_{\xi_2}(\sG) \bullet \M_{\xi_1}(\sF)$ to $\M_{\xi_3 \star \xi_2 \star \xi_1}(\sH \star \sG \star \sF)$ are equal. 
All the constructions above are applicable to the extended case. 

\begin{lem} \label{cmin}
Both $c_{\xi_2, \xi_1}(\sG, \xi_1)$ and $c_{\xi_2, \xi_1}(\xi_2, \sF)$ are isomorphisms in $(S, \Fr)$-$\gmod$. The extended case holds similarly.
\end{lem}

\begin{proof}
Lemma \ref{inter} shows that the composition \[\M_{\xi_2}(\sG) \bullet \barr(w^{\beta_1}) \cong \M_{\xi_2}(\sG) \bullet \M_{\xi_1}(\xi_1) \xra{c_{\xi_2, \xi_1}(\sG, \xi_1)} \M_{\xi_2 \star \xi_1}(\sG \star \xi_1) \xra[\sim]{(\star \xi_1)^{-1}} \M_{\xi_2}(\sG)\] sends $g \otimes 1$ to $g$. Hence $c_{\xi_2, \xi_1}(\sG, \xi_1)$ is an isomorphism. 

It requires slightly more effort to prove that $c_{\xi_2, \xi_1}(\xi_2, \sF)$ is an isomorphism. Lemma \ref{coass} (2) implies the composition
\begin{align*}
	\barr(w^{\beta_2}) \bullet \M_{\xi_1}(\sF) & \xra{=} \M_{\xi_2}(\xi_2) \bullet \M_{\xi_1}(\sF) \\
	&\xra{c_{\xi_2, \xi_1}(\xi_2, \sF)} \M_{\xi_2 \star \xi_1}(\xi_2 \star \sF)  \\
	&\xra{adj} \Hom^\bullet(\Av_! \Theta\star \xi_2 \star \xi_1, \xi_2 \star \sF)\\
	& \xra{\tau(\xi_2)}\Hom^\bullet(\xi_2 \star \Av_! \Theta \star \xi_1, \xi_2 \star \sF)\\
	& \xra{(\xi_2 \star)^{-1}}\Hom^\bullet(\Av_! \Theta\star \xi_1,\sF) \\
	&\xra{\alpha}\M_{\xi_1}(\sF)	
\end{align*}
sends $1 \otimes f$ to $f$, where $adj$ is the isomorphism given in Corollary \ref{Avfor}. 
\end{proof}
%$\barr(w^{\beta_2}) \bullet \M_{\xi_1}(\sF) \xra{=} \M_{\xi_2}(\xi_2) \bullet \M_{\xi_1}(\sF) $

%$\xra{c_{\xi_2, \xi_1}(\xi_2, \sF)} \M_{\xi_2 \star \xi_1}(\xi_2 \star \sF) $

%$\xra{\alpha^{-1}} \Hom^\bullet(\Av_! \Theta\star \xi_2 \star \xi_1, \xi_2 \star \sF)$ 

%$\xra{\tau(\xi_2)}\Hom^\bullet(\xi_2 \star \Av_! \Theta \star \xi_1, \xi_2 \star \sF)$

%$\xra{(\xi_2 \star)^{-1}}\Hom^\bullet(\Av_! \Theta\star \xi_1,\sF)$

%$\xra{\alpha}\M_{\xi_1}(\sF)$ 

%sends $1 \otimes f$ to $f$. 
%\end{proof}

\begin{lem}\label{cic}
If $s \in \extw$ is a simple reflection such that $s \in \ewlo$, then $c_{\xi_2,  \delta}(\sG, \IC(s))$ is an isomorphism in $(S, \Fr)$-$\gmod$. The extended case holds similarly.
\end{lem}

\begin{proof}
We use the same notation as Lemma \ref{conic}. 
Let $\overline{g}: \overline{\Theta} \ra \pi_{s*}\sG$ be a morphism and $g: \Theta \ra \sG$ be its adjoint. Notice that the following two diagrams commute. 

\[\begin{tikzcd}
\overline{\Theta} \arrow[d, "\overline{g}"]  \arrow [r,"adj"] & \pi_{s*}\pi_s^*\overline{\Theta} \arrow[d, "\pi_{s*}\pi_s^* \overline{g} "] \\ \pi_{s*}\sG \arrow[r, "adj"] & \pi_{s*}\pi_s^*\pi_{s*} \sG,
\end{tikzcd} \qquad \qquad
\begin{tikzcd}
\Theta \star \xi_2 \star \IC(s) \langle -1 \rangle \arrow[d, "g"]  \arrow [rr," \text{Lemma } \ref{proj}"] & & \pi_s^*\pi_{s*}(\Theta \star \xi_2) \arrow[d, "\pi_s^*\pi_{s*} g "] \\ \sG \star \IC(s) \langle -1 \rangle \arrow[rr, "\text{Lemma } \ref{proj}"] & & \pi_s^*\pi_{s*} \sG.
\end{tikzcd}\] 

Hence, by diagram chasing, it suffices to check that the composition \[\pi_s^*\Theta= \pha \Theta \star \xi_2 \xra{\chi} \Theta \star \Av_!\Theta \star \xi_2 \xra{\tau(\xi_2)} \Theta \star \xi_2 \star \Av_!\Theta\xra{\theta_s^\dagger} \Theta \star \xi_2 \star \IC(s) \langle -1 \rangle =\pi_s^*\pi_{s*}\pi_s^*\overline{\Theta}\] is equal to the one induced by the adjoint map $\overline{\Theta}\ra \pi_{s*}\pi_s^*\overline{\Theta}$. To see this, we compose it with the natural map $\Theta \star \xi_2 \star \IC(s) \langle -1 \rangle \ra \Theta \star \xi_2 \star \delta=\Theta \star \xi_2$ and reduce to show that $\pi_s^*\overline{\Theta} \ra \pi_s^* \pi_{s*}\pi_s^*\overline{\Theta} \ra \pi_s^*\overline{\Theta}$ is the identity map. This is clear as its stalk at $e$ is the identity and the degree zero part of the endomorphism ring of $\Theta \star \xi_2$ is $\ql$.
\end{proof}

\begin{cor}\label{ssiso}
The map $c_{\xi_2, \xi_1}(\sG, \sF)$ is an isomorphism whenever $\sF$ is a semisimple complex. The extended case holds similarly.
\end{cor}

Note that the statement is asymmetric, unlike its reductive group analog. We do not claim that $c_{\xi_2, \xi_1}(\sG, \sF)$ is an isomorphism when $\sG$ is a semisimple complex. 

\begin{proof}
By the decomposition theorem, $\omega \sF$ can be written as a sum of shifted direct summands of certain $\omega(\IC(s_{i_1})_{\sL_1} \star \IC(s_{i_2})_{\sL_2} \star \cdots \star \IC(s_{i_n})_{\sL_n})$ for some simple reflections $s_{i_j}$ and suitable $\sL_j$. Applying Lemma \ref{cmin} and Lemma \ref{cic} $n$ times will give the result.  
\end{proof}

Now we are going to prove the main result in this section that the Soergel functor is fully faithful on semisimple complexes. The theorem is false if we only consider the $R$-bimodule structure instead of the $S$-module or $\eR$-bimodule structure. More precisely, $m(\sF, \sG)$ fails to be surjective. This indicates the consideration of the central torus is necessary. 

\begin{thm}\label{fufa}
Let $\beta \in \pha \lpewul$ and $\xi$ be a minimal sheaf in the block $\beta$. Let $\sF$ and $\sG$ be two semisimple complexes in $\lpdl^{\beta}$. Then the natural map
\[m(\sF, \sG):\Hom^\bullet(\sF, \sG) \ra \Hom^\bullet_{S\text{-}\gmod}(\M_{\xi}(\sF),\M_{\xi}(\sG)) \] is an isomorphism in $(S, \Fr)$-$\gmod$. The extended case holds similarly. That is, \[\widetilde{m}(\esF, \esG):\Hom^\bullet(\esF, \esG) \ra \Hom^\bullet_{\eR \otimes \eR \text{-}\gmod}(\eM_{\exi}(\esF),\eM_{\exi}(\esG)) \] is an isomorphism in $(\eR \otimes \eR, \Fr)$-$\gmod$ for any extended minimal sheaf $\exi$ and semisimple complexes $\esF, \esG$ in $\lpedl^{\beta}$.
\end{thm}

\begin{proof}
It is clear that $m(\sF, \sG)$ is $\Fr$-equivariant, hence it suffices to prove that the non-mixed map is an isomorphism in $S$-$\gmod$. Therefore we may and will assume $\usF, \usG \in \pha \lpudl^{\beta}$. We induct on the dimension of the support of $\sFu$ and do not fix a particular block, that is $\sL$ and $\sL'$ can change.  

When the support of $\sFu$ is zero-dimensional, $\usF$ is a direct sum of shifted $\ICu(u)$ for $\ell(u)=0$. We assume that $\usF$ is $\underline{\delta}_\sL$ and $\usG$ is $\underline{\IC}(w)_\sL$ as the general case can be proved similarly. Consider the following commutative diagram,
\[\begin{tikzcd}
\Hom^\bullet(\underline{\delta}_\sL, i_{e*}i_e^!\usG) \arrow[d, "a"]  \arrow [rr, "m(  \underline{\delta}_\sL \text{, } i_{e*}i_e^!\usG)"] & &  \Hom^\bullet_{S\text{-}\gmod}(\barr(e), \Mu^\circ(i_{e*}i_e^!\usG)) \arrow[d, "\Mu^\circ(a)"] \\ \Hom^\bullet(\underline{\delta}_\sL, \usG) \arrow [rr, "m(  \underline{\delta}_\sL \text{, }\usG)"] & & \Hom^\bullet_{S \text{-}\gmod}(\barr(e), \Mu^\circ(\usG)),
\end{tikzcd}\]
where $a$ is induced by the adjunction map $i_{e*}i_e^!\usG \ra \usG$. 

It is clear that $a$ is an isomorphism. Since $i_e^!\usG$ is a direct sum of shifted $\underline{\delta}_\sL$, the top horizontal map is an isomorphism. To see $\Mu^\circ(a)$ is an isomorphism, notice that we have showed $\Mu^\circ(\usG)$ has an increasing filtration with associated graded $\Gr^F_n=\oplus_{\ell(w')=n}\Hom^{\bullet}(i_{w'}^*\Av_!\underline{\Theta}, i_{w'}^! \ICu(w)_\sL^+)$ in Lemma \ref{chess}. Apply $\Hom(\barr(e),-)$ to the exact sequence, $0 \ra F_{\le n-1} \ra F_{\le n} \ra \Gr^F_n \ra 0$ for $1 \le n \le \ell(w)$. Note that $\Gr^F_n$ is the direct sum of $\barr(w')$ for $w' \le w$ with length $n$. From the $\extw$-action on $\eR$, we know that $\eR(w')$ is both a left and right free $\eR$-module. Since this action is faithful, there are no non-zero maps from $\eR(e)$ to $\eR(w')$ when $\ell(w') \ge 1$. The analogous statement holds for $\barr(e)$ and $\barr(w')$. In other words, $\Hom(\barr(e), F_{\le n-1}) \ra \Hom(\barr(e),F_{\le n})$ is bijective for all $1 \le n \le \ell(w)$. Therefore $\Mu^\circ(a):\Hom(\barr(e), F_0) \ra \Hom(\barr(e),F_{\le \ell(w)})$ is bijective. This completes the base case. 

Suppose now the dimension of the support of $\usF$ is $n$ and the statement holds for any $\usF$ of dimension of support smaller than $n$. We first prove the case when the support of $\usF$ is contained in the neutral component of the affine flag variety. By the decomposition theorem, $\ICu(w)_\sL^+$ is a direct summand of $\underline{\IC}(\wu)_\sL^+=\ICu(s_{i_1})_{\sL_1}^+ \star \ICu(s_{i_2})_{\sL_2}^+ \star \cdots \star \ICu(s_{i_n})_{\sL_n}^+$ for some reduced word expression $\wu=(s_{i_1},s_{i_2}, \cdots, s_{i_n})$ of $w$ and suitable $\sL_i$. It suffices to show the statement holds for $\ICu(\wu)^+_\sL$. Let $\wu'=(s_{i_1}, s_{i_2}, \cdots, s_{i_{n-1}})$ be a reduced word expression of $w'=s_{i_1}s_{i_2}\cdots s_{i_{n-1}}$ and $s=s_{i_n}$, then $\ICu(\wu)^+_\sL \cong \ICu(\wu')^+_{s\sL} \star \ICu(s)^+_{\sL}$. Consider the following diagram where each solid arrow is well-defined up to a nonzero scalar.  

\[ \hspace*{-0.0cm}\begin{tikzcd}[column sep = small]
\Hom^\bullet(\ICu(\wu')^+_{s\sL} \star \ICu(s)^+_{\sL}, \usG) \arrow[d, "m"]  \arrow [r, "adj"] &  \Hom^\bullet(\ICu(\wu')^+_{s\sL} , \usG \star \ICu(s)^+_{s\sL}) \arrow[d, "m'"] 
\\ \Hom^\bullet(\Mu(\ICu(\wu')^+_{s\sL} \star \ICu(s)^+_{\sL}),\Mu ( \usG)) \arrow[d, "c"] & \Hom^\bullet(\Mu(\ICu(\wu')^+_{s\sL} ),\Mu ( \usG\star \ICu(s)^+_{s\sL})) \arrow[d, "c'"]
\\ \Hom^\bullet(\Mu(\ICu(\wu')^+_{s\sL})  \bullet \Mu(\ICu(s)^+_{\sL}),\Mu ( \usG)) \arrow [r, dotted, "b"] & \Hom^\bullet(\Mu(\ICu(\wu')^+_{s\sL}) ),\Mu ( \usG) \bullet \Mu(\ICu(s)^+_{\sL}),
\end{tikzcd}\]
where $adj$ is an isomorphism given by Lemma \ref{equiv} if $s \notin \ewlo$ or Lemma \ref{proj} if $s \in \ewlo$; $m$ and $m'$ are induced by the Soergel functor; $c$ and $c'$ are given by Corollary \ref{ssiso}. Every solid arrow except for $m$ is an isomorphism. If we can define an isomorphism $b$ so that the diagram commutes up to a non-zero scalar, then $m$ will be also an isomorphism. 

When $s \notin \ewlo$, $ \Mu(\ICu(s)^+_{\sL})$ is isomorphic to $\barr(s)$. Notice that there exists a natural isomorphism $R \cong \barr(e) \ra \barr(s) \bullet \barr(s) \cong R$ as $S$-module. We define the map $b$ to be induced by this isomorphism. To check the commutativity, it suffices to check for the universal object, that is when $\usG= \ICu(\wu)^+_\sL$ and these two compositions send the identity to the same element. It is true because $adj$ is also induced by an isomorphism $\underline{\delta}_{s\sL} \ra \ICu(s)^+_{\sL} \star \ICu(s)^+_{s\sL}$.

When $s \in \ewlo$, we prove the general statement: for $M_1, M_2 \in S$-$\gmod$, there is a bifunctorial isomorphism of $S$-$\gmod$ 
\[\Hom^\bullet_{S \text{-}\gmod}(M_1, M_2 \bullet \Mu(\ICu(s)^+_{\sL})) \cong \Hom^\bullet_{S \text{-}\gmod}(M_1 \bullet \Mu(\ICu(s)^+_{\sL}), M_2).\] It suffices to construct the unit and counit map. %We use the $R$-bimodule with $z$-action description for $S$-module. From this perspective, we have $\Mu(\ICu(s)^+_{\sL}) \cong R \otimes_{R^s} R \langle 1 \rangle$ as an $R$-bimodule, where the element $z$ acts as $0$ when $s$ is a non-affine simple reflection and non-trivially otherwise.  

Let $\alpha_s$ be a degree two element in $\eR$ corresponding to $s$, that is $s(\alpha_s)=-\alpha_s$. Here $\alpha_s$ is unique up to a scalar and $\alpha_s \in R$. The unit map is the grading-preserving $S$-linear map from $\barr(e)$ to $\barr(e) \otimes_{\eR^s} \eR \otimes_{\eR^s} \eR \langle 2 \rangle$, which sends $1$ to $\alpha_s \otimes 1 \otimes 1 + 1 \otimes 1 \otimes \alpha_s$. The counit map is the grading-preserving $S$-linear map $\barr(e) \otimes_{\eR^s} \eR \otimes_{\eR^s} \eR \ra \barr(e)$, which sends $r_1 \otimes 1 \otimes r_2$ to zero and $r_1 \otimes \alpha_s \otimes r_2$ to $r_1r_2$. The map $b$ is obtained by specializing $M_1=\Mu(\ICu(\wu')^+_{s\sL})$ and $M_2=\Mu(\usG)$. To check commutativity, it suffices to check when $\usG= \ICu(\wu)^+_\sL$ and these two compositions send the identity to the same element. This is true because the tensor product and the forgetful functor correspond to $\pi_s^*$ and $ \pi_{s*}$ respectively. 

If the support of $\sFu$ is not in the neutral component of the affine flag variety, we only need to replace $\ICu(w)$ with $\ICu(u) \star \ICu(w)$ for some length-zero element $u$ in the above argument.

The proof for the extended case is similar (and easier), hence omitted. 
\end{proof}

\section{Soergel bimodule}\label{sectSB}
In this section, we recall some basics of Soergel bimodules without proofs. Details can be found in \cite[Section 8]{LY} and \cite{Ha}. The main result is to identify $\M^\circ(\IC(w)^\dagger_\sL)$ and $\eM^\circ(\eIC(w)^\dagger_\sL)$ with certain extended Soergel bimodules. 

Let $\sL \in \lio$. Denote $H_\sL$ to be the endoscopic group of $G$ corresponding to $\sL$. It is a reductive group sharing the same maximal torus with $G$ (cf. \cite[Section 9.1]{LY}). The subscript would be dropped whenever there is no ambiguity.

From this point onward, we fix a representation of $H$ and a $\gm$-extension of $\hk$ as outlined in Section \ref{cent}. It is worth noting that $\elit:=\X_*(\eT)\otimes_{\Z}\ql$ depends on the chosen representations of $G$ and $H$.  Nevertheless, we have the following proposition.

\begin{prop}\label{isoCox}
$(\ewlo, \elit)$ and $(\ewho, \elit)$ are isomorphic as representations of Coxeter groups. 
\end{prop}	

\begin{proof}
	As stated in Section \ref{endodef}, we know that both $\ewlo$ and $\ewho$ are generated by the same set of reflections, namely the set of reflections corresponding to real affine coroots in $\widetilde{\Phi} ^\vee_\sL$. Proposition \ref{action} describes the actions of $\ewlo$ and $\ewho$ on $\X_*(\eT)$ explicitly. The two actions coincide when restricted to the finite Weyl group, as stated in Lemma \ref{finiteweyl}.  
	
	Let $\langle \cdot, \cdot \rangle_G$ and $\langle \cdot, \cdot \rangle_H$ be the bilinear forms associated with the $\gm$-extensions of $\gk$ and $\hk$ respectively, as described in Lemma \ref{bilinear}. Suppose $\lambda_1, \lambda_2$ are two coroots for different simple factors $W_1,W_2$ of $\ewlo \cong \ewho$. We claim that $\langle \lambda_1, \lambda_2 \rangle_G=\langle \lambda_1, \lambda_2 \rangle_H=0$. Hence it suffices to prove the statement for each simple factor for $\ewlo$ and $\ewho$. 
	
	Lemma \ref{finiteweyl} implies that $\langle \lambda_1, \lambda_2 \rangle_G=\langle w_1\lambda_1, w_1\lambda_2 \rangle_G=\langle w_1\lambda_1, \lambda_2 \rangle_G$ for any $w_1 \in W_1$. Therefore $|W_1| \langle \lambda_1, \lambda_2 \rangle_G=\langle \sum_{w_1 \in W_1}w_1\lambda_1, \lambda_2 \rangle_G=0$ as $\sum_{w_1 \in W_1}w_1\lambda_1=0$. Similarly, we have $\langle \lambda_1, \lambda_2 \rangle_H=0$. Hence, the claim is proved.
	
	For any simple Lie algebra, Chevalley's theorem implies that the reflection representation of its Weyl group $W$ is irreducible. Hence, there is a unique up to scalar $W$-invariant non-trivial bilinear form on $\X_*(T)$. Proposition \ref{action} guarantees that both $\langle \cdot, \cdot \rangle_G$ and $\langle \cdot, \cdot \rangle_H$ are non-trivial when restricted to coroot lattice of each simple factor. Therefore both bilinear forms are isomorphic (over $\Q$ and $\ql$) for each simple factor. Hence, the proposition is proved. 
\end{proof}

The standard theory of Soergel bimodules works perfectly for $(\ewho, \elit)$. Let $SB(\ewho)$ be the category of Soergel bimodules for $(\ewho, \elit)$. For any sequence of simple reflections $(s_{i_n}, s_{i_{n-1}}, \cdots, s_{i_1})$ in $\ewho$, we define the Bott-Samelson bimodule $\So(s_{i_n}, s_{i_{n-1}}, \cdots, s_{i_1})_{\ewho}$ as the graded $\eR$-bimodule $\eR \otimes _{\eR^{s_{i_n}}} \eR \otimes _{\eR^{s_{i_{n-1}}}} \cdots \eR \otimes _{\eR^{s_{i_1}}} \eR$. The following proposition is well-known and we will only sketch the proof. Details can be found in \cite[Satz in Einleitung]{Ha}.

\begin{prop}\label{SBendo}
		For each $w \in \extw_H^\circ$, there is a unique (up to isomorphism) indecomposable graded $\eR \otimes \eR$-module $\So(w)_{\ewho}$ such that \begin{enumerate}
		\item $\Supp(\So(w)_{\ewho}) \supset \Gamma(w)$ as a subset of $\Spec (\eR \otimes \eR)$, where $\Gamma(w)=\{(wx,x) : x \in \elit \}$ is the graph of the $w$ action on $\elit$.
		\item For some reduced expression $w=s_{i_n}s_{i_{n-1}}\cdots s_{i_1}$ in $\ewho$, $\So(w)_{\ewho}$ is a direct summand of $\So(s_{i_n}, s_{i_{n-1}}, \cdots, s_{i_1})_{\ewho}$.
	\end{enumerate}
\end{prop}

\begin{proof}
	The $\eR$-bimodule $\So(s_{i_n}, s_{i_{n-1}}, \cdots, s_{i_1})_{\ewho}$ can be regarded as the $\eT \times \eT$-equivariant intersection cohomology of a Bott-Samelson variety in the affine flag variety of $H$. By the decomposition theorem, the $\eT \times \eT$-equivariant intersection cohomology of a Schubert variety in the affine flag variety of $H$ satisfies the conditions (c.f. \cite[Satz in Einleitung]{Ha}). 
	
	Since the global section functor is fully-faithful for semisimple complexes (see Theorem \ref{fufa}), $\So(w)_{\ewho}$ must be the equivariant cohomology of a summand of the Bott-Samelson sheaf on the Schubert variety by condition $(2)$. Condition $(1)$ guarantees the summand has to be the IC sheaf of the Schubert variety. 
\end{proof}	

	Because of Proposition \ref{isoCox}, the theory of Soergel bimodules for $(\ewlo, \elit)$ exists through the transport of structure. Let $SB(\ewlo)$ be the category of Soergel bimodules for $(\ewlo, \elit)$. By utilizing the canonical identification between $\ewlo$ and $\ewho$, we define $\So(x)_{\ewlo}$ as the graded $\eR$-bimodule $\So(x)_{\ewho}$ introduced in Proposition \ref{SBendo} for $x \in \ewlo \cong \ewho$.

Let $\sL \in \lio$ and $w \in \extw$. Let $\beta \in \pha _{w\sL} \ulextw_\sL$ be the block containing $w$. Write $w=xw^\beta$ for $x \in \ewwlo$. Similar to Section 8.3 in \cite{LY}, we define the extended Soergel bimodule \[\So(w)_{\sL}:=\So(x)_{\ewwlo} \otimes_{\eR} \eR(w^\beta).\] 

Since $\So(x)_{\ewwlo}$ is the equivariant intersection cohomology of a Schubert variety in the affine flag variety of $H_{w \sL}$, the degree zero part of $\So(x)_{\ewwlo}$ is one-dimensional. Moreover, the endomorphism ring of $\So(x)_{\ewwlo}$ consists of scalars. 

For each $x \in \ewwlo$, we fix a non-zero element $\textbf{1}$ in the degree zero part of $\So(x)_{\ewwlo}$. Then $\So(w)_{\sL}$ carries a rigidification with the degree zero element $\textbf{1} \otimes 1$.  For any sequence $(s_{i_n},s_{i_{n-1}}, \cdots, s_{i_1})$ of simple reflections in $\extw$ and a length-zero element $u$, let $\sL_j=s_{i_j}\cdots s_{i_1}u\sL$. Define the extended Bott-Samelson bimodule \[\So(s_{i_n},s_{i_{n-1}}, \cdots, s_{i_1};u)_{\sL}:=\So(s_{i_n})_{\sL_{n-1}} \otimes_{\eR}\So(s_{i_{n-1}})_{\sL_{n-2}} \otimes_{\eR} \cdots \otimes_{\eR} \So(s_{i_1})_{\sL_1} \otimes_{\eR} \eR(u).\]
When $\sL$ is trivial and $u=e$, the modules $\So(s_{i_n},s_{i_{n-1}}, \cdots, s_{i_1};e)_{\sL}$ and $\So(w)_{\sL}$ become the usual Bott-Samelson bimodules and Soergel bimodules for the affine Weyl group $\ewlo$. 

There is a criterion for an indecomposable graded $\eR \otimes \eR$-module to be an extended Soergel bimodule, similar to the conditions stated in Proposition \ref{SBendo}. 

\begin{lem}\label{extSB}
Let $\sL \in \lio$ and $w \in \extw$. Let $u$ be the unique length-zero element in $\extw$ such that $y=wu^{-1}$ is in the affine Weyl group. Let $M$ be an indecomposable graded $\eR \otimes \eR$-module such that \begin{enumerate}
\item $\Supp(M) \supset \Gamma(w)$ as a subset of $\Spec (\eR \otimes \eR)$, where $\Gamma(w)=\{(wx,x) : x \in \elit \}$ is the graph of the $w$ action on $\elit$.
\item For some reduced expression $y=s_{i_n}s_{i_{n-1}}\cdots s_{i_1}$ in $\extw$, $M$ is a direct summand of $\So(s_{i_n}, s_{i_{n-1}}, \cdots, s_{i_1};u)_\sL $.
\end{enumerate}
Then $M \cong \So(w)_\sL$. 
\end{lem}

\begin{proof}
The proof is similar to Lemma 8.5 in \cite{LY}. We will only sketch the proof. By replacing $M$ with $M \otimes_{\eR} \eR(u^{-1})$, we may assume that $u=e$. Let $M':= M \otimes_{\eR} \eR(w^{\beta,-1})$ and $v=yw^{\beta,-1}$. The analog of Lemma 8.4 in \cite{LY} implies that there exists a reduced expression $v=t_m t_{m-1} \cdots t_1$ in $\ewwlo$ such that $\So(s_{i_n}, s_{i_{n-1}}, \cdots, s_{i_1};e)_\sL $ is isomorphic to $\So(t_m, t_{m-1}, \cdots, t_1)_{\ewwlo} \otimes_{\eR} \eR(w^\beta)$. By Proposition \ref{SBendo}, we know that $M'$ is isomorphic to $\So(v)_{\ewwlo}$.
\end{proof}

Let $(\eR \otimes \eR, \Fr)$-$\gmod_{\pure}$ be the full subcategory of $(\eR \otimes \eR, \Fr)$-$\gmod$ containing those $M= \bigoplus_{n \in \Z} M^n$ such that $M^n$ is pure of weight $n$ as $\Fr$-module. The forgetful functor $\omega : (\eR \otimes \eR, \Fr)$-$\gmod_{\pure} \ra \eR \otimes \eR$-$\gmod$ has an one-side inverse \[(-)^\natural: \eR \otimes \eR \text{-gmod} \ra (\eR \otimes \eR, \Fr)\text{-gmod}_{\pure}\] defined by declaring $\Fr$ to act on the $n$-th graded piece by $q^{n/2}$. Let $\SB_m(\ewlo) \subset (\eR \otimes \eR, \Fr)\text{-gmod}_{\pure}$ be the full subcategory consisting of those $M$ such that $\omega M \in \SB(\ewlo)$. It inherits a monoidal structure from $(\eR \otimes \eR, \Fr)\text{-gmod}_{\pure}$.

\begin{prop}\label{icso}
Let $\sL \in \lio$ and $w \in \ewlo$. Then there is a unique isomorphism in $( \eR \otimes \eR, \Fr)$-$\gmod$ \[\eM^\circ (\eIC (w)_\sL^\dagger \langle -\ell_{\sL}(w) \rangle) \cong \So(w)^\natural_{\ewlo}\] which sends $\widetilde{\theta}_w^\dagger$ to $\mathbf{1} \in \So(w)^\natural_{\ewlo}$.
\end{prop}

\begin{proof}
For $y \in \extw$, let $\beta \in \pha _{y\sL} \ulextw_\sL$ be the unique block containing $y$. Let $\xiu$ be a minimal IC sheaf for $\beta$ and let $M$ be $\eMu_{\xiu}(\eICu(y)_\sL[-\ell_\beta(y)])$. Theorem \ref{fufa} shows $\End(M)=\ql$, hence $M$ is indecomposable. The support of the last piece of the filtration of $M$ in Lemma \ref{chess} corresponds to $\Gamma(y)=\Supp(\eR(y)) \subset \Supp(M)$. $M$ satisfies condition $(2)$ in Lemma \ref{extSB} by the decomposition theorem combined with Lemma \ref{conmic} and Lemma \ref{conic}. Hence, we can apply Lemma \ref{extSB} to prove that $\eMu_{\xiu}(\eICu(y)_\sL[-\ell_\beta(y)]) \cong \So(y)_\sL$. The rest of the proof is similar to Proposition $8.7$ in \cite{LY}.
\end{proof}

\begin{prop}\label{filt}
Let $M \in \SB_m(\ewlo)$. There exists a finite filtration $0=F_0M \subset F_1M \subset \cdots \subset F_nM=M$ in $\SB_m(\ewlo)$ satisfying \begin{enumerate}[$(1)$]
\item For $1 \le i \le n$, $\Gr^F_iM \cong \So(w_i)^\natural \langle n_i \rangle \otimes V_i$ for some $w_i \in \ewlo$, $n_i \in \Z$ and finite-dimensional $\Fr$-module $V_i$ pure of weight zero. 
\item The filtration $\omega F_\bullet M$ of $\omega M$ splits in $\eR \otimes \eR$-$\gmod$.
\end{enumerate}
\end{prop}

\begin{proof}
The proof is similar to Proposition 8.10 in \cite{LY}. By Theorem \ref{fufa} and Proposition \ref{icso}, the homomorphism space between two Soergel bimodules at certain degrees vanish due to perverse degree consideration. The filtration in the proposition is analogous to the perverse filtration. 
\end{proof}

There are analogous definitions and results for the $S$-module case. Proposition \ref{Soers} and Proposition \ref{icso} guarantee the
existence of $\bSo(w)_{\sL}$, the Soergel bimodules for $S$ and $ \bSo(s_{i_n},s_{i_{n-1}}, \cdots, s_{i_1}; u)_{\sL}$, the Bott-Samelson bimodules for $S$. They are sent to $\So(w)_{\sL}$ and $ \So(s_{i_n},s_{i_{n-1}}, \cdots, s_{i_1}; u)_{\sL}$ under the induction functor $\Ind$. They are unique up to unique isomorphism after rigidification. 
After defining Soergel bimodules for $S$, we can define the Frobenius version $(S, \Fr)$-$\gmod_{\pure}$ and its full subcategory $\bSB_m(\ewlo)$. The above results and proofs can be generalized to $S$-modules by replacing $\otimes_{\eR}$ with $\bullet$, and $\So$ with $\bSo$ respectively. For completeness, let us state the analog of Proposition \ref{icso} for the $S$-module case. 

\begin{prop}\label{icso2}
Let $\sL \in \lio$ and $w \in \ewlo$. Then there is a unique isomorphism in $(S, \Fr)$-$\gmod$ \[\M^\circ (\IC (w)_\sL^\dagger \langle -\ell_{\sL}(w) \rangle) \cong \bSo(w)^\natural_{\ewlo}\] which sends $\theta_w^\dagger$ to $\mathbf{1} \in \bSo(w)^\natural_{\ewlo}$.
\end{prop}

%The support of $\So(w)_{\sL}$ lies inside the hyperplane $\{z_1=z_2\}\subset \Spec (\hatr \otimes \hatr)$, where $z_1,z_2$ are the generators of two copies of $\X^*(\gmr)$ in $\hatr \otimes \hatr$. As a result, $\So(w)_{\sL}$ is naturally a graded $R \otimes \ql[z] \otimes R$-module. On the other hand, any $R \otimes \ql[z] \otimes R $ module $M$ can be thought as a $\hatr \otimes \hatr$ module via the natural isomorphism $\hatr \otimes \hatr / (z_1-z_2) \ra R \otimes \ql[z] \otimes R$. This also works for graded and $\Fr$-equivariant counterparts. By abuse of notation, we use same symbol to denote the resulting module. Notice that $\Hom^\bullet_{R \otimes \ql[z] \otimes R \text{-gmod}} (M_1, M_2)=\Hom^\bullet_{\hatr \otimes \hatr \text{-gmod}} (M_1, M_2)$.

\section{Equivalence for the neutral block}\label{sect9}

In this section we prove the main statement, which is the equivalence between neutral blocks of monodromic Hecke categories. It relies heavily on the machinery developed in \cite[Appendix B]{BY}.

Let $\sL \in \lio$ and $H$ be the endoscopic group of $G$ corresponding to $\sL$. Let $\hk$ be the loop group of $H$ and $\bi_H$ be its standard Iwahori subgroup. 
The usual affine Hecke category for $\hk$ is \[D_{\hk}:=D^b_m(\bi_H \bsl \hk / \bi_H).\] It is the affine monodromic Hecke category for $H$ with the trivial character sheaf on $T$. Hence, all constructions before are applicable to $D_{\hk}$. We denote the IC sheaf, standard sheaf, and costandard sheaf in $D_{\hk}$ by $\IC(w)_\hk, \Delta(w)_\hk$, and $ \nabla(w)_\hk$ respectively. Unlike its reductive group analog, $D_{\hk}$ can usually be decomposed into smaller categories. We define $D_{\hk}^{\circ}$ to be its neutral block, which is a special case in Definition \ref{blockdef}. One can check that $D_{\hk}^{\circ}$ contains the sheaves supported on the neutral connected component of $\bi_H \bsl \hk / \bi_H$. The orbits in this component are those whose index is in $\ewlo$.

As in the previous sections, adding a tilde (resp. an underline) means its extended (resp. non-mixed) counterparts. Here are the main theorems of the paper. 

\begin{thm}\label{mainthm}
Let $\sL \in \Ch(T)$ and $\hk$ be the loop group of the endoscopic group of $G$ associated with $\sL$. There is a natural monoidal equivalence of triangulated categories \[ \Psi_\sL^\circ: D_{\hk}^{\circ} \xra{\sim} \pha \ldl^\circ\] such that \begin{enumerate}[$(1)$]
\item For all $w \in \ewlo$, \[\Psi_\sL^\circ(\IC(w)_\hk) \cong \IC(w)^\dagger_\sL, \quad \Psi_\sL^\circ(\Delta(w)_\hk) \cong \Delta(w)^\dagger_\sL, \quad \Psi_\sL^\circ(\nabla(w)_\hk) \cong \nabla(w)^\dagger_\sL.\] 
In particular, $\Psi_\sL^\circ$ is t-exact for the perverse t-structures.
\item There is a bi-functorial isomorphism of graded $(S, \Fr)$-modules for all $\sF, \sG \in D_{\hk}^{\circ}$ \[\Hom^\bullet(\sF, \sG) \xra{\sim} \Hom^\bullet(\Psi_\sL^\circ(\sF), \Psi_\sL^\circ(\sG)).\]
\end{enumerate}
\end{thm}

\begin{thm}[Extended version of Theorem \ref{mainthm}]\label{mainthm2}
We use the same notation as Theorem \ref{mainthm}. There is a natural monoidal equivalence of triangulated categories \[ \ePsi_\sL^\circ: \eD_{\hk}^{\circ} \xra{\sim} \pha \ledl^\circ\] such that \begin{enumerate}[$(1)$]
\item For all $w \in \ewlo$, \[\ePsi_\sL^\circ(\eIC(w)_\hk) \cong \eIC(w)^\dagger_\sL, \quad \ePsi_\sL^\circ(\eDe(w)_\hk) \cong \eDe(w)^\dagger_\sL, \quad \ePsi_\sL^\circ(\ena(w)_\hk) \cong \ena(w)^\dagger_\sL.\] 
In particular, $\ePsi_\sL^\circ$ is t-exact for the perverse t-structures.
\item There is a bi-functorial isomorphism of graded $(\eR \otimes \eR, \Fr)$-modules for all $\esF, \esG \in \eD_{\hk}^{\circ}$ \[\Hom^\bullet(\esF, \esG) \xra{\sim} \Hom^\bullet(\ePsi_\sL^\circ(\esF), \ePsi_\sL^\circ(\esG)).\]
\end{enumerate}
\end{thm}
Their proofs are very similar. We only discuss the proof for the first theorem, which occupies the rest of the section.

%Unless specified otherwise, $\le$ is understood as the Bruhat order in $\ewlo$ in this section. 

Let us introduce several related categories as in Section 9.3 in \cite{LY}. For $w \in \ewlo$, let $\ldswlo \subset \pha \ldlo$ be the full subcategory consisting of complexes whose support lies on $ \biu \bsl \gksw / \biu$. Define $D(\le w)_{\hk}^{\circ} \subset D_{\hk}^{\circ}$ in a similar fashion. 
We also let $\lcswlo \subset \pha \ldswlo$ (resp. $\lclo \subset \pha \ldlo$) be the full subcategory consisting of objects that are pure of weight zero. From Proposition \ref{parity}, we observe that any object $\sF$ in $\lclo$ is very pure in the sense that $i_w^*\sF$ and $i_w^!\sF$ are pure of weight zero. Let $\lcuswlo$ (resp. $\lculo)$ be the essential image of $\lcswlo$ (resp. $\lclo$) under $\omega: \pha \ldswlo \ra \pha \lduswlo$ (resp. $\omega: \pha \ldlo \ra \pha \ldulo$). Let $K^b(\lclo)$ be the homotopy category of bounded complexes in $\lclo$ and $K^b(\lclo)_0 \subset K^b(\lclo)$ be the thick subcategory consisting of the complexes that are null-homotopic when mapped to $K^b(\lculo)$. Similarly, one can define $K^b(\lcswlo)$ and $K^b(\lcswlo)_0$. 

It is clear that the above categories have extended counterparts, we denote them with a tilde. 

Let $\SB(\ewlo)(\le w)$ be the full subcategory of $\SB(\ewlo)$ consisting of Soergel bimodules supported on the union of $\Gamma(v)$ for $v \le w$. Let $\SB_m(\ewlo)(\le w)$ be its mixed analog, which is the full subcategory of $\SB_m(\ewlo)$ consisting of those $M$ such that $\omega M \in \SB(\ewlo)(\le w)$. Since the graph $\Gamma(v)$ is $\Delta_{\alge}^{\cen}$-invariant, it makes sense to define $\bSB(\ewlo)(\le w)$ and $\bSB_m(\ewlo)(\le w)$ for $S$-module counterparts. 

The formalism in \cite[Appendix B]{BY}  shows that there is a triangulated functor (the realization functor) $\rho_{\le w}: K^b(\lcswlo) \ra \pha \ldswlo$. The arguments in Lemma 9.4 in \cite{LY} imply the following lemma.

\begin{lem}For any $w \in \ewlo$, the functor $\rho_{\le w}$ descends to an equivalence
\[\rho_{\le w}: K^b(\lcswlo)/ K^b(\lcswlo)_0 \xra{\sim} \pha \ldswlo.\] The extended analog holds similarly. 
\end{lem}

From the construction of $\rho_{\le w}$, we know that $\rho_{\le w}$ is compatible with $\rho_{\le w'}$ whenever $w \le w'$.

\begin{prop}\label{puresh}
The restriction of $\M^\circ$ gives a monoidal equivalence \[\varphi_0: \pha \lclo \xra{\sim} \bSB_m(\ewlo)\] such that for $\sF, \sG \in \pha \lclo$, there is a canonical isomorphism in $(S, \Fr)$\text{-}$\gmod$ \[\Hom^\bullet(\sF, \sG) \cong \Hom^\bullet_{S \text{-}\gmod}(\varphi_0(\sF), \varphi_0(\sG)).\] Moreover $\varphi_0$ restricts to an equivalence $\varphi_{0,\le w}: \pha \lcswlo \xra{\sim} \bSB_m(\ewlo)(\le w)$ for any $w\in \ewlo$. The extended case holds similarly. 
\end{prop}

\begin{proof}
Let $\varphi_0: \pha \lclo \ra S$-gmod be the restriction of $\M^\circ$. Corollary \ref{ssiso} shows that $\varphi_0$ is a monoidal functor. 

First, we show that the image of $\varphi_0$ lies in $\bSB_m(\ewlo)$. By Proposition \ref{icso}, we know that $\varphi_0(\omega \sF) \cong \omega \varphi_0(\sF) \in \bSB(\ewlo)$. It remains to show that $\Ext^i(\ltlm, \sF)= \Ext^i(\Av_! \pha\ltlm, \sF)$ is pure of weight $i$ for any $\sF \in \pha \lclo$. By using the "filtration" induced by the Schubert stratification, it suffices to prove that $\Av_! \pha\ltlm$ is $*$-pure of weight zero and $\sF$ is $!$-pure of weight zero (cf.\pha \cite[Lemma 3.1.5]{BY}). The $!$-purity statement of $\sF$ has already been shown in Proposition \ref{parity}. We know that $i_w^*\Av_! \pha\ltlm$ is a one-dimensional local system by Proposition \ref{avstalk}. Since the hyperbolic localization functor preserves the purity of weight (cf.\pha\cite[Theorem 8]{Br}), the stalk at $\wdo$ is pure of weight zero and hence $\Av_! \pha\ltlm$ is $*$-pure of weight zero. Moreover, it is clear that $\varphi_0$ sends $\lcswlo$ to $\bSB_m(\ewlo)(\le w)$.

Note that $\Ext^i(\sF, \sG)$ is pure of weight $i$ because of the $*$-purity of $\sF$ and $!$-purity of $\sG$ by Proposition \ref{parity}. This shows $\hom_{\lclo}(\sF, \sG)=\Hom(\sF, \sG)^{\Fr}=\Hom^\bullet(\sF, \sG)^{\Fr}$. Meanwhile, we have  $\hom_{(S, \Fr)\text{-gmod}}(M,M')=\Hom_{S \text{-gmod}}(M,M')^{\Fr}$ for $M, M'\in \bSB_m(\ewlo)$. Theorem \ref{fufa} states that the map $\varphi_0$ sending $\Hom^\bullet(\sF, \sG)$ to $ \Hom^\bullet_{S \text{-}\gmod}(\varphi_0(\sF), \varphi_0(\sG))$ is an isomorphism. Taking the Frobenius invariant of both sides implies $\varphi_0$ is fully faithful. 

Finally, we show that if $M \in \bSB_m(\ewlo)(\le w)$, then there exists $\sF \in \pha \lcswlo$ such that $\varphi_0(\sF) \cong M$. The argument for the reductive group case works perfectly well here. The idea is to utilize the filtration in Proposition \ref{filt} and to compare the extension classes of sheaves of pure weight zero and the Soergel bimodules by Theorem \ref{fufa}. Notice that the support of an extension between two sheaves are contained in the union of that of the two sheaves. Hence $\sF$ satisfies the support condition.
The extended case can be proved similarly. 
\end{proof}

We introduce some notation before stating the next theorem. For $\sF, \sG \in \pha \ldlo$, let $\Ext^n(\sF, \sG)_m$ be the weight $m$ summand of the $\Fr$-module $\Ext^n(\sF, \sG)$. For $M,M' \in K^b(S$-mod), their morphism space $\HOM_{K^b(S \text{-mod})}(M,M')$ in $K^b(S$-mod) is the space of homotopy classes of $S$-linear chain maps $M \ra M'$. Denote the degree shift of complexes in $K^b(S$-mod) as $\{1\}$. Denote \[\HOM^\bullet_{K^b(S \text{-mod})}(M,M')=\bigoplus_{n \in \Z}\HOM_{K^b(S\text{-mod})}(M,M'\{n\}).\]

If $M,M' \in K^b(S$-gmod), $\HOM_{K^b(S \text{-mod})}(M,M')$ has a natural internal grading, whose $m$-th graded piece is denoted by $\HOM_{K^b(S \text{-mod})}(M,M')_m$. Moreover, when $M,M' \in K^b((S, \Fr)$-gmod), $\HOM_{K^b(S \text{-mod})}(M,M')_m$ carries a Frobenius action naturally. Replacing $S$ by $\eR \otimes \eR$ will give the extended analog.

\begin{thm}\label{gensh} \begin{enumerate}[$(1)$]
\item 

Let $K^b(\bSB_m(\ewlo)(\le w))_0 \subset K^b(\bSB_m(\ewlo)(\le w))$ be the thick subcategory consisting of complexes that become null-homotopic after applying $\omega$. Then $\varphi_{0,\le w}$ induces an equivalence of triangulated categories, \[ K^b(\lcswlo)/K^b(\lcswlo)_0 \xra{K^b(\varphi_{0,\le w})} K^b(\bSB_m(\ewlo)(\le w))/ K^b(\bSB_m(\ewlo)(\le w))_0.\] 
\item 
Let $\varphi_{\sL,\le w}$ be the composition $K^b(\varphi_{0,\le w}) \circ \rho_{\le w}^{-1}$, then for any $\sF, \sG \in \pha \ldswlo$, we have a functorial isomorphism of $(S, \Fr)$-modules \[\Hom^\bullet(\sF, \sG) \xra{\sim}\HOM_{K^b(S \text{-mod})}(\omega \varphi_{\sL,\le w} (\sF), \omega \varphi_{\sL,\le w} (\sG)),\] which sends $\Ext^n(\sF, \sG)_m$ to $\HOM_{K^b(S \text{-mod})}(\omega \varphi_{\sL,\le w} (\sF), \omega \varphi_{\sL,\le w} (\sG)\{n-m\})_m$ for all $m,n \in \Z$.
\end{enumerate}
\end{thm}

\begin{proof}
Both statements can be proved by methods in the reductive group analog (c.f. \cite[Theorem 9.6]{LY}). For $(1)$, it is easy to see that any null-homotopy maps for the objects in $K^b(\bSB_m(\ewlo)(\le w))_0$ can be transported to the Soergel bimodule side via $\varphi_0$ and vice versa. 

For $(2)$, the crucial observation is the spectral sequence induced by stupid filtrations of $\rho_{\le w}^{-1}\sF$ and $\rho_{\le w}^{-1}\sG$, which abuts to $\Ext(\sF, \sG)$, degenerates on the $E_2$ page because differentials are weight preserving. 
\end{proof}

\begin{proof}[Proof of Theorem  \ref{mainthm}]

%$\extw$ acts on $X_*(\elit)$ and $X_*(\elit')$ are isomorphic if $T \sur T'$ finite isogeny, 
%Therefore some scaling for each  

%We denote objects  corresponding to $H$ by adding a prime. Since $\Q z$ is differ by a $\Q$ multiple. We shall check that the Soergel bimodule category for $(\ewlo, \elit)$ and $(\ewlo, \elit')$ are isomorphic. 

We apply Theorem \ref{gensh} to the endoscopic group $H$ with the trivial character sheaf on $T$. We obtain an equivalence \[\varphi_{H,\le w}: D(\le w)_{\hk}^{\circ} \xra{\sim} K^b(\bSB_m(\widetilde{W_H^\circ})(\le w))/ K^b(\bSB_m(\widetilde{W_H^\circ})(\le w))_0, \] where $\widetilde{W_H^\circ}$ is the (non-extended) affine Weyl group associated to $H$, which is the same as $\ewlo$ as a subset in $\extw$. However, $\le$ is the Bruhat order in $\extw$ induced by this canonical isomorphism, rather than the Bruhat order of $\widetilde{W_H^\circ}$. The two Bruhat orders are different as stated in Remark \ref{orderdif}.

For $\sF, \sG \in D(\le w)_{\hk}^{\circ}$, there is a natural isomorphism of $(S, \Fr)$-modules \[\Hom^\bullet(\sF, \sG) \xra{\sim}\HOM_{K^b(S \text{-mod})}(\omega \varphi_{H,\le w} (\sF), \omega \varphi_{H,\le w} (\sG)).\] 
Let $\Psi_{\sL, \le w}^\circ=\varphi_{\sL,\le w}^{-1} \circ \varphi_{H,\le w}$. Thanks to Proposition \ref{isoCox}, we know that the two Soergel bimodule categories for $(\ewlo, \elit)$ and $(\ewho, \elit)$ are isomorphic by definition. Therefore $\Psi_{\sL, \le w}^\circ$ is an equivalence of triangulated categories between $D(\le w)_{\hk}^{\circ}$ and $\ldswl^\circ$. Moreover, $\Psi_{\sL, \le w}^\circ$ is compatible with natural inclusions of triangulated categories when $w$ runs through $\ewlo$ (cf. \cite[Proposition B3.1]{BY}). Hence  the limit of $\Psi_{\sL, \le w}^\circ$ is well-defined, which we denote by $\Psi_{\sL}^\circ$. It is an equivalence between $D_{\hk}^{\circ}$ and $\ldl^\circ$. Furthermore, it is clear that from construction of $\Psi_{\sL}^\circ$, the map in $(2)$ is an isomorphism. 

It is clear from \cite[Remark B3.2]{BY} that $\Psi_{\sL}^\circ$ is a monoidal functor. By Proposition \ref{icso}, $\Psi_\sL^\circ(\IC(w)_\hk) \cong \IC(w)^\dagger_\sL$. It remains to show $\Psi_\sL^\circ(\Delta(w)_\hk) \cong \Delta(w)^\dagger_\sL$ and $\Psi_\sL^\circ(\nabla(w)_\hk) \cong \nabla(w)^\dagger_\sL$. Section 9.8 in \cite{LY} gives a criterion for a sheaf being the standard sheaf (or costandard sheaf). The criterion remains true for the affine case. This is because the information of knowing the homomorphism space from a sheaf to all IC sheaves is enough to determine whether the sheaf is the standard sheaf or not. Hence the proof is complete. 
\end{proof}

\end{document}